\documentclass[11pt,leqno]{amsart}
\usepackage{amsthm,amsfonts,amssymb,amsmath,oldgerm}
\numberwithin{equation}{section}
\usepackage{fullpage}

\renewcommand\d{\partial}

\def\g{\gamma}

\def\l{\lambda}
\def\eps{\varepsilon }

\renewcommand\d{\partial}

\newcommand\R{\mathbb R}

\def\g{\gamma}

\def\eps{\varepsilon}

\def\l{\lambda}
\newcommand\errfn{\textrm{errfn}}
\newcommand\br{\begin{remark}}
\newcommand\er{\end{remark}}
\newcommand\bp{\begin{pmatrix}}
\newcommand\ep{\end{pmatrix}}
\newcommand{\be}{\begin{equation}}
\newcommand{\ee}{\end{equation}}
\newcommand\ba{\begin{equation}\begin{aligned}}
\newcommand\ea{\end{aligned}\end{equation}}

\newcommand{\bap}{\begin{app}}
\newcommand{\eap}{\end{app}}
\newcommand{\begs}{\begin{exams}}
\newcommand{\eegs}{\end{exams}}
\newcommand{\beg}{\begin{example}}
\newcommand{\eeg}{\end{exaplem}}
\newcommand{\bpr}{\begin{proposition}}
\newcommand{\epr}{\end{proposition}}
\newcommand{\bt}{\begin{theorem}}
\newcommand{\et}{\end{theorem}}
\newcommand{\bc}{\begin{corollary}}
\newcommand{\ec}{\end{corollary}}
\newcommand{\bl}{\begin{lemma}}
\newcommand{\el}{\end{lemma}}
\newcommand{\bd}{\begin{definition}}
\newcommand{\ed}{\end{definition}}
\newcommand{\brs}{\begin{remarks}}
\newcommand{\ers}{\end{remarks}}

\newcommand{\B }{\mathcal{B}}

\newcommand{\CalF}{\mathcal{F}}

\newcommand{\CalN}{\mathcal{N}}
\newcommand{\N}{\mathcal{N}}

\newcommand{\RR}{{\mathbb R}}

\newcommand{\const}{\text{\rm constant}}
\newcommand{\Id}{{\rm Id }}

\newcommand{\diag}{{\rm diag }}

\newcommand{\Span}{{\rm Span }}

\newtheorem{theorem}{Theorem}[section]
\newtheorem{proposition}[theorem]{Proposition}
\newtheorem{corollary}[theorem]{Corollary}
\newtheorem{lemma}[theorem]{Lemma}

\theoremstyle{remark}
\newtheorem{remark}[theorem]{Remark}
\theoremstyle{definition}
\newtheorem{definition}[theorem]{Definition}

\newtheorem{example}[theorem]{Example}

\newcommand{\RM}{\mathbb{R}}
\newcommand{\ZM}{\mathbb{Z}}

\newcommand{\CM}{\mathbb{C}}

\newcommand\cS{{\mathcal S}}
\newcommand\cR{{\mathcal R}}
\newcommand\cQ{{\mathcal Q}}

\newcommand\cO{{\mathcal O}}
\newcommand\cN{{\mathcal N}}
\newcommand\cI{{\mathcal I}}

\newcommand\cM{{\mathcal M}}

\newcommand{\f}{\frac}

\newcommand{\beq}{\begin{equation}}
\newcommand{\eeq}{\end{equation}}
\newcommand{\ks}{\bar k}
\newcommand{\Ms}{\bar M}
\newcommand{\cs}{\bar c}
\newcommand{\qs}{\bar q}

\title{
Behavior of periodic solutions of viscous conservation
laws under Localized and nonlocalized perturbations
}

\author{ Mathew A. Johnson}
\address{University of Kansas, Lawrence, KS 66045}
\email{matjohn@math.ku.edu}
\thanks{{\it Mathew A. Johnson}, University of Kansas, Lawrence, KS 66045, matjohn@math.ku.edu\\
Research of M.J. was partially supported under NSF grant no. DMS-1211183 and by the University of Kansas General Research Fund allocation 2302278.}
\author{Pascal Noble}
\address{Universit\'e Lyon I, Villeurbanne, France}
\email{noble@math.univ-lyon1.fr}
\thanks{{\it Pascal Noble}, Universit\'e Lyon I, Villeurbanne, France, noble@math.univ-lyon1.fr\\
Research of P.N. was partially supported by the French ANR Project no.
ANR-09-JCJC-0103-01}
\author{L.Miguel Rodrigues}
\address{Universit\'e Lyon 1, Villeurbanne, France}
\email{ rodrigues@math.univ-lyon1.fr}
\thanks{{\it L.Miguel Rodrigues}, Universit\'e Lyon 1, Villeurbanne, France, rodrigues@math.univ-lyon1.fr \\
Stay of M.R. in Bloomington was supported by
French ANR project no. ANR-09-JCJC-0103-01}
\author{Kevin Zumbrun}
\address{Indiana University, Bloomington, IN 47405}
\email{kzumbrun@indiana.edu}
\thanks{{\it Kevin Zumbrun}, Indiana University, Bloomington, IN 47405, kzumbrun@indiana.edu\\
Research of K.Z. was partially supported
under NSF grant no. DMS-0300487}
\begin{document}

\begin{abstract}
We establish nonlinear stability and asymptotic behavior of
traveling periodic waves of viscous conservation laws
under localized perturbations or
nonlocalized perturbations asymptotic to constant shifts in phase,
showing that long-time behavior is governed by an associated
second-order formal Whitham modulation system.
A key point is to identify the way in which initial perturbations
translate to initial data for this formal system,
a task accomplished by detailed estimates on the linearized solution operator 
about the background wave.
Notably, our approach gives both a common theoretical treatment and a 
complete classification in terms of ``phase-coupling'' or ``-decoupling''
of general systems of conservation or balance laws, encompassing cases
that had previously been studied separately or not at all.
At the same time, our refined description of solutions gives the new
result of nonlinear asymptotic stability with respect to localized
perturbations in the phase-decoupled case, further distinguishing
behavior in the different cases.
An interesting technical aspect of our analysis
is that for systems of conservation laws
the Whitham modulation description is
of system rather than scalar form,
as a consequence of which renormalization methods 
such as have been used to treat the reaction-diffusion case 
in general do not seem to apply.
\end{abstract}

\date{\today}
\maketitle

\begin{center}
{\it Keywords}: periodic traveling waves; balance and conservation laws; asymptotic stability.
\end{center}

\begin{center}
{\it 2010 MSC}:  35B40, 35B10, 35B35, 35L65.
\end{center}

%\clearpage
\tableofcontents
%\clearpage
%%%%%%%%%%%%%%%%%

\section{Introduction}\label{s:introduction}

One of the triumphs in recent years in the dynamical study of
partial differential equations (PDE) 
has been the development of a rigorous theory of
modulation of periodic traveling waves in optics,
pattern-formation, and other equations, 
both illuminating and expanding on formal predictions
made by WKB-type expansion much earlier on, as for example in \cite{W,FST}.
Among many other results, we mention in particular the resolution 
in \cite{S1,S2,S3} using Bloch transform/renormalization techniques
of the then 30-year open problem of stability of periodic 
reaction-diffusion waves 
with respect to localized perturbations\footnote{
Verifying formal predictions and rigorous spectral descriptions of \cite{Eck}.}
and, under nonlocalized perturbations, the rigorous verification 
in \cite{DSSS} using related techniques
of the associated second order (``diffusive'')
formal WKB expansion in various settings,
in particular in the small-wavelength limit.
Most recently, 
the WKB expansion has been verified 
for solutions of reaction-diffusion equations in the long-time limit,
in \cite{SSSU} by methods related to those of \cite{S1,S2,S3,DSSS} and 
in \cite{JNRZ1,JNRZ2} by rather different techniques originating
from the study of conservation laws \cite{JZ2,JZ3,Z4,Z5}.

From these analyses emerges the clear picture of asymptotic behavior
as dominated by a single critical mode of the linearized equations,
corresponding to translational invariance of the underlying equations,
that is governed approximately by 
the phase equation of the formal WKB approximation:
(the integral of) a scalar convected Burgers equation.
However, there are many physically interesting applications to
which this well-developed theory does not apply.
Specifically, when there exist conserved quantities, whether deriving
from Hamiltonian structure/symmetries of the equations,\footnote{
As for example for the Korteweg--de Vries (KdV) equation \cite{W,Se,JZ4,JZB} 
or Euler-Korteweg system \cite{BNR}.}
or, as in the case of parabolic conservation laws considered
here, simply from divergence form of the equations/conservation of mass,
then there exist additional critical modes, and the formal WKB
prediction becomes that of a more complicated {\it hyperbolic--parabolic
system of conservation laws} rather than the scalar convected Burgers
equation of the reaction-diffusion case.

Perhaps the best-known example of such a model is the Kuramoto--Sivashinsky
equation, for which the formal asymptotic description of behavior via a 
hyperbolic--parabolic system of
conservation laws was pointed out
already in \cite{FST} under the alternative form of a damped scalar 
wave equation (the ``viscoelastic behavior'' of the title).
Further examples 
%include models for 
arise in the modeling of
viscoelasticity with strain-gradient effects, inclined thin-film flow, and
B\'enard--Marangoni or surfactant-driven Marangoni flow; 
see Section \ref{s:examples} and
Appendix \ref{s:gen}.

Despite the physical motivation coming from such examples, until very recently
there was no rigorous analysis of nonlinear stability or behavior in
this (system) case.
Indeed, as discussed in Remark \ref{slowmod_Bloch_rmk}, the renormalization
techniques of the asymptotically scalar
reaction-diffusion case in the presence of multiple 
characteristic speeds (linear group velocities) appear to break down.
Using a technically rather different set of techniques, nonlinear stability 
under localized perturbations
has now been shown for such systems in 
\cite{OZ4}\footnote{
Concerning the more tractable (since faster-decaying) 
three and higher dimensional case.
}
\cite{JZ2,JZ3,JZN,BJNRZ1,BJNRZ2}
in great generality, in particular resolving the 
longstanding open problem of nonlinear stability of spectrally stable
Kuramoto--Sivashinsky waves, dating
back to the numerical confirmation in \cite{FST} of existence
of bands in parameter space of spectrally stable waves.
However, up to now, asymptotic behavior has not been determined 
in this more complicated, system, case even for localized perturbations.

More, as discussed in \cite{JZ2,JZ3}, 
%for localized perturbations,
there 
%CHANGED-DL
%is 
was 
some question in this case precisely what behavior 
%CHANGEDkz (minor wording)
one
%we 
might expect.
Specifically, one-dimensional
nonlinear modulational stability under localized perturbations
of spectrally stable periodic traveling wave solutions of
viscous conservation laws was shown
in \cite{JZ2} and \cite{JZ3} in two different cases,
depending roughly on whether or
not the wave speed is stationary to first order along the manifold of nearby
periodic solutions.
These two analyses were motivated by a common connection
observed by Serre \cite{Se} to an associated formal Whitham averaged system
obtained by WKB approximation.
However, despite this shared
heuristic description, the authors observed
some puzzling asymmetries in the results obtained; see, for example,
the discussion in \cite[Section 1.3]{JZ3} 
on the varying linearized and nonlinear decay rates obtained in these
different cases under localized perturbations.
In particular, the nonlinear decay rate obtained for localized
perturbations in the stationary case
was slower than what might be guessed from the formal Whitham approximation
with zero initial phase modulation; however, it was left as an open problem
whether this intuitive initialization was correct, or whether
localized initial
perturbations could excite the phase mode through nonlinear
interaction in some way.

Here, we sharpen and extend these previous results (\cite{JZ2,JZ3})
in several ways,
in particular allowing more general, nonlocalized, perturbations
and rigorously identifying time-asymptotic behavior as agreeing to
leading order with the solution of the formal Whitham system with
appropriately prescribed initial data. 
Our analysis loosely follows, 
and also greatly extends, the approach of \cite{JNRZ1,JNRZ2} 
in the reaction-diffusion case;
as noted earlier, we do not see a way to apply here the 
more familiar techniques of \cite{S1,S2,S3,DSSS,SSSU}.
In the process, we explain the asymmetries observed in \cite{JZ2,JZ3}
as connected with the different ways that initial data 
align with characteristic modes for the common Whitham system governing large-time asymptotics.

A striking consequence of our results is that spectrally stable waves about which wave
speed depends to first order on wave number alone are  not only boundedly nonlinearly stable, but {\it asymptotically stable} with respect to localized perturbations. 
This resolves a question brought up early on in \cite{OZ1,OZ2,Se}
that was left open in the analysis of \cite{JZ2,JZ3}.
On the other hand, with respect to nonlocalized perturbations, waves in the two different cases behave essentially alike.

To put things another way, we show that the case that the part
of the Whitham system corresponding to phase perturbations
decouples from the rest of the Whitham equations yields decay rates 
exactly corresponding to those of the 
%CHANGED-DL
%(scalar 
(scalar\footnote{See Remark \ref{rd_to_scalar} below.}
Whitham equation) 
reaction-diffusion case, both for localized and nonlocalized perturbations.
Indeed, we find that reaction-diffusion and conservation laws
can be put in a common framework 
$$
u_t+f(u)_x+g(u)=(B(u)u_x)_x, 
\quad u\in \RM^n,
$$
consisting of a continuum of models,
with $f\equiv 0$ corresponding to the reaction-diffusion case and
$g\equiv 0$ to the conservative case, for which a complete
classification of behavior can be obtained.

This analysis puts the conservative theory now on a par with that of 
the reaction-diffusion case, at least as far as time-asymptotic
stability and asymptotic behavior.  It is an interesting open problem to
reproduce in the conservative case a small-wavelength description as
obtained in \cite{DSSS} for the reaction-diffusions case.
See \cite{NR1,NR2} for some preliminary results in this direction.

\subsection{Slow modulation behavior}
We begin by emphasizing some insights gained from the WKB approximation process,
which requires, first, a description of nearby periodic traveling waves. 
For definiteness/clarity of exposition, we restrict to the simplest case of a semilinear second-order parabolic system of conservation laws.
However, our analysis extends with little change to the quasilinear 
$2r$-parabolic or (under appropriate structural conditions as in
\cite{Ka,Z2,Z4}) the symmetrizable hyperbolic-parabolic case; 
see \cite{BJNRZ1,BJNRZ2,BJNRZ3,JZN} for related analyses in these
and more general situations. Examples include periodic solutions of the equations of one-dimensional viscoelasticity with 
strain-gradient effects \cite{OZ1,BLeZ,Y} and of the Kuramoto--Sivashinsky equations \cite{K,S,FST} and Saint-Venant equations \cite{D,BM,N1,N2} modeling inclined thin film flow. We discuss in Appendix \ref{s:gen} the changes needed to handle these interesting physical applications.

Consider a periodic traveling-wave solution of a parabolic or 
``viscous'' system of conservation laws
$$
u_t+ f(u)_x =u_{xx} ,
$$
$u,f$ valued in $\RM^n$, $x,t\in \RM$, or, equivalently, a standing-wave solution $u(x,t)=\bar U(x)$ of
\be\label{cons}
u_t+ \ks(f(u)_x-\cs u_x) =\ks^2 u_{xx},
\ee
where $\cs$ is the speed of the original traveling wave, and 
the wave number $\ks$ is chosen so that
\be\label{per}
\bar U(x+1)=\bar U(x).
\ee

Integrating the traveling-wave equation $\ks \bar U''=f(\bar U)'-\cs\,\bar U'$ obtained by substituting $u(x,t)=\bar U(x)$ in \eqref{cons}, we obtain
\be\label{ode}
\ks \bar U' = f(\bar U)-\cs\,\bar U+\qs,
\ee
where $\qs\in \RM^n$ is a constant of motion.
Setting $\bar U_0:=\bar U(0)$, we have evidently $(2n+2)$ parameters $(\ks,\cs,\bar U_0,\qs)$ determining candidates for
periodic solutions, and $n$ constraints $\bar U(1)=\bar U_0$,
suggesting, in the absence of additional special structure\footnote{
For example, Hamiltonian structure or existence of additional
conserved quantities other than $\bar q$ 
\cite{Se,JZ3,BNR}.}
that the set of nearby periodic solutions form a manifold of dimension $n+2$.
Denoting by $\Ms:=\int_0^1 \bar U(x)dx$ the mean of $\bar U$,
we make the genericity assumptions:
\begin{enumerate}
  \item[\bf{(H1)}] $f\in C^K(\RM^n)$ for some
%TODO: check this again.-KZ
% CHANGED : +1 for higher order estimates MR
$K\geq4$.
  \item[\bf{(H2)}]  Up to translation, the set of $1$-periodic solutions of \eqref{cons} (with $M,k$ replacing $\Ms,\ks$) in the vicinity of $\bar U$, $M=\Ms$, $k=\ks$, forms a smooth $(n+1)$-dimensional manifold
\be\label{manparam}
\left\{\ (U(M,k;\cdot),c(M,k))\ \middle|\ (M,k)\in\Omega\ \right\}=\left\{\ (U^{M,k}(\cdot),c(M,k))\ \middle|\ (M,k)\in\Omega\ \right\}
\ee
where $\Omega$ is some open subspace of $\R^{n+1}$ containing $(\Ms,\ks)$ and the role of $M$ is defined implicitly by
\be\label{implicitM}
 M:= \int_0^1 U^{M,k}(x)dx.
\ee
\end{enumerate}

Then, the formal approximate solution of $u_t+\ks(f(u)_x-\cs u_x)-\ks^2 u_{xx}=0$
obtained by a nonlinear WKB expansion, as derived to varying orders
of accuracy in \cite{Se,OZ3,NR1,NR2}, is
\be\label{wref}
u(x,t)\approx U^{(\cM,\kappa)(x,t)}(\Psi(x,t)),
\ee
where the mean $\cM$ and wave number $\kappa:=\ks\Psi_x$ satisfy the
Whitham equations
\ba\label{whitham}
\cM_t+\ks(F-\cs \cM)_x &=\ks^2(d_{11} \cM_x + d_{12}\kappa_x)_x,\\
\kappa_t+\ks(-\omega-\cs\kappa)_x&=\ks^2(d_{21}\cM_x
+ d_{22}\kappa_x)_x,\\
\ea
an enlarged, $(n+1)\times (n+1)$ system of viscous conservation laws,
where $\omega(\cM,\kappa)=-\kappa c(\cM,\kappa)$ denotes time frequency,
$F(\cM,\kappa):=\int_0^1f(U^{\cM,\kappa}(x))dx$ mean flux,
and $d_{ij}(\cM,\kappa)$ are determined by higher-order corrections
as described in \cite{NR1,NR2}.
For convenience of the reader,
we recall these derivations in Appendix \ref{W_app}.
%in a form tailored for the needs of this paper.
%The phase $\Psi$ may also be recovered through the solution of
The phase $\Psi$ may be recovered through the solution of
\be\label{psieq}
\Psi_t= \omega(\cM,\kappa ) +\cs\,\kappa
+\ks d_{21}(\cM,\kappa)\cM_x +\ks d_{22}(\cM,\kappa)\kappa_x.
\ee

\br\label{genrmk}
\textup{
Assumption (H2), corresponding to {\it evolutionarity} of \eqref{whitham} considered as an equation on the manifold of periodic solutions, is necessary for spectral stability in the sense usually defined;
specifically, as described in Lemma \ref{DHlem}, it is implied by condition (D3) below.
Thus, there is no loss of generality, and considerable gain in clarity, in assuming (H2) from the outset as we do here.
}
\er

%CHANGEDkz, moved to this more appropriate (I think) spot...
%CHANGED-DL
%HERE
\br\label{rd_to_scalar}
\textup{Since our analyis is built to deal with nondegenerate cases, when we 
%CHANGEDkz
%mention 
treat
systems where bulk forces are incorporated in all the equations as in reaction-diffusion or convection-reaction-diffusion systems 
%addedd
in nondivergence form, 
a similar count of dimensions leads us to assume that, up to translation, the set of $1$-periodic solutions forms a smooth $1$-dimensional manifold
$$
\left\{\ (U(k;\cdot),c(k))\ \middle|\ k\in\Omega\ \right\}=\left\{\ (U^{k}(\cdot),c(k))\ \middle|\ k\in\Omega\ \right\}
$$
where $\Omega$ is some open interval,
with no constant of integration, hence no
additional parameters $\mathcal{M}$ involved.
 The analog of \eqref{whitham} is then a scalar equation for the evolution of local wavenumber $\kappa$.}
\er

Three kinds of modulation are involved in \eqref{wref}: 
modulation in phase, wave number and mean.
However, a prominent role is played by modulation in phase. 
Indeed, it is a familiar scenario that stability of patterns 
involve description of the evolution of 
various modulation parameters, and that among them the major role is 
devoted to parameters 
determining spatial positions.
%TODO (Mat?): put references here to multi-pulse, etc.
%Promislow, Sandstede, Kaper, Doelmann, ???
%redundatn-KZ
Note however that whereas for patterns whose variation is essentially 
localized in space, such as fronts, solitons or multi-solitons, kinks, shocks,
etc.,
there are a finite number of parameters to follow,
with their evolution described by a system of ordinary differential 
equations,
here, in the periodic setting, 
there is a
continuous description involving function-valued modulation parameters whose evolution obeys a partial differential system (here \eqref{whitham}), the 
reduction being not from continuous to discrete dynamics
but from dynamics about periodic solutions to dynamics about constants:
an averaging process.
Note also that the special role of phase is already encoded in the formal description \eqref{wref}-\eqref{whitham} 
since the parabolic nature of \eqref{whitham} hints at $(\cM-\Ms,\kappa-\ks)=(\cM-\Ms,\ks\d_x(\Psi-\Id))<<\Psi-\Id$.

%CHANGED-DL
%For 
This explains why, for 
the purposes of the stability analyses of \cite{JZ2,JZ3},
it was sufficient to retain from \eqref{wref}-\eqref{whitham} only 
the coarser approximation
\be\label{wrefpoor}
u(x,t)\sim U^{(\Ms,\ks)}(\Psi(x,t))
\ee
neglecting all but phase modulations.
Yet in doing so one gives up any hope to describe the precise behavior of the phase $\Psi$ appearing in \eqref{wrefpoor} since this would require a full modulation approximation and in particular knowledge of $\cM$. Without this precise description of the phase, decay rates obtained for $\Psi-\Id$ and its derivatives may indeed seem mysterious.

In contrast, let us explain what may be guessed from \eqref{wref}-\eqref{whitham} 
about behavior under localized perturbations, that is when initially $\kappa_0-\ks$ is mean-free. To which extent this simplification will lead to higher order decay rates for $\Psi-\Id$ or $\d_x(\Psi-\Id)$ (directly needed to analyze \eqref{wrefpoor}) is of course related to whether at some order the hyperbolic part of the equation for $\kappa_t$ uncouples from the full system \eqref{whitham}. In particular, if 
to second order $\omega$ is independent of $\cM$ then one recovers for the phase higher decay rates corresponding to simpler systems for which there is no extra parameter $\cM$ and \eqref{whitham} is reduced to a scalar conservation law, typical examples being reaction-diffusion systems treated in \cite{JNRZ1,JNRZ2,SSSU};
in other words decay rates for wave number perturbation are those of a solution of a viscous Burgers equation with mean-free initial datum. 
If, on the other hand,
$\omega$ is independent of $\cM$ only up to linear order, then intermediate decay rates are 
%CHANGED-DL
%obtained, 
obtained,\footnote{
%Actually these 
%Note that these
These
intermediate decay rates require some assumptions about characteristics speeds
% here 
provided by assumption (H3) below; see Remark \ref{no_interaction}.}
slower than those for reaction-diffusion systems but faster than those for the general situation when no uncoupling is present or when the wave undergoes a nonlocalized perturbation. 

The latter observations, to be established rigorously in the following, 
were not only inaccessible to proof but, as discussed in 
\cite[``Discussion and open problems'']{JZ2},
actually undecidable from the point of view of the lower-order description 
\eqref{wrefpoor} of \cite{JZ2,JZ3}. Indeed, translated into the present terminology,
the question posed in \cite{JZ2} which of $\Psi$, $\Psi_x$ is the primary variable
(with respect to true behavior) is essentially 
the question whether a localized initial perturbation
induces nontrivial data for $k$ in \eqref{whitham}, the answer to
which is a key step in our analysis, and a rather technical one.

With this is mind, our strategy will be first to validate the scenario 
\eqref{wref}-\eqref{whitham}, then to derive some consequences from the 
analysis of \eqref{whitham}. 
But we first need to give precise definitions
of terms such as ``spectrally stable'' 
((D1)-(D3) below), ``localized perturbation,'' and ``linearly uncoupled.''

\subsection{Setting and preliminary observations}

Linearizing \eqref{cons} about $\bar U$ yields the periodic coefficient
equation
\be\label{lin}
(\d_t-L)v=0,\qquad Lv:= (\ks^2\partial_x^2 +\ks\cs\partial_x -\ks\partial_x A)v,
\qquad
A(x):=df(\bar U(x)),%.
\ee
where here $L$ is considered as a closed operator acting on $L^2(\RM;\RM^n)$ with 
densely defined domain $H^2(\RM;\RM^n)$.\footnote{Henceforth, in our notation for Lebesgue and Sobolev spaces
we will suppress the definition of the range; in particular, we will write $L^p(\RM)$ for the equivalence class
of $p$-integrable $\RM^n$-valued functions $L^p(\RM;\RM^n)$.} 
Introducing the family of operator-valued symbols
$$
L_\xi:= e^{-i\xi\,\cdot} L e^{i\xi\,\cdot}=
\ks^2(\partial_x +i\xi)^2 +\ks(\partial_x+i\xi)(\cs - A),
\quad \xi\in[-\pi,\pi],
$$
operating on periodic functions on $[0,1]$, determined by the defining relation
\be\label{defrel}
L (e^{i\xi\,\cdot}f)= e^{i\xi\,\cdot} (L_\xi f)
\quad \hbox{\rm  for $f\in H^2_{\rm per}([0,1])$},
\ee
we define following \cite{S1,S2,JZ2,JZ3,JZN,BJNRZ2}
the {\it diffusive spectral stability} conditions:
\begin{enumerate}
  \item[\bf{(D1)}]
$\sigma(L)\subset\{\lambda\ |\ \Re \lambda<0\}\cup\{0\}$.
  \item[\bf{(D2)}]
There exists a $\theta>0$ such that for all $\xi\in[-\pi,\pi]$ we have
$\sigma(L_{\xi})\subset\{\lambda\ |\ \Re \lambda\leq-\theta|\xi|^2\}$.
  \item[\bf{(D3)}]
$\lambda=0$ is an eigenvalue of $L_0$ with generalized eigenspace
$\Sigma_0$ of dimension $n+1$.
\end{enumerate}

\br\label{specremarks}\textup{
As the coefficients of $L$ are 1-periodic, Floquet theory implies that the spectrum of $L$
considered as an operator on $L^2(\RM)$ is purely continuous, and that $\lambda\in\sigma(L)$
if and only if the spectral problem $Lv=\lambda v$ has an $L^\infty(\RM)$ eigenfunction of the form
$v(x;\lambda,\xi)=e^{i\xi x}w(x;\lambda,\xi)$
for some $\xi\in[-\pi,\pi]$ and $w(\cdot;\lambda,\xi)\in L^2_{\rm per}([0,1])$;
that is, that
$$
\sigma_{L^2(\RM)}\left(L\right)=\bigcup_{\xi\in[-\pi,\pi]}\sigma_{L^2_{\rm per}([0,1])}\left(L_\xi\right).
$$
See \cite{G} for more details. 
In particular, since the spectrum of a given operator $L_\xi$ is purely discrete, 
consisting of isolated eigenvalues of finite multiplicity which, furthermore, depend continuously on
$\xi$, this provides a discrete parameterization of the essential spectrum of $L$. 
}

\textup{
Applying standard spectral
perturbation theory \cite{K} to the operators $L_\xi$, we obtain from (D3) that there exists for $\xi$ sufficiently
small an invariant $(n+1)$-dimensional subspace $\Sigma_\xi$ of $L_\xi$ and associated total eigenprojection $\Pi(\xi)$ bifurcating analytically from $\Sigma_0$ and its associated eigenprojection $\Pi_0$, with all other eigenvalues of $L_\xi$ having real part uniformly bounded above by some negative constant\footnote{Mark that the important property we have used through these arguments and obtained from the introduction of Bloch symbols $L_\xi$ is compactness, which plays for this periodic setting the role of the finite dimensionality that one obtains with Fourier symbols associated to constant-coefficient operators.}.
}
\er

\br\label{Drmks}
\textup{
Variations $\partial_M U_{|(\Ms,\ks)}$, $\bar U'$
along the manifold of nearby periodic solutions lie always in $\Sigma_0$
(see proof of Lemma \ref{jordanlem}), accounting for $n+1$ dimensions,
whereas $\partial_k U_{|(\Ms,\ks)}$
usually does not. Thus, (D3) is an assumption of minimal dimension, corresponding also to
the assumption that there are no neutral modes of $L_0$ other than those 
accounted for by modulation along the ``slow manifold'' $U^{(M,k)}(\,\cdot\,+\beta)$ as in \eqref{wref}.
Assumption (D2) may be recognized as an ``asymptotic parabolicity''
assumption encoding time-asymptotic diffusion
comparable to that of a second-order heat equation; it is directly related to the parabolicity of system \eqref{whitham}.
Assumption (D1) encodes that the spectrum corresponding to marginal stability is minimal, thus 
confined to $\{0\}$.
}
\er

The following observation 
hints what may be gained at the linear level 
from ``linear uncoupling'' $\d_M c|_(\Ms,\ks) \ne 0$, 
at the same time relating $\Sigma_0$ explicitly to variations along the manifold of periodic traveling waves nearby $\bar U$.
Note \cite{Se}, that this condition corresponds with decoupling
of the $\kappa$ equation in the first-order 
part of the linearization about $(\Ms,\ks)$ of the 
Whitham system \eqref{whitham}. 

\bl\label{jordanlem}
Assuming (H1)--(H2) and (D3),
$L_0$ has a nontrivial Jordan block at $\lambda=0$
if and only if $\d_M c \ne 0$ at $(\Ms,\ks)$, 
or equivalently $\d_M \omega (\Ms,\ks)\ne 0$ in \eqref{whitham}, 
in which case there is a single Jordan chain of height two
ascending from the genuine right eigenfunction $\bar U'$.
In either case,
$$
\Sigma_0=\Span \{\partial_M\bar U,\bar U'\}.
$$
\el

\begin{proof}
Variations $\partial_M \bar U$ comprise an $n$-dimensional subspace
of solutions of
\[
L_0\partial_M \bar U=-\ks(\partial_M c|_{(\Ms,\ks)})\bar U',
\]
i.e.,
either the eigenvalue or generalized eigenvalue equation at $\lambda=0$,
complementary to genuine eigenfunction $\bar U'$
(mean zero, hence independent of $\partial_M\bar U$ by
$\int_0^1 \partial_M \bar U(x) dx =\Id$, a consequence of \eqref{implicitM}).
Comparing dimensions, we thus have $\Sigma_0=\Span \{\partial_M\bar U, \bar U'\}$.
If $\partial_M c|_{(\Ms,\ks)}=0$, then $\Span \{\partial_M\bar U\}$, hence also $\Sigma_0$,
consists entirely of genuine eigenfunctions. If, on the other hand, there is a direction $\nu$ in which $\partial_M c|_{(\Ms,\ks)}\cdot \nu \ne 0$, then there is a single generalized eigendirection
$\Span \{\partial_M \bar U\cdot \nu\}$ over $\bar U'$.
\end{proof}

The following lemma justifies the apparently special assumption (H2).

\bl\label{DHlem}
Assuming (H1), (D3)
%CHANGED-DL
%$\Rightarrow$
implies 
(H2).
\el

\begin{proof}
Observe that variations with respect to $(U_0, q)$
in \eqref{ode} satisfy the eigenvalue ODE for $L_0$
while variations in $c$ satisfy the generalized eigenvalue
equation associated with genuine eigenfunction $\bar U'$,
hence the subspace of all elements of their linear span satisfying the
constraint of periodicity is contained in $\Sigma_0$.
Condition (D3) implies that this subspace is dimension $\le n+1$,
from which we may deduce that the $n$-dimensional
periodicity condition $U(1)=U_0$ is full rank at the values $(\ks, \cs,\bar U_0,\qs)$ corresponding to $\bar U$, as the kernel with respect to the $2n+1$ parameters $(c,U_0,q)$ is dimension $\le n+1$.
This guarantees existence of a smooth parametrization
$U(\alpha,\beta;\cdot)=U^{\alpha}(\cdot+\beta)$, with $(\alpha,\beta)$ lying in some open set of $\RM^n\times\RM$, of the manifold of nearby
periodic solutions, whereupon we may conclude using (D3) again that
$\det\d_\alpha(M,k)_{|(\bar\alpha,\bar\beta)}\neq0$ 
by \cite[Theorem 1.3]{OZ3}.\footnote{
A periodic Evans function computation showing that
the $(n+1)$st derivative of the Evans function at $\lambda=0$
is proportional to $\det\d_\alpha(M,k)$,
hence $\det\d_\alpha(M,k)_{|(\bar\alpha,\bar\beta)}\neq0$
is necessary for (D3) (implicit also in the earlier
work \cite{Se}).}
Thus 
$(\alpha,\beta)\mapsto (M,k,\beta)$ is locally invertible, yielding a smooth parametrization by $(M,k,\beta)$.
\end{proof}

We complete our set of assumptions with a final nondegeneracy condition corresponding to {\it strict hyperbolicity} at $(\Ms,\ks)$ of the first-order part of \eqref{whitham},
namely, the assumption:
\begin{enumerate}
  \item[\bf{(H3)}] the eigenvalues $a_j$ of $\frac{\partial (F-\cs M,-\omega-\cs k)}{\partial (M,k)}|_{(\Ms,\ks)}$ are distinct.
\end{enumerate}

The role of this assumption is made clearer by the following connections established at the linear spectral level in \cite{Se,OZ3} between the Whitham system \eqref{whitham} and long-time ($\sim$ low-frequency for parameters, 
$\sim$ low-Floquet exponent $\xi$ for original functions) behavior.

\bpr[\cite{OZ1,Se}]\label{whiteig}
Assuming (H1)-(H2), the $(n+1)$-multiplicity eigenvalue $\lambda=0$ of $L_\xi$ at $\xi=0$ bifurcates in a differentiable way for $\xi\ne 0$ sufficiently small into $n+1$ eigenvalues
\be\label{as}
\lambda_j(\xi)=-i\ks\xi a_j + o(\xi),
\quad j=1,\dots, n+1,
\ee
where $a_j$ are the eigenvalues of  $\d_{(M,k)} (F-\cs M,-\omega-\cs k)|_{(\Ms,\ks)}$, that is, $\ks a_j$ are the characteristic velocities of the first-order part of the Whitham modulation equations \eqref{whitham} at the values $(\Ms,\ks)$ associated with $\bar U$. Moreover, assuming (H1)--(H3), this bifurcation is analytic.
\epr

Proposition \ref{whiteig} was established in \cite{Se,OZ3} 
using direct Evans function calculations. 
In Section~\ref{s:genprep}, we provide an alternative
proof based on direct spectral perturbation expansion
(as in \cite{NR1,NR2}) which is better suited
to the techniques utilized in our analysis, 
and yields also information about eigenprojections.
In the meantime, we observe the following interesting corollary.

\bc\label{DHcor}
Assuming (H1)-(H2), $\sigma(L)\subset\{\lambda\,|\,\Re\lambda\leq0\}$ implies that characteristics $a_j$ are real.
That is, weak hyperbolicity ($a_j$ real)\footnote{
Full hyperbolicity requiring of course also semisimplicity of 
$a_j$ as eigenvalues of
$\frac{\partial (F-\cs M,-\omega-\cs k)}{\partial (M,k)}|_{(\Ms,\ks)}$
.} of the first-order Whitham equations at $(\Ms,\ks)$ is necessary for spectral stability; in particular both (D1) and (D2) imply this notion of weak hyperbolicity.
\ec

Corollary \ref{DHcor} gives rigorous validation of the Whitham equations as formal predictors of stability.
Indeed, their hyperbolicity is often used as a definition of ``modulational stability.''

%CHANGED-DL
\br\label{H3_regularity}
\textup{
Assumption (H3) provides two kinds of regularity in a simple unified way:
on one hand it gives the analyticity of critical spectral modes of $L$, and on the other hand, when combined with weak hyperbolicity (here following from (D1) or (D2)), it yields strict hyperbolicity of the Whitham's system. We expect that, by usual considerations, it could be replaced with symmetrizability of the Whitham's system and a direct smoothness assumption on spectral expansions.
}
\er

\medskip

We still need to say some words about what we mean by a nonlocalized perturbation. First, a localized perturbation of $\bar U$ is something that may be written as $\bar U+v$ with $v$ localized (and smooth), say $v\in L^1(\RM)\cap H^K(\RM)$. Note that in order for the process of gluing together a left portion of the original wave,
some function on a finite interval, and a right portion of the original wave
to yield a localized perturbation (according to our definition), the 
left-hand and right-hand copies of the original wave should be in phase. 
It is 
%CHANGED-DL
%only 
%
this stringent condition (corresponding to a mean-free condition for the local wave number) that we want to relax in going to nonlocalized perturbations. 

Thus, 
%CHANGED-DL
%we will only 
rather than localized perturbations, we 
consider perturbations of the type $(\bar U+v)\circ\Psi$ with $v$ and $\d_x(\Psi-\Id)$ localized, allowing for changes in phases between 
limiting left and right waves ($\Psi-\Id$ is not localized) but not a change in the 
waves themselves, for instance in its wave number or its mean. In other words, our nonlocalized perturbations will still yield localized data for the Whitham system \eqref{whitham}.

\subsection{Results and implications}
With these preparations, we are ready to state our two main theorems. The first one is an extension of \cite{JZ2,JZ3} to stability under nonlocalized perturbations. The second one provides 
asymptotic behavior by validating the scenario \eqref{wref}-\eqref{whitham}.

Here, and throughout the paper,
given two real valued functions $A$ and $B$, we say that $A\lesssim B$
or that for every $x\in\textrm{dom}(A)\cap\textrm{dom}(B)$, $A(x)\lesssim B(x)$
if there exists a constant $C>0$ such that $A(x)\leq CB(x)$ for each $x\in\textrm{dom}(A)\cap\textrm{dom}(B)$.
Even in a chain of inequalities, we will also feel free to denote by $C$ harmless constants with different values.

%CHANGED-DL: (h_0(-infty)+h_0(infty))/2
\begin{theorem}[Stability]\label{oldmain}
Let $K\geq 3$. Assuming (H1)--(H3) and (D1)-(D3), let
$$
E_0:=\|\tilde u_0(\cdot-h_0(\cdot))-\bar U(\cdot)\|_{L^1(\RM)\cap H^K(\RM)}
+\|\partial_x h_0\|_{L^1(\RM)\cap H^K(\RM)}
$$
be sufficiently small, for some choice of phase shift $h_0$.
Then, there exists a global solution $\tilde u(x,t)$ of \eqref{cons}
with initial data $\tilde u_0$ and a phase function
$\psi(x,t)$ such that $\psi(\cdot,0)=h_0$ and, 
introducing a global phase shift $\psi_\infty=(h_0(-\infty)+h_0(\infty))/2$, 
for $t\geq0$ and $2\le p \le \infty$,
\ba\label{mainest}
\|\tilde u(\cdot-\psi(\cdot,t), t)-\bar U(\ \cdot\ )\|_{L^p(\RM)}
&\lesssim E_0 (1+t)^{-\frac{1}{2}(1-1/p)}\\
\|\nabla_{x,t}\,\psi(t) \|_{L^p(\RM)}
&\lesssim E_0 (1+t)^{-\frac{1}{2}(1-1/p)},\\
\ea
and
\ba\label{andpsi}
%\|\tilde u(t)-\bar U\|_{L^\infty(\RM)}, \quad
%\|\psi(t)\|_{L^\infty(\RM)} &\lesssim E_0;
\|\tilde u(t)-\bar U(\ \cdot\ -\psi_\infty)\|_{L^\infty(\RM)}, \quad
\|\psi(t)-\psi_\infty\|_{L^\infty(\RM)} &\lesssim E_0,\\
\|\tilde u(t)-\bar U\|_{L^\infty(\RM)} &\lesssim E_0+|\psi_\infty\ mod\ 1|;
\ea
in particular
\be\label{backest}
%
%\|\tilde u(\cdot, t)-\bar U(\cdot+\psi(\cdot,t))\|_{L^p(\RM)}
\|\tilde u(\cdot, t)-\bar U(\cdot+\psi(\cdot,t))\|_{L^p(\RM)}
\lesssim E_0 (1+t)^{-\frac{1}{2}(1-1/p)}.
\ee
\end{theorem}

%TODO-MR: reword if needed
%CHANGED-MR
\br\label{orbital_modulational}
\textup{
The above result suggests the introduction of 
``space-modulated distances''\footnote{
%These are not distances.}
These are not true distances, but rather measures 
%of distance
%Strictly speaking, these are not true distances, but rather measures of distance
 associated with a seminorm.}
$$
\delta_X(u,v)\ =\ \inf_\Psi\quad \|u\circ\Psi-v\|_X\ +\ \|\d_x(\Psi-\Id)\|_X.
$$
In these terms, it states 
$\delta_{L^1\cap H^K}-\delta_{L^2\cap L^\infty}$ asymptotic 
stability,\footnote{The proof gives also a $\delta_{L^1\cap H^K}-\delta_{H^K}$ asymptotic stability.} and $\delta_{L^1\cap H^K}-\|\,\cdot\,\|_{L^\infty}$ bounded 
(orbital) stability. 
In the following, 
among other things, 
we discuss situations, involving appropriate
uncoupling conditions, under which one may go from this 
``space-modulated'' asymptotic stability to 
the usual $\|\,\cdot\,\|_{L^1\cap H^K}-\|\,\cdot\,\|_{L^2\cap L^\infty}$ asymptotic stability. 
Note that this 
notion of ``space-modulated'' stability is a natural generalization of the more common one of orbital stability for patterns with localized variations (e.g. fronts, shocks, kinks, solitons, etc.), where the above infimum is taken over uniform translations only.
}
\er
%

%CHANGED-DL: (h_0(-infty)+h_0(infty))/2
\begin{theorem}[Asymptotic behavior]\label{main}
Let $\eta>0$, arbitrary, and $K\geq4$.
Under the assumptions of Theorem \ref{oldmain}, 
%with normalization $h_0(-\infty)=-h_0(\infty)$,\footnote{
%This may be achieved without loss of generality by a shift in $\bar U$;
%see Remark \ref{phase_wavenumber_rmk}.}
and suitable parametrization 
%\eqref{manparam},\footnote{
%Specifically, one satisfying the normalizations imposed 
%in \eqref{norm_kg} to get Lemma~\ref{keylemg} see Remark \ref{phase_wavenumber_rmk}.}
there exist $M(x,t)$, and $\psi(x,t)$ such that $\psi(\cdot,0)=h_0$ 
and, 
with global phase shift $\psi_\infty=(h_0(-\infty)+h_0(\infty))/2$, 
for $t\geq0$, $2\le p\le\infty$,
\ba\label{refinedest}
%CHANGED-DL : one sign
%\|\tilde u(\cdot- \psi(\cdot,t), t)-U^{\Ms+M(\cdot,t),\ks/(1+\psi_x(\cdot,t))}(\cdot)\|_{L^p(\RM)}
\|\tilde u(\cdot- \psi(\cdot,t), t)-U^{\Ms+M(\cdot,t),\ks/(1-\psi_x(\cdot,t))}(\cdot)\|_{L^p(\RM)}
&\lesssim E_0\,\ln(2+t)\, (1+t)^{-\frac{3}{4}},\\
\|(\ks\psi_x,M)(t)\|_{L^p(\RM)}
&\lesssim E_0 (1+t)^{-\frac{1}{2}(1-1/p)},\\
%
%\|\psi(t)\|_{L^\infty(\RM)} &\lesssim E_0.
\|\psi(t)-\psi_\infty\|_{L^\infty(\RM)} &\lesssim E_0.
\ea
Moreover, setting $\Psi(\cdot,t)=(\Id-\psi(\cdot,t))^{-1}$, 
$\kappa=\ks\d_x\Psi$, $\cM(\cdot,t)=(\Ms+M(\cdot,t))\circ\Psi(\cdot,t)$, 
and defining $(\cM_W,\kappa_W)$ and $\Psi_W$ to be solutions of
equations \eqref{whitham}, and \eqref{psieq} with initial data 
\be\label{initial_data}
\begin{array}{rcl}
\cM_W(\cdot,0)&=&
\Ms+\tilde u_0-\bar U\circ\Psi(\cdot,0)+\left(\dfrac{1}{\d_x\Psi(\cdot,0)}-1\right)\left(\bar U\circ\Psi(\cdot,0)-\Ms\right),\\
\kappa_W(\cdot,0)&=&\ks\d_x\Psi(\cdot,0),\\
\Psi_W(\cdot,0)&=&\Psi(\cdot,0),
\end{array}
\ee
we have, for $t\geq0$, $2\le p\le\infty$,
\ba\label{refinedest_comp}
\|(\cM,\kappa)(t)-(\cM_W,\kappa_W)(t) \|_{L^p(\RM)}
&\lesssim E_0 (1+t)^{-\frac{1}{2}(1-1/p)-\frac{1}{2}+\eta},\\
\|\Psi(t) -\Psi_W(t) \|_{L^p(\RM)}
&\lesssim E_0 (1+t)^{-\frac{1}{2}(1-1/p)+\eta};\\
\ea
in particular, $\kappa=\ks\d_x\Psi$, $\kappa_W=\ks\d_x\Psi_W$, and
\be\label{last}
\begin{array}{rcl}
\|\tilde u(\cdot,t)
-U^{\cM(\cdot ,t), \kappa(\cdot,t)}(\Psi(\cdot,t) )\|_{L^p(\RM)}
&\lesssim& E_0\,\ln(2+t)\,  (1+t)^{-\frac{3}{4}},\\
\|\tilde u(\cdot,t)
-U^{\cM_W(\cdot ,t), \kappa_W(\cdot,t)}(\Psi_W(\cdot,t) )\|_{L^p(\RM)}
&\lesssim& E_0 (1+t)^{-\frac{1}{2}(1-1/p)+\eta}.
\end{array}
\ee
\end{theorem}

%CHANGED-DL
\br\label{phase_wavenumber_rmk}
\textup{
%A choice in parametrization is 
A suitable choice in parametrization is made here
to ensure that interdependences on $k$ and $\beta$ in $U^{(M,k)}(\cdot+\beta)$ are compatible with the expected relation between local phase and local wave number, $\kappa=\ks\d_x\Psi$. Explicitly, our normalizing choice is performed in \eqref{norm_kg} (which involves $\tilde q_{n+1}(0)=\bar u^{adj}$ defined in Proposition~\ref{linspecg}) to get Lemma~\ref{keylemg}.
}
\er

\br\label{formal_rmk}
\textup{
Prescription of the initial data \eqref{initial_data}, especially for $\cM_W(\cdot,0)$, is a subtle point\footnote{
This issue does not arise in the related analysis \cite{JNRZ2} of the
reaction-diffusion case, as $M$ does not appear.} not evident from the
viewpoint of formal approximation \eqref{wref}-\eqref{whitham}. In particular, the 
appearance of a term related to phase variations in \eqref{initial_data}(i) arises in our analysis through a detailed study of 
the contribution of high frequencies of the local wave number to variations of the low Floquet 
number part of the solution (see the key equality \eqref{key_initial_data}). Nevertheless, 
%CHANGED-MJ
%at the end, 
in the end,
%ENDCHANGED 
each term involved in \eqref{initial_data}(i) 
has a nice interpretation,
with $\tilde u_0-\bar U\circ\Psi(\cdot,0)$ accounting for the contribution of amplitude variations to the initial perturbation of the mean $\Ms$ and $\left(1/\d_x\Psi(\cdot,0)-1\right)\left(\bar U\circ\Psi(\cdot,0)-\Ms\right)$ encoding the contribution of period variations. 
%CHANGED-DL
To be more specific, on one hand, setting $\tilde d_0=\tilde u_0-\bar U\circ\Psi(\cdot,0)$, we observe that $\tilde d_0$ differs from $x\longmapsto \int_{-1/2}^{1/2} \tilde d_0(x+y)\,dy$, which is easier to interpret, by a localized zero-mean function, a difference that is asymptotically irrelevant\footnote{This follows from the general theory for parabolic systems of conservation laws, see for instance Proposition \ref{cl_lemma}.} at our level of description. Likewise, assuming on the other hand validity of the approximation $\bar U\circ\Psi(\cdot,0)\sim\bar U+(\Psi(\cdot,0)-\Id_\R)\bar U'$ leads us to consider $x\longmapsto \int_{-1/2}^{1/2} (\Psi(x+y,0)-(x+y))\bar U'(x+y)\,dy$ which is
$$
x\longmapsto-\int_{-\frac12}^{\frac12} (\d_x\Psi(x+y,0)-1)\bar U(x+y)\,dy\ +\ \bar U\big(x-1/2\big)\,\big[\Psi\big(x+1/2,0\big)-\Psi\big(x-1/2,0\big)\big]
$$
where, up to localized zero-mean functions, the first part of the sum is $-(\d_x\Psi(\cdot,0)-1)\bar U$ and the second reduces to $(\d_x\Psi(\cdot,0)-1)\Ms$. Note that the fact that in the end of this latter formal computation we recover the formula of the Theorem only in an approximate way reveals that the first approximation in the argument is invalid. Yet this incorrect approximation possesses a correct 
%CHANGEDkz
%analogous 
analog
(see Section~\ref{s:pert}) leading to the initial data of the Theorem. 
Note also that any small localized perturbation of $(\Ms,\ks)$ may be realized as initial data in \eqref{initial_data} by appropriately choosing $h_0$ and $\tilde u_0$ so that at our level of accuracy the full dynamics of \eqref{whitham} near $(\Ms,\ks)$ are present in \eqref{cons} around $\bar U$.}
\er

\br\label{better_rmk}
\textup{
Bounds \eqref{last} are both of form \eqref{wref}, with $\ks\Psi_x=\kappa$, validating a slow modulation 
picture of behavior. Yet comparison of \eqref{last}(i) with \eqref{last}(ii) reveals that by allowing $(\cM,\kappa)$ to satisfy \eqref{whitham} only in an approximate way we here construct a phase modulation more accurate at least by factor $(1+t)^{-1/4-\eta}$ than that of the formal Whitham construction \eqref{wref}-\eqref{whitham}.\footnote{Here we are using the additional fact (not explicitly stated here) that estimate \eqref{last}(ii) is sharp. 
%\cite{LZ} (not stated here) 
%that estimate \eqref{aeq1} is sharp
On the other hand, we do not expect \eqref{last}(i) to be sharp (see Remark~\ref{nonsharp_rmk}).}
This is again a manifestation of the fact that comparisons to periodic functions 
are very sensitive even to small perturbations in description of respective spatial positions as encoded by local phases.
}
\er

%NEWCHANGED-MR
%\subsubsection{Decay for localized perturbations}
{\it Decay for localized perturbations.} 
Standard bounds on localized solutions of systems of parabolic
conservation laws of form \eqref{whitham} (see Proposition \ref{cl_lemma} below)
show that in general
$$
\|\kappa(t)-\ks\|_{L^p(\RM)}\sim (1+t)^{-\frac{1}{2}(1-1/p)},
\quad
\|\Psi(t)-\Id\|_{L^\infty(\RM)}\sim 1,
$$
so that
\be\label{formal_bd2}
\|U^{\cM(\cdot,t),\kappa(\cdot,t)}(\Psi(\cdot,t))
-\bar U(\cdot)\|_{L^\infty(\RM)}
\sim
\|\Psi(t)-\Id\|_{L^\infty(\RM)}
\sim 1
\ee
$\gg (1+t)^{-\frac{1}{2}+\eta}$.
Together with \eqref{last},
this rigorously validates the formal Whitham approximation
while simultaneously showing that estimates \eqref{mainest}--\eqref{andpsi} are
sharp for nonlocalized perturbations, $h_0\not \equiv 0$, 
leading always to nontrivial localized data \eqref{initial_data} 
in $\kappa_W$ for the Whitham system \eqref{whitham},
and for localized perturbations $h_0\equiv 0$
are sharp in the generic case where no uncoupling is present. 

However, an interesting further implication of \eqref{last}
is that when $\kappa_W$ decouples to sufficient order from
the rest of the Whitham equations,
the estimates \eqref{mainest}--\eqref{andpsi} {\it can be sharpened for
localized perturbations, to yield asymptotic decay}.
To make this latter point precise, we introduce the following definitions.

\begin{definition}\label{cases}
We say that a wave is {\it linearly phase-decoupled} if 
$\partial_M c|_{(\Ms,\ks)}=0$, or, equivalently, 
$\partial_M \omega (\Ms,\ks)=0$ in \eqref{whitham}:
that is, $\kappa$
is a characteristic variable for \eqref{whitham} at the special point $(\Ms,\ks)$. Otherwise, we will say 
that it is {\it linearly phase-coupled}, or simply ``generic type.''

\end{definition}

\begin{definition}\label{quadcases}
We say that a wave is {\it quadratically phase-decoupled} if both
$\partial_M c|_{(\Ms,\ks)}=0$ and $\partial_M^2 c|_{(\Ms,\ks)} =0$.
or, equivalently, 
$\partial_M \omega (\Ms,\ks)=0$ and $\partial_M^2 \omega (\Ms,\ks)=0$ 
in \eqref{whitham}.
\end{definition}

This simple classification unifies and generalizes a number of observations
in \cite{OZ1,OZ2,JZ2,JZ3}. As we have seen in Lemma~\ref{jordanlem} and shall discuss further in Remark~\ref{scalingrmk}, linear phase-decoupling implies that {\it to linear order} the phase behaves similarly as in the reaction-diffusion case studied in \cite{S1,S2,JZ1} (localized perturbations) and \cite{JNRZ1,JNRZ2,SSSU} (nonlocalized), for which the associated Whitham system consists of a single 
%CHANGED-DL
%equation
equation\footnote{See Remark \ref{rd_to_scalar}.}
$$
\kappa_t-\ks(\omega(\kappa)+\cs\kappa)_x=\ks^2(d(\kappa)\kappa_x)_x,
\qquad
\omega(k)=-k\,c(k),
$$
encoding the nonlinear dispersion relation induced by the periodic existence theory. In particular, we shall show that spectrally stable linearly phase-decoupled waves like spectrally stable reaction-diffusion waves are {\it linearly and nonlinearly asymptotically stable} and not only boundedly stable {\it with respect to localized perturbations}. Yet {\it at the nonlinear level} the phase behaves similarly as in the reaction-diffusion case, sharing the same decay rates, only if quadratic decoupling is present. The situation is actually simpler in higher dimensions where the asymptotic dynamics are essentially linear and the distinction between linearly decoupled and generic cases is sufficient \cite{OZ4,JZ2}. We make these observations precise in the following corollary. For a proof, see Appendix \ref{s:pfCorphasethm}.

\bc[Localized perturbations]\label{phasethm}
Under the assumptions of Theorem \ref{main}, 
for localized perturbations $h_0\equiv 0$, 
and $\psi$ defined as in Theorem \ref{main},
if $\bar U$ is linearly phase-decoupled and 
$E_1:=E_0+\||\cdot|\,(\tilde u_0-\bar U)\|_{L^1(\RM)}$ is sufficiently small, then,
for $t>0$ and  $2\le p\le \infty$,
\ba\label{linimp}
\|\nabla_{x,t} \psi(t) \|_{L^p(\RM)}
&\lesssim
E_1 (1+t)^{-\frac{1}{2}(1-1/p)-\frac14+\eta},\\
\|\tilde u(t)-\bar U\|_{L^p(\RM)}, \quad
\| \psi(t) \|_{L^p(\RM)} &\lesssim
E_1(1+t)^{\frac{1}{2p}-\frac14+\eta},
\ea
while if $\bar U$ is quadratically phase-decoupled and $E_0$ is sufficiently small, then
\ba\label{quadimp}
\|\nabla_{x,t} \psi(t) \|_{L^p(\RM)}
&\lesssim
E_0 (1+t)^{-\frac{1}{2}(1-1/p)-\frac12+\eta},\\
\|\tilde u(t)-\bar U\|_{L^p(\RM)}, \quad
\| \psi(t) \|_{L^p(\RM)} &\lesssim
E_0(1+t)^{-\frac12(1-1/p)+\eta }\ ;
\ea
for $t>0$ and $2\le p\le\infty$.
In either case, $\bar U$ {\rm is nonlinearly asymptotically stable}
from\\ $L^1(\R;(1+|x|)dx)\cap H^K(\R)$ to $L^p(\R)$, for all $2<p\le \infty$.
%CHANGED-DL: moved in remark below
%\footnote{
%The anomalous $(1+t)^{\frac14}$ factor appearing in \eqref{linimp} is due to quadratic order coupling \cite{L};
%see Remark \ref{Liu_explanation}.}
\ec
%ENDCHANGED

\br\label{comprd}
\textup{
Comparing bounds \eqref{quadimp} for localized perturbations and
\eqref{mainest}--\eqref{andpsi} for nonlocalized perturbations
to those obtained in \cite{SSSU,JNRZ1,JNRZ2}, we see that bounds
for the quadratically phase-decoupled case exactly match the
bounds for reaction-diffusion systems.
}
\er

%TODO-MR: reword if needed
%CHANGED-MR
\br\label{zero-mean}
\textup{
One may wish to express localization as a mean-free condition on $\d_xh_0$. Actually, in the above bounds, the condition $h_0\equiv0$ may indeed be relaxed 
%CHANGEDkz added
to the condition that
%to
$\d_xh_0$ is mean-free and 
either $E_1:=E_0+\||\cdot|\,\d_xh_0\|_{L^1(\RM)}$ is small 
in the quadratically phase-decoupled case or 
$E_1:=E_0\||\cdot|\,\d_xh_0\|_{L^1(\RM)}+\||\cdot|\,(\tilde u_0(\cdot-h_0(\cdot))-\bar U)\|_{L^1(\RM)}$ is small in the linearly phase-decoupled case. 
%CHANGED-DL
In either case, the conclusion is (asymptotic) orbital stability with asymptotic phase $\psi_\infty=(h_0(-\infty)+h_0(\infty))/2$ (in the sense of \cite{He}).
}
\er
%

%CHANGED-DL
\br\label{no_interaction}
\textup{
For the analysis of localized perturbations in the linearly phase-decoupled case, assumption (H3) plays a role deeper than just providing regularity in a simple way. Indeed, in this case, the extra damping $(1+t)^{-1/4}$ in \eqref{linimp} encodes the fact that quadratic interactions between diffusion waves traveling at different characteristic speeds are asymptotically irrelevant (\cite{L}, see Remark \ref{Liu_explanation}). Thus, here one should not expect to be able to replace (H3) with something weaker than: the linear group velocity associated to the wavenumber mode is different from all other characteristic speeds.
}
\er
%

%CHANGED-DL: through expamples k-->kappa
\subsection{Examples}\label{s:examples}
Having established the importance for asymptotic behavior of
the Whitham equations, we now give some examples indicating
their range of possible behaviors.
In this section, we relax the restriction,
made for expositional simplicity, to second-order
parabolic semilinear systems of conservation laws and discuss
a full range of models arising in applications,
including reaction-diffusion equations (Example \ref{sheg}),
equations with higher-order or partial diffusion
(Examples \ref{sheg}, \ref{kseg} and \ref{sveg}),
and even mixed conservative/nonconservative equations (Example \ref{sveg}).

From the point of view of the present paper, the main example
is Example \ref{viscoeg}, which illustrates for second-order parabolic
semilinear conservation laws both phase-decoupling and phase-coupling.
However, we emphasize that the analysis of all of these models
may be carried out with minor changes
within the same basic analytical framework set out here
and in \cite{JZ2,JZ3}.
We discuss this further in Appendix~\ref{s:gen}, along with the
question of numerical or analytical verification of the stability
conditions, needed to conclude validity of the Whitham equations.

To simplify the discussion, we restrict to the first-order part of the Whitham equations,
which suffices to determine the main qualitative features of solutions-
in particular, phase-decoupling vs. coupling-
and has a common derivation/form \cite{Se} independent
of second- and higher-order terms.

\begin{example}[\cite{S2}]\label{sheg}
\textup{
The Swift-Hohenberg equation
\be\label{sheq}
u_t+(1+\partial_x^2)^2u-ru+f(u)=0,
\ee
where $r\in\RM$ is a bifurcation parameter and $f$ is some sufficiently smooth nonlinearity, admits for certain values of $r$
periodic waves of speed $c\equiv 0$. This equation arises as a simplified equation for the Taylor-Couette problem and is proved to possess diffusively spectrally stable waves \cite{S2}. As a reaction-diffusion equation, with no conservative part, this yields 
(see, e.g., \cite{HK,DSSS,SSSU,JZ2,JZ3}) a scalar first-order Whitham 
%CHANGED-DL
%equation 
equation\footnote{See Remark \ref{rd_to_scalar}.}
$$
\kappa_t=0.
$$
}
\end{example}

\begin{example}[\cite{OZ1,Se,BLeZ}]\label{viscoeg}
\textup{
The equations of one-dimensional viscoelasticity
with artificial viscosity and strain-gradient effects (``capillarity'')
may be expressed in
Lagrangian coordinates, after a change of variables \cite{ScS,OZ2}, as
\ba\label{viscoeq}
\tau_t - u_x&=\eps_1 \tau_{xx},\\
u_t - \sigma(\tau)_x& =\eps_2 u_{xx},\\
\ea
where $\tau= \chi_x\in \RM^d$ and $u=\chi_t\in \RM^d$ are derivatives 
of deformation $\chi:x\to \RM^d$, $\sigma$ is the stress-strain 
relation of the elastic material, and $\eps_1,\eps_2>0$ 
are scalar coefficients related to viscosity/capillarity; see \cite{BLeZ}. 
It is readily verified by energy considerations
\cite{Se,OZ1,PSZ} that periodic waves may only have speed $c=0$.
By this, together with Galilean invariance with respect to shifts in $u$,
we find, setting $T$, $U$, $\Sigma$ to be the means of
$\bar \tau$, $\bar u$, and $\sigma(\bar \tau)$
over one period, that the associated first-order Whitham system has form
\ba\label{viscowhit}
T_t - U_x&= 0,\\
U_t - \Sigma(T,\kappa)_x& =0,\\
\kappa_t=0,
\ea
hence, since $\omega(T,U,k)\equiv 0$ is evidently independent of $(T,U)$, is quadratically (indeed, totally) phase-decoupled.
That is, in terms of the phase equation principally determining behavior, \eqref{viscoeq} and \eqref{sheq} exhibit parallel behavior
$\kappa_t=0$,\footnote{At second order, a linear
heat equation $\kappa_t=d\kappa_{xx}$, $d>0$.}
despite their different origins.
For localized data $h_0\equiv 0$, giving $\kappa_0\equiv \ks$, this reduces
to the first-order wave equation
$T_t - U_x= 0$, $U_t - \Sigma(T,\ks)_x =0$,
which is strictly hyperbolic when $\Sigma$ is monotone decreasing, 
as occurs for some but not all cases \cite{Y}. 
By Corollary \ref{phasethm}, any spectrally
stable periodic solutions of \eqref{viscoeq}
would be nonlinearly asymptotically stable with respect to localized
perturbations, answering a question posed
in \cite{OZ2,Se,JZ2}.}

\textup{
The same equations written in Eulerian coordinates (in
terms of $\rho :=\tau^{-1}$ and $m:=u/\tau$) are {\it phase-coupled}
\cite{Se}, hence at best only nonlinearly bounded stable, with the explanation that
modulations in wave speed in this case lead to deviation of characteristic paths,
hence solutions are no longer compared along the Lagrangian trajectories
where they are most closely matched. To put things another way, coordinatization by Lagrangian markers accomplishes
a substantial part of the modulation that in Eulerian coordinates yields decay estimate \eqref{mainest}(i).
}
\end{example}

\br\label{pszrmk}
\textup{
It has been shown in \cite{OZ1,PSZ} that under a wide variety of circumstances, 
in particular, always for one-dimensional deformations, $d=1$,
periodic solutions of \eqref{viscoeq} are spectrally unstable,
{\it whether or not the first-order Whitham equations \eqref{viscowhit}
are of hyperbolic type.}
This shows the importance of the full diffusive stability conditions,
beyond the intuitive conditions of Corollary~\ref{DHcor}.
It is an interesting open problem whether there exist stable
waves for $d>1$ \cite{Y,PSZ}.
}
\er

\begin{example}\label{kseg}
\textup{
For the Kuramoto--Sivashinsky equation
\be\label{kseq}
u_t +(u^2/2)_x + u_{xx}+ u_{xxxx}=0,
\ee
setting $U$, $\Sigma$ to be the means of $\bar u$, and $\bar u^2/2$
over one period, and making use of the Galilean invariance $x\to x-ct$, $u\to u+c$,
we find that $c(U_1+U_2,k)=c(U_1,k)+U_2$  and $\Sigma(U_1+U_2,k)=\Sigma(U_1,k)+U_1U_2+U_2^2/2$. 
It is known that, within a certain parameter range, \eqref{kseq}
supports odd- hence mean-free- profiles with $c\equiv 0$ (see Remark \ref{reflectrmk}).
For these solutions, $c(0,k)=0$, hence $c(U,k)=U$. Thus, the associated first-order Whitham system is 
{\it linearly phase-coupled}, of form
\ba\label{kswhit}
U_t + \big(\Sigma(U_\star,\kappa)+(U-U_\star)^2/2\big)_x&= 0,\\
\kappa_t+((U-U_\star)\kappa)_x &=0\ .
\ea
Linearizing about constant solution $(U_\star,\ks)$ and reintroducing the 
phase through $k=\psi_x$, this gives a second-order wave equation
$\psi_{tt} + \ks(\d_k\Sigma)(U_\star,\ks)\ \psi_{xx}=0$ in the phase \cite{FST} provided $(\d_k\Sigma)(U_\star,\ks)<0$. As illustrated numerically in \cite{BJNRZ3}, linear phase-coupling has the effect that nonlocalized perturbations in the phase can arise even through localized initial perturbations.
Numerical studies \cite{FST,BJNRZ3} indicate that there exist ``bands'' in parameter space
of spectrally stable waves, satisfying hypotheses (H1)--(H3), (D1)--(D3); see Appendix \ref{s:gen} for further discussion.
}
\end{example}

\begin{example}\label{sveg}
\textup{
The Saint-Venant equations for inclined thin-film flow appear
in Lagrangian coordinates as
\ba \label{sveq}
\tau_t - u_x&= 0,\\
u_t+ ((2F)^{-1}\tau^{-2})_x&=
1- \tau u^2 +\nu (\tau^{-2}u_x)_x ,
\ea
where $\tau$ is the reciprocal of fluid height, $u$ is velocity averaged with respect to depth, 
$x$ denotes a Lagrangian marker moving with the flow, and $\nu$ and $F$ are dimensionless constants, 
with force term $1-\tau u^2$ representing the balance between 
gravity and turbulent bottom friction. 
In terms of structure, this is intermediate between the reaction-diffusion
case of \eqref{sheg} and \eqref{viscoeq}, having a first equation in
conservative (divergence) form and a second equation in
nonconservative convection-reaction-diffusion form.
The same derivation as for \eqref{whitham} yields the first-order
Whitham system
\ba\label{svwhit}
T_t- U(T,\kappa)_x&=0,\\
\kappa_t- (c(T,\kappa)\kappa)_x&=0,
\ea
where $T$ and $U$ are defined as the means of $\bar \tau$ and $\bar u$ over one period, and $(T,k)$
parametrize the associated two-parameter family of periodic traveling waves
with speed $c=c(T,k)$; see \cite{JZN,BJNRZ1,NR2} for further details.
For this model, the speed $c$ is never zero, and in particular 
depends typically nontrivially on $T$. Thus, this system, like \eqref{kswhit}, is in general fully phase-coupled.
Numerical experiments \cite{BJNRZ1,BJNRZ4} indicate that \eqref{svwhit} can be either hyperbolic (consistent
with stability) or elliptic (implying instability), depending on 
parameter values; moreover, there exists a band of parameters on which waves satisfy the stability
hypotheses (H1)--(H3), (D1)--(D3).
}
\end{example}

\br[Phase-decoupling and symmetry]\label{reflectrmk}
\textup{
As illustrated by Example \ref{viscoeg},
phase-decoupling is not always an isolated degeneracy on
a special set of parameters, but for models with special structure
may hold on an open set of parameters/waves; indeed, more,
we may have $c\equiv 0$.
This is reminiscent of the well-known principle in the reaction-diffusion
setting that, by reflection symmetry of the equations
$u_t+f(u)=u_{xx}$, even-symmetric standing-wave solutions generically
persist as families of solutions with $c(k)\equiv 0$.
For, otherwise, the fact that reflection preserves $k$ would violate
local uniqueness of solutions as a function of $k$.
Alternatively, one may observe that zero-speed waves satisfy a Hamiltonian
ODE, hence, by a dimensional count, exhaust the available dimensions
in the set of nearby solutions.
This principle is illustrated in the behavior cited in Example \ref{sheg}.
}

\textup{
Likewise, in Example \ref{viscoeg}, one finds \cite{JZN,BJNRZ1} that
zero-speed waves satisfy a Hamiltonian ODE identical to that of
the reaction-diffusion case, with $d$ free parameters
given by a constant of integration, hence, by a dimensional count,
generically fill up the $(d+1)$-dimensional set of nearby solutions,
giving $c(T,k)\equiv 0$.
Similarly, the Kuramoto--Sivashinsky equations \eqref{kseq}, 
are invariant under $x\to -x$, $c\to -c$, $u\to -u$, from which we
may deduce that odd-symmetric zero-speed solutions generically
persist, as cited in Example \ref{kseg}.
For, otherwise, the fact that reflection preserves $k$ 
would violate uniqueness with respect to $k$ of solutions with
fixed zero mean.
}
\er

\subsection{Discussion and open problems}\label{s:discussion}
Our results extend to the conservative case the results established
recently for reaction-diffusion systems in \cite{SSSU,JNRZ1,JNRZ2}
regarding behavior, and extend to nonlocalized perturbations the results
obtained for conservation laws
in \cite{JZ1,JZ2,JZN,BJNRZ1,BJNRZ2} regarding stability under localized perturbation.
The method of analysis used here is similar to but much more complicated than the arguments used in \cite{JNRZ1,JNRZ2}
to study the reaction-diffusion case
and the reader is encouraged to consult these references as motivation in a simpler context. 
As noted above, the methods of \cite{SSSU} do not seem to apply.

The main new difficulties overcome in the present analysis
beyond that of \cite{JNRZ1,JNRZ2}
are the treatment of nonlocalized perturbations in a way including the phase-coupled case, 
which has an essentially different Jordan block structure from that of the phase-decoupled case, and the identification of 
the Whitham equations with the asymptotic second-order modulation system arising naturally in our analysis
via a system of integral equations. 
The latter task involves surprisingly subtle aspects not present
in the reaction-diffusion case concerning the influence of phase modulation
$\Psi$ on the mean $\mathcal{M}$, first, through high-frequency resonances,
on its initial data (see Remark \ref{formal_rmk}), 
and, second, through the influence of the implicit
nonlinear change of independent coordinates \eqref{pertvar}
used in our nonlinear iteration scheme on the form of the Whitham
equations (see Section \ref{s:implicitchange}).\footnote{
Recall that $\mathcal{M}$ does not appear
in the Whitham equation for the reaction-diffusion case.}
Though we give a unified proof, regardless of coupling distinctions, 
mainly out of a desire for clarification of the essential features of modulation theory, 
a proof of Theorem \ref{main} would not be much simpler 
had we restricted it to the linearly-uncoupled case.

Comparing bounds \eqref{refinedest}-\eqref{last} to the corresponding bounds for reaction-diffusion systems
in \cite[Theorem 1.3]{JNRZ2}, we see that they are identical; that is, modulations are equally well-approximated for
systems of conservation laws as for reaction-diffusion systems by the formal Whitham approximation \eqref{whitham}.
It follows (through Proposition \ref{cl_lemma}) that for nonlocalized perturbations, {\it behavior and decay rates are also
essentially identical in these two cases,} as the formal asymptotics suggest.

On the other hand, for localized data, $h_0\equiv 0$, the decay estimates established for the reaction-diffusion case in \cite{OZ4,JZ1}
are faster by a factor $(1+t)^{-\frac12}$, or ``roughly one derivative'' in terms of standard heat bounds, than those of Theorem \ref{oldmain}
in the generic conservative case, which are the same for localized as for nonlocalized perturbations.
As discussed in \cite{JZ3}, linearly phase-decoupled waves exhibit a similar behavior at the linearized level. However Corollary~\ref{phasethm} shows 
that a quadratic decoupling is needed to yield a similar behavior at the nonlinear level, while waves that are linearly phase-decoupled show an intermediate behavior, asymptotic stability but with slower rates. 

At broadest level, our results confirm that an accurate distinction 
is not between reaction-diffusion and conservation law systems but
between {\it phase-decoupled}, and {\it non-phase-decoupled} waves 
(the former trivially including the reaction-diffusion case), 
which indeed exhibit the asymptotic behavior suggested by their common formal asymptotic description in terms of the Whitham equation(s).
A key new piece of information supplied by our analysis that is not present in the formal Whitham derivation
is the way in which initial data is taken on by the time-asymptotic Whitham system. Though the ultimate prescription 
in Theorem \ref{main} is simple, it is determined by a detailed series of linear and nonlinear estimates that are quite far from the techniques of formal asymptotic expansion.

We stress, finally, that our nonlinear iteration scheme is quite
robust. In particular, there is no use of analytic semi-group properties
in our argument, hence it is not sensitive to changes in order (e.g., to KS or KS-KdV) or type 
(e.g., quasilinear or degerate as for Saint-Venant) of the equations under study.
To control regularity, we mainly  use nonlinear $H^s$ damping estimates and 
$C^0$ semi-group resolvent  bounds in  $H^s$; both given by standard energy estimates techniques  
(Kawashima's if needed \cite{Ka,Z2,Z4}).
This allows a wide range of generalizations, as discussed in Appendix
\ref{s:gen}.

The diffusive spectral stability conditions (D1)--(D3) have been shown numerically to hold for ``bands'' of stable periodic waves,
in several interesting settings, and with a high degree of precision; see, for example, \cite{FST,BJNRZ2,BJNRZ3}. We view the numerical proof of these conditions, or analytical proof in interesting asymptotic limits (in the spirit of \cite{JNRZ3}), as important open problems for the theory.
The determination of asymptotic behavior in the small-wavelength limit, analogously as in \cite{DSSS}
for the reaction-diffusion case, is another important open problem. 
Likewise, extensions to the case that not only the phase
but the wave number $\kappa=\Psi_x$ has different values at plus and
minus infinity, corresponding to Riemann data for the Whitham equation 
\eqref{whitham}, is an interesting direction for future investigation; 
see \cite{DSSS,BNSZ} in the reaction-diffusion case.

\medskip

{\bf Plan of the paper:}
The plan of the rest of the paper is as follows.
In Section \ref{s:prelim} we set up the framework of the proofs (introduction of the phase, 
integral transform, etc.), then in Section \ref{s:stabg} we prove Theorem~\ref{oldmain}.
Finally, in Section \ref{s:gbehavior}, we give the proof of Theorem~\ref{main}. 
In Appendices \ref{algebraic}--\ref{s:gen} we provide, respectively,
algebraic relations obtained by differentiation of the traveling-wave ODE, 
derivation of the Whitham system \eqref{whitham},
simplifications of this system afforded by the theory of parabolic conservation laws 
for data consisting of localized perturbations of a constant state, 
and some hints regarding generalizations to other situations of mathematical or physical interest.

\section{Preliminaries}\label{s:prelim}

In this section, we discuss several technical preliminaries that we will find useful throughout our analysis. 
Specifically:
\begin{itemize}
\item we introduce the Bloch transform, the fundamental integral transform that we will use in 
deriving all of our linear estimates;
\item we show how phase shift $\psi$ is introduced in \eqref{cons} and how this affects the equations;
\item we prove a nonlinear damping energy estimate (here simply following from the parabolic nature of \eqref{cons}), establishing that high derivatives of the solution decay in time at least as fast as low derivatives so that 
%CHANGEDkz
%the issue 
technical issues are mainly,
%is mainly, 
as expected, in decay rates and localization (small Floquet or small Fourier numbers) and not in regularity;
%TODO: the issue.. technical issues are confined.. Or?? what?
\item we give estimates useful to analyze the effect of 
a change of independent variables on our bounds.
\end{itemize}
These preliminary issues were already 
present implicitly or explicitly in \cite{JZ2,JZ3}.

\subsection{Bloch decomposition}\label{s:transform}

To begin our analysis of the stability of a fixed 1-periodic stationary solution $\bar U$ of \eqref{cons},
recall from above that linearizing the flow of \eqref{cons} about $\bar{U}$ leads to the consideration
of the 1-periodic coefficient linear evolution equation \eqref{lin}. 
From Floquet theory, one may guess that it would be desirable in
analyzing this equation to decompose solutions as superpositions
of functions having a given Floquet exponent 
$\xi\in [-\pi,\pi]$,
i.e., functions $e^{i\xi \cdot}h(\cdot)$ with $h\in L^2([0,1])_{\rm per}$,
as described in Remark \ref{specremarks}.
This may be accomplished using the {\it Bloch transform}.

Given a function $g\in L^2(\RM)$, its 
{\it Bloch decomposition}, or 
inverse Bloch transform representation,
is defined as
\be\label{Brep}
g(x)=\int_{-\pi}^\pi e^{i\xi x}\check{g}(\xi,x)d\xi,
\ee
where
$$
\check{g}(\xi,x):=\sum_{k\in\ZM}e^{2\pi ikx}\hat{g}(\xi+2\pi k),
$$
and
$\hat{g}(z):=\frac{1}{2\pi}\int_{\RM}e^{-i\omega z}g(\omega)d\omega  $
is the Fourier transform of $g$.
Note that for any $\xi\in[-\pi,\pi]$, $\check{g}(\xi,\cdot)$ is a 1-periodic function, 
hence, as desired, $e^{i\xi\cdot}\check{g}(\xi,\cdot)$ has 
Floquet exponent $\xi$. 

Letting $\mathcal{B}:L^2(\RM)\to L^2([-\pi,\pi];L^2_{\rm per}([0,1])),$ $g\mapsto \check{g}$ denote the Bloch transform,  
we readily see that for our given linearized operator $L$, defined in \eqref{lin}, and $g\in L^2(\RM)$
we have $\mathcal{B}(Lg)(\xi,x)=L_\xi\left[\check{g}(\xi,\cdot)\right](x)$,
hence the associated Bloch operators $L_\xi$ may be viewed as operator-valued 
symbols under $\mathcal{B}$, acting on $L^2_{\rm per}([0,1])$.
Similarly, from the identity $\mathcal{B}\left(e^{tL}g\right)(\xi,x)=\left(e^{tL_\xi}\check g(\xi,\cdot)\right)(x)$, a consequence
of \eqref{defrel}, we find the {\it Bloch solution formula} for the periodic-coefficient operator $L$:
\be\label{fullS}
(S(t) g)(x):=(e^{tL}g)(x) = \int_{-\pi}^{\pi} e^{i\xi x} (e^{tL_\xi}  \check g(\xi, \cdot))(x) d\xi.
\ee
In particular, we see that the Bloch transform $\mathcal{B}$ diagonalizes the periodic coefficient
operator $L$ in the same way that the Fourier transform diagonalizes constant-coefficient operators.

Using the representation formula \eqref{fullS}, bounds on the Bloch solution operator $e^{tL_\xi}$ can be converted 
to bounds on the linearized solution operator $e^{tL}$.  To facilitate these bounds, we notice by the standard Parseval identity
that the rescaled Bloch transform $\sqrt{2\pi}\mathcal{B}$ is an isometry on $L^2(\RM)$, i.e.
\begin{equation}\label{psvl}
\|g\|_{L^2(\RM)}^2=2\pi\int_{-\pi}^\pi\int_0^1|\mathcal{B}(g)(\xi,x)|^2dx~d\xi=2\pi\|\check{g}\|_{L^2([-\pi,\pi];L^2([0,1]))}^2.
\end{equation}
More generally, by interpolating \eqref{psvl} with the triangle inequality, corresponding to the case $q=1$ and $p=\infty$ below,
we obtain the generalized Hausdorff--Young inequality $\|g\|_{L^p(\RR)} \le \| \check g\|_{L^q([-\pi,\pi],L^p([0,1]))}$
for $q\le 2\le p $ and $\frac{1}{p}+\frac{1}{q}=1$, which, by \eqref{Brep}, yields for any $1$-periodic
functions $g(\xi,\cdot)$
\be\label{hy}
\Big\|\int_{-\pi}^{\pi}e^{i\xi \cdot} g(\xi,\cdot)d\xi \Big\|_{L^p(\RR)}
\le \|g\|_{L^q([-\pi,\pi],L^p([0,1]))}
\:\; {\rm for} \:\;
q\le 2\le p \:\; {\rm and } \:\; \frac{1}{p}+\frac{1}{q}=1.\footnote{
Here and elsewhere, we denote $\|g\|_{L^q([-\pi,\pi],L^p([0,1]))}:=\left(\int_{-\pi}^\pi\|g(\xi,\cdot)\|_{L^p([0,1])}^{q}d\xi\right)^{1/q}$.}
\ee
It is from this convenient formulation that we will obtain our linear estimates.

\br\label{nonsharp_rmk}
\textup{
To keep technicalities as low as possible, we have indeed compelled ourselves to prove all our linear estimates using \eqref{hy} alone. 
The only price to pay is that we are thus confined to high-norm estimates, in $W^{k,p}$, $2\le p\le \infty$,
as a consequence of which, due to the details of the nonlinear iteration estimates, 
bound \eqref{refinedest}(i) (thus \eqref{last}(i)) is not expected to be sharp except for $p=2$.
With more work, but in the spirit of the present paper, we expect that one may actually prove a $(1+t)^{-1/2(1-1/p)-1/2+\eta}$ decay ($\eta>0$ arbitrary).
See Remarks 1.7 and 4.3 of \cite{JZ3} for further discussion in the
somewhat simpler setting of the reaction-diffusion case.
}
\er

\br\label{slowmod_Bloch_rmk}
\textup{
Another important property of the Bloch transform that we will 
use repeatedly throughout our proofs is that it is well-behaved with respect to 2-scale analysis 
of a slow modulation \emph{ansatz}. Indeed if $g$ is 1-periodic and $h$ is slow, in the sense that the Fourier transform of $h$ is supported in $[-\pi,\pi]$, then $\mathcal{B}(gh)(\xi,x)=g(x)\widehat{h}(\xi)$. If no slowness assumption is made but still $g$ is 1-periodic, then 
there still holds $\mathcal{B}(gh)(\xi,x)=g(x)\check{h}(\xi,x)$. Mark that the fact that in general high frequencies of $h$ are involved in the low Floquet 
number part of $gh$ will add substantial difficulties to the linear analysis below. 
This nice property that the Bloch transform separates scales may also 
be used to perform two separate change of frames for fast and slow variables through $\check{g}(\xi,x)\mapsto e^{i\xi c_1t}\check{g}(\xi,x-c_2 t)$. 
This feature was used in a crucial way in the analysis of reaction-diffusion systems carried out in \cite{SSSU}, 
where the proof relied strongly on self-similar techniques such as renormalization, requiring the Whitham equation to be essentially reduced to a viscous Burgers equation by going into the frame of its characteristic velocity (called linear group velocity) while keeping the original equation in the co-moving frame of the wave (the one of the phase velocity). This elegant strategy appears to completely break down in the system case considered here, 
for which many linear group velocities are involved. 
}
\er

\subsection{Nonlinear perturbation equations}\label{s:pert}

We now discuss how the introduction of a phase shift affects system \eqref{cons}. 
Following \cite{JZ2,JNRZ1}, we introduce the perturbation variable
\begin{equation} \label{pertvar}
v(x,t)=\tilde{u}(x-\psi(x,t),t) -\bar U(x)
\end{equation}
where $\tilde{u}(x,t)$ satisfies \eqref{cons} and $\psi(x,t)$ is a phase shift to be determined together with $v$. 
We recall the following representation convenient for nonlinear iteration and established in \cite{JZ2}.

\begin{lemma}[\cite{JZ2}]\label{lem:canest}
The nonlinear residual $v$ and phase shift $\psi$ linked by \eqref{pertvar} satisfy
\begin{equation}\label{veq}
\left(\partial_t-L\right)(v+\psi\bar U_x)=\cN,\qquad\textrm{with}\qquad
\cN:=\partial_x \cQ+ \partial_x \cR +\partial_t\cS,
\ee
where
\be\label{eqn:Q}
\cQ:=-\ks\left(f(\bar U+v)-f(\bar U)-df(\bar U)v\right),
\ee
\be\label{eqn:R}
\cR:=- v\psi_t+\ks^2 v_x\frac{\psi_{x}}{1-\psi_x}+ \ks^2\bar U_x\frac{\psi_x^2}{1-\psi_x},
\ee
and
\be\label{eqn:S}
\cS:= v\psi_x.
\ee
\end{lemma}

\br\label{goodchoice}
\textup{
As noted in \cite{JZ2}, the advantage of \eqref{pertvar},
$
\tilde u(\cdot,t)\ =\ (\bar U+v(\cdot,t))\circ\Psi(\cdot,t)
$
(with $\Psi(\cdot,t)=(\Id-\psi(\cdot,t))^{-1}$) over the (probably more natural) choice 
$
\tilde u(\cdot,t)\ =\ \bar U\circ\Psi(\cdot,t)+\tilde v(\cdot,t)
$
suggested by \eqref{wref}, is that the phase shift enters the equation only through commutators between composition with $\Psi$ and differentiation, 
hence only gradients of $\psi$ appear in the source terms $\cQ,\cR,\cS$ on the right-hand side of \eqref{veq}.
By contrast, the corresponding terms for choice $\tilde v$ would involve also terms of order $|\psi||\psi_x|$ 
including the nondecaying phase $\psi$ itself, thus making the decay too slow 
for our nonlinear iteration to close. This observation is by now classical 
in the stability analysis of traveling patterns;  see the related observations of \cite{DSSS} regarding stability of periodic reaction-diffusion waves, 
and \cite{Z1,Z2,MZ,HoZ} and \cite{TZ1,TZ2,TZ3,Z3}\footnote{
Specifically, \cite[Section 3]{TZ1} and
\cite[Section 2.2.1]{TZ2} (group invariance and uniqueness),
\cite[Section 3]{Z3} (translation-invariant center--stable manifold),
and \cite[Theorem 2.2.0]{TZ3} (Nash--Moser uniqueness theorem).
} 
for similar, earlier, observations in the context of viscous shock stability and bifurcation.
}

\textup{
However, notice also the more subtle aspect of the decomposition \ref{veq}
that it groups within the term $(\partial_t-L)(\psi \bar U_x)$
appearing on the left-hand side
the linear order source terms $\psi_t \bar U_x$ and $\psi_x \bar U_{xx}$
that are individually too large to handle in our later nonlinear iteration;
the term $(\partial_t-L)(\psi \bar U_x)$ is then used to cancel instantaneous
phase-modulations arising in the solution of the linearized equations,
as described just below.
This approach to detecting nonlinear cancellation
originates from the study of stability of
viscous shock solutions of systems of parabolic conservation laws
(see, e.g., \cite[Eq. (2.30)]{Z1}, \cite{Z2}, and especially 
\cite[Eq. (5.23), Cor. 5.4, p. 453]{HoZ}) and is fundamentally different from those introduced in the 
reaction-diffusion setting in \cite{DSSS,SSSU} 
based on normal forms and successive reductions/renormalizations,
which, as discussed in Remark \ref{slowmod_Bloch_rmk}(ii), appear
unlikely to work in the present case. Indeed, we view this distinction
as the key to the successful treatment in \cite{JZ2,JZ3} 
of nonlinear modulational stability in the presence of conserved quantities.
}
\er

%CHANGED-MR
%\subsubsection{Isolation of the phase}\label{s:phasedecomp}
{\it Isolation of the phase.} 
To motivate the more technical analysis of Section \ref{s:stabg},
we describe in informal fashion the way that we determine
the phase $\psi$, separating out principal nonlinear behavior.
Using Duhamel's formula together with \eqref{veq}, we can write an integral equation for 
$v$ as
\be\label{tied}
v(\cdot,t)+\psi(\cdot,t)\bar{U}_x\ =\ e^{tL}(v(\cdot,0)+\psi(\cdot,0)\bar{U}_x)+\int_0^t e^{(t-s)L}\mathcal{N}(s)ds,
\ee
with initial data $\psi(\cdot,0)=h_0$, $v(\cdot,0)=\tilde u_0\circ(\Id-h_0)-\bar U$. 
Noting, by the bounds of Section~\ref{s:stabg},
that $e^{tL}=\bar U_x e_{n+1}\cdot s^{\rm p}(t) +\tilde S(t)$, where
%$s^{\rm p}(t):L^\infty \to \RM^{n+1}$
$s^{\rm p}(t)$ is some operator sending $\R^n$-valued functions into $\R^{n+1}$-valued functions\footnote{The first space is the one of $U$-values, the second one is the one of modulation parameters $(\cM,\kappa)$.} 
($e_{n+1}$ denoting the $(n+1)$th standard basis element)
and $\tilde S(t)$ is a faster-decaying residual,
that is, that the principal part of the linear solution operator
is a linear phase-modulation consisting of $e_{n+1}\cdot s^{\rm p}(t)$
times the instantaneous shift $\bar U_x$, we remove this principal part by defining 
implicitly
\be\label{1psi}
\psi(t)\sim
 e_{n+1}\cdot s^{\rm p}(t)(v(\cdot,0)+ \psi(\cdot,0)\bar{U}_x)
+\int_0^t e_{n+1}\cdot s^{\rm p}(t-s)\mathcal{N}(s)ds,
\ee
where the $\sim$ here indicates equality for $t\geq 1$. This gives an expression
\ba\label{1v}
v(t)&\sim
\tilde S(t)(v(\cdot,0)+ \psi(\cdot,0)\bar{U}_x)
+\int_0^t \tilde S(t-s)\mathcal{N}(s)ds
\ea
for $v$ in which no $s^{\rm p}(t)$ terms appear, closing the system in $(\psi,v)$.

This simple prescription follows the principle that we should
choose the nonlinear phase so as to remove from the linear
description of the residual $v$ all contributions representing
linearized, or ``instantaneous'' phase modulation, at least away
from the initial layer $0\leq t\leq 1$.
On this latter time interval, where we are constrained
by the restriction $\psi(\cdot, 0)=h_0$,
we instead interpolate between
the right-hand side of \eqref{1psi} and the initial data $\psi(\cdot,0)=h_0$,
as described in \eqref{vimplicit}--\eqref{closedg}.

\subsection{Nonlinear damping estimate}\label{s:damping}
To complement the linear bounds, established below, that form the core of the proof, we will use the following damping-type bound established
by energy estimate in \cite{JZ2}, useful in controlling higher derivatives by lower ones, enabling us to close a nonlinear iteration with decay rates 
of the lower derivatives.

\begin{proposition}[\cite{JZ2}]\label{damping}
Assuming (H1)-(H3), there exist positive constants $\theta$, $C$ and $\varepsilon$ such that if $v$ and $\psi$ solve \eqref{veq} on $[0,T]$ for some $T>0$ and
$$
\sup_{t\in[0,T]}\|(v,\psi_x)(t)\|_{H^K(\RM)}+\sup_{t\in[0,T]}\|\psi_t(t)\|_{H^{K-1}(\RM)}\leq\varepsilon
$$
then, for all $0\leq t\leq T$,
\be\label{Ebds}
\|v(t)\|^2_{H^K(\RM)}
\leq Ce^{-\theta t}
\|v(0)\|_{H^K(\RM)}^2+
C\int_0^t e^{-\theta(t-s)}
\left(\|v(s)\|^2_{L^2(\RM)}+
\|(\psi_t,\psi_x,\psi_{xx})(s)\|_{H^{K-1}(\RM)}^2\right)ds.
\ee
\end{proposition}

\begin{proof}(from \cite{JZ2})
Rewriting \eqref{veq} as
\be\label{vperteq2}
(1-\psi_x) v_t-\ks^2v_{xx}=\ks\cs v_x-\psi_t(\bar U_x+v_x)-\ks(f(\bar U+v)-f(\bar U))_x
+\ks\left(\frac{\ks\psi_x}{1-\psi_x}(\bar U_x+v_x)\right)_x,
\ee
taking the $L^2(\RM)$ inner product against
$\sum_{j=0}^K \dfrac{(-1)^{j}\d_x^{2j}v}{1-\psi_x}$,
integrating by parts, and rearranging, we obtain
\[
\frac{d}{dt}
\|v\|_{H^K(\RM)}^2(t)
 \leq -\tilde{\theta} \|\partial_x^{K+1} v(t)\|_{L^2(\RM)}^2 +
\tilde C\left( \|v(t)\|_{H^K(\RM)}^2
+\|(\psi_t,\psi_x,\psi_{xx})(t)\|_{H^{K-1}(\RM)}^2 \right),
\]
for some positive $\tilde C$ and $\tilde\theta$, so long as $\|(v,v_x,\psi_t,\psi_x,\psi_{xx})(t)\|_{H^{K-1}(\RM)}$ remains sufficiently small. Sobolev interpolation
$\|g\|_{H^K(\RM)}^2\leq a^{-1}\|\partial_x^{K+1} g\|_{L^2(\RM)}^2+a\| g\|_{L^2(\RM)}^2$
gives, then, for $a>0$ sufficiently large,
\[
\frac{d}{dt}\|v\|_{H^K(\RM)}^2(t) \leq -\theta \|v(t)\|_{H^K(\RM)}^2 +
C\left( \|v(t)\|_{L^2(\RM)}^2+\|(\psi_t, \psi_x)(t)\|_{H^K(\RM)}^2 \right),
\]
from which \eqref{Ebds} follows by Gronwall's inequality.
See \cite{JZ2} for further details.
\end{proof}

\br\label{damping_rmk}
\textup{
It should be pointed out that one technical feature of our nonlinear iteration is the loss of derivatives of the nonlinear perturbation variable $v$, i.e.
our iteration argument shall control $L^p$ norms of $v$ in terms of $H^s$ norms of $v$ and gradients of $\psi$. This loss of derivatives is compensated by the above ``nonlinear damping" estimate, which is due to the (here, total) parabolicity of the governing equations. Of course, here, other strategies would be available (such as maximal regularity). The advantage of the above energy estimate is that it generalizes to partially parabolic systems (through the introduction of Friedrichs symmetrizers 
with Kawashima compensators;
%CHANGED-MR
% \cite{JZN}).
see \cite{JZN} for such an analysis in the context of the Saint-Venant system).
In contrast, 
the fact that the Bloch transform involves only bounded Floquet numbers translates immediately into the fact that local parameters such as $\psi$ (or $M$) 
should be ``slow,'' in the sense of Remark \ref{slowmod_Bloch_rmk},
so that arbitrarily many derivatives are gained at linear level.
}
\er

\subsection{Inverse modulation bounds}
An important technical detail \cite{DSSS,JNRZ1,JNRZ2}, and a new issue beyond the viscous shock wave case mentioned
in Remark \ref{goodchoice}, is the relation between quantities
$\|\tilde u(\cdot- \psi(\cdot,t), t)
-U^{\Ms+M(\cdot,t),\ks/(1-\psi_x(\cdot,t))}(\,\cdot\,)\|_{L^p(\RM)}$
conveniently estimable by our analysis and the corresponding quantities
$\|\tilde u(\cdot,t)
-U^{\cM(\cdot ,t), \kappa(\cdot,t)}(\Psi(\cdot,t) )\|_{L^p(\RM)}$
arising through formal Whitham approximation. With this in mind, we remark that the following (sharp) estimate shows that
$
\|F(\cdot)-G(\cdot + \phi(\cdot))\|_{L^p(\RM)}
$
is essentially equivalent to
$\|F(\cdot- \phi(\cdot))-G(\cdot)\|_{L^p(\RM)}$
plus $\|\phi\|_{L^\infty(\RM)}\|\phi_x\|_{L^p(\RM)}$.

\begin{lemma}\label{switchlem}
Let $\phi$ be bounded with $\|\phi_x\|_{L^\infty(\RM)}<1$. Then $\Id-\phi$ is invertible and
\be\label{switcheq}
\begin{array}{rcl}
\|F-G\circ(\Id-\phi)^{-1}\|_{L^p(\RM)}
&\le&(1+\|\phi_x\|_{L^\infty(\RM)})^{\frac1p}\quad
\|F\circ(\Id-\phi)-G\|_{L^p(\RM)}\\
\|F-G\circ(\Id+\phi)\|_{L^p(\RM)}
&\le&(1+\|\phi_x\|_{L^\infty(\RM)})^{\frac1p}\quad
\|F\circ(\Id-\phi)-G\|_{L^p(\RM)}\\
&&+\|G_x\|_{L^\infty(\RM)}(1+\|\phi_x\|_{L^\infty(\RM)})^{\frac1p}\|\phi\|_{L^\infty(\RM)}\|\phi_x\|_{L^p(\RM)}.
\end{array}
\ee
\end{lemma}

\begin{proof}
By the implicit function theorem and boundedness of $\phi$, the map $\Id-\phi$ is invertible. Let us write its inverse $\Id+\tilde{\phi}$. Since the Jacobian of $\Id+\tilde\phi$ is bounded below by $(1+\|\phi_x\|_{L^\infty(\RM)})^{-1}$, we remark that
$$
\|[F\circ(\Id-\phi)-G]\circ(\Id+\tilde\phi)\|_{L^p(\RM)}\ \leq\ (1+\|\phi_x\|_{L^\infty(\RM)})^{\frac1p}
\|F\circ(\Id-\phi)-G\|_{L^p(\RM)}
$$
and the first part of \eqref{switcheq} follows. We then split $F-G\circ (\Id+\phi)$ as
$$
F-G\circ(\Id+\phi)\ =\ [F\circ(\Id-\phi)-G]\circ(\Id+\tilde\phi)\ +\ G\circ(\Id+\tilde\phi)-G\circ(\Id+\phi).
$$
Now the intermediate value theorem yields
$$
\|G\circ(\Id+\tilde\phi)-G\circ(\Id+\phi)\|_{L^p(\RM)}\ \leq\ \|G_x\|_{L^\infty(\RM)}
\|\tilde\phi-\phi\|_{L^p(\RM)}.
$$
But, from equality $\tilde\phi=\phi\circ(\Id+\tilde\phi)$ we infer
$$
\tilde{\phi}(x)-\phi(x)\ =\ \tilde\phi(x)\ \int_0^1\ \phi_x(x+t\tilde\phi(x))\,dt
$$
from which H\"older's inequality yields
$$
\|\tilde{\phi}-\phi\|_{L^p(\RM)}^p\ \leq\ \|\tilde\phi\|_{L^\infty(\RM)}^p\ \int_0^1\ \|\phi_x\circ(\Id+t\tilde\phi)\|_{L^p(\RM)}^p\,dt.
$$
This concludes the proof since $\|\tilde\phi\|_{L^\infty(\RM)}\leq \|\phi\|_{L^\infty(\RM)}$ and, for $t\in[0,1]$, $\Id+t\tilde\phi$ is invertible with a Jacobian bounded below by $(1+\|\phi_x\|_{L^\infty(\RM)})^{-1}$. 
See \cite[Remark 1.4]{JNRZ2} for related comments.
\end{proof}

\br\label{needsmall}
\textup{
Note that the quantity $\tilde u(\cdot-\psi(\cdot,t),t)$
estimated in our analysis does not necessarily control $\tilde u(\cdot, t)$
unless $\ks\psi_x$ is small in $L^\infty$ (local invertibility of $\Id-\psi$)
and bounded in $L^1$ (global invertibility).
{\it Thus, our approach} (and likewise that of \cite{JNRZ1,JNRZ2,DSSS,SSSU})
{\it is inherently a small-variation analysis in wave number,}
whether or not the Whitham system admits large-variation
solutions.\footnote{
In particular, for the three example systems considered in the introduction, the Whitham system does
have large-amplitude solutions for data merely bounded in
$L^1\cap L^2$ since it has an associated convex entropy; see \cite{HoS}.}}
%TODO-MR: update footnote if needed

\textup{
However, if the Whitham system has an associated convex entropy \cite{La,Sm},
then, for $\kappa-\ks\sim \ks\psi_x$ initially small in $L^\infty$ and bounded in $L^1$
(hence small in $L^2$), by the results of \cite{HoS},
it has a solution that remains small in $L^\infty$, and decays
as $(1+t)^{-\frac12(1-1/p)}$ in all $L^p$, whence, combining the
stability and behavior arguments of Sections \ref{s:stabg}--\ref{s:gbehavior}
and closing an iteration for $z$ 
%CHANGED-MR
%(defined in quoted sections) 
(the refined perturbation variable defined in \eqref{pertvarrg}) 
instead of $v$, it might be possible
to relax the assumption $\|\psi_x(0)\|_{L^1\cap L^\infty}\ll 1$
to $\|\psi_x(0)\|_{L^1}=\cO(1)$ and $\|\psi_x(0)\|_{L^\infty}\ll 1$, allowing large variations in phase.
In the absence of a convex entropy, one might instead
assume closeness of $\psi_x(0)$ in $L^1\cap H^s$ to a special ``diffusion wave'' solution (in the scalar case, a distorted
Gaussian obtained by Hopf--Cole transformation; in the system case, a more complicated coupled superposition of such waves \cite{LZ})
that has evolved sufficiently long to be of small gradient, again relaxing slightly the restriction
that $\psi$ be of small initial variation. These would be interesting directions for future investigation.
}
\er

\section{Stability}\label{s:stabg}

In this section, we prove Theorem~\ref{oldmain}. In particular, to separate further \eqref{tied}, we first need a precise spectral analysis (proving Proposition~\ref{whiteig} 
%CHANGED-MR
%on the road) 
along the way) 
that will allow for a separation of $e^{tL}$ into a part aligned with $\bar U_x$ plus a faster-decaying term. 
Throughout this analysis, we shall often refer to algebraic relations obtained from the profile equation \eqref{ode} and stored in Appendix~\ref{algebraic}.

\subsection{Spectral analysis and nonlinear decomposition}\label{s:genprep}

\begin{proof}[Proof of Proposition \ref{whiteig}]
By \eqref{alg_order1}, one may choose $(\d_{M_1}\bar U,\dots,\d_{M_n}\bar U,\bar U')$ as a right basis for $\Sigma_0$, and $(e_1,\dots,e_n,\bar u^{adj})$ as 
the dual left basis, where $e_j$ denotes the constant function equal to the $j$th standard Euclidean basis element, and $\bar u^{adj}$ denotes a generalized zero eigenfunction of the adjoint operator $L_0^*=\ks^2\partial_x^2+\ks(df(\bar U)-\cs)\partial_x$ such that $\langle\bar u^{adj},\partial_M\bar U\rangle=0$ and $\langle\bar u^{adj},\bar U'\rangle=1$. By standard spectral perturbation theory \cite{Ka} we may build spectral projectors and appropriately extend
%CHANGED-MJ can only extend locally in \xi
(locally near $\xi=0$)
these dual bases in an analytic way into dual right and left bases $\{q_j(\xi)\}_{j=1}^{n+1}$
and $\{\tilde q_j(\xi)\}_{j=1}^{n+1}$ associated to the spectrum of $L_\xi$ in some
fixed neighborhood of the origin. This spectrum is then precisely the one of the matrix
\be\label{Lambda_xi}
\Lambda_\xi\ =\ \left(\langle \tilde q_j(\xi),L_\xi q_l(\xi) 
\rangle\right)_{j,l},
\ee
which we expand as $\xi\to 0$ as
$$
\Lambda_\xi\ =\ \Lambda_0+(i\ks\xi)\Lambda^{(1)}+(i\ks\xi)^2\Lambda^{(2)}+\cO(\xi^3).
$$
We expand also $L_\xi=L_0+(i\ks\xi)L^{(1)}+(i\ks\xi)^2L^{(2)}$; see \eqref{alg_L} for definitions of the $L^{(j)}$.

Note that, replacing, if necessary, simultaneously $q_{n+1}(\xi)$ with 
$$
q_{n+1}(\xi)-\xi\sum_{j=1}^n\langle\tilde q_j(0),\d_\xi q_{n+1}(0)\rangle\ q_j(\xi)
$$
and, for $j\neq n+1$, $\tilde q_j(\xi)$ with $\tilde q_j(\xi)+\xi\langle\tilde q_j(0),\d_\xi q_{n+1}(0)\rangle\ \tilde q_{n+1}(\xi)$, we may assume 
without loss of generality that, for $j\neq n+1$,
$$
\langle\tilde q_j(0),\d_\xi q_{n+1}(0)\rangle =0,\qquad \langle\d_\xi\tilde q_j(0),q_{n+1}(0)\rangle =0,
$$
the second inequality stemming from the first by expanding to first order in $\xi$ the duality relation $\langle\tilde q_j(\xi), q_{n+1}(\xi)\rangle=0$. Now note also that since $L_0 q_{n+1}(0)=0$, expanding to first order in $\xi$ the fact that $L_\xi q_{n+1}(\xi)$ lies in $\Sigma_\xi:=\Span\{q_j(\xi)\}_{j=1}^{n+1}$ the critical space of $L_\xi$, we find that $L_0\d_\xi q_{n+1}(0)+(i\ks)L^{(1)}q_{n+1}(0)$ lies in $\Sigma_0$, the generalized kernel of $L_0$. Using now \eqref{alg_order1}(iii) yields that $L_0(\d_\xi q_{n+1}(0)-(i\ks)\d_kU_{|(\Ms,\ks)})$ lies also in the generalized kernel of $L_0$,
thus so does $\d_\xi q_{n+1}(0)-(i\ks)\d_kU_{|(\Ms,\ks)}$. Orthogonality 
relations from above and \eqref{alg_order1}(iii) lead then to $\d_\xi q_{n+1}(0)-(i\ks)\d_kU_{|(\Ms,\ks)}\in\CM\bar U'$.

%CHANGED-MJ expanded for ease of reading.
%The above preparation provides
The above preparation yields the following representations for $\Lambda_0$ and the first- and second-order correctors
of the matrix $\Lambda_\xi$:
%
%\ba\label{structg}
$$
\begin{aligned}
\Lambda_0&=\bp 0_{n\times n}&0_{n\times 1}\\ \langle\tilde q_{n+1}(0),L_0 q_l(0)\rangle&0\ep=\bp 0_{n\times n}&0_{n\times 1}\\\d_M\omega|_{(\Ms,\ks)}&0\ep,\\[1em]
\Lambda^{(1)}&=\bp \frac{1}{i\ks}\langle \d_\xi\tilde q_j(0),L_0 q_l(0)\rangle+\langle\tilde q_j(0),L^{(1)}q_l(0)\rangle& 0_{n\times 1}\\ * & \frac{1}{i\ks}\langle \tilde q_{n+1}(0),L_0\d_\xi q_{n+1}(0)\rangle+\langle\tilde q_{n+1}(0),L^{(1)}q_{n+1}(0)\rangle\ep\\
&=\bp -\left((\d_MF)|_{(\Ms,\ks)}-\cs\Id\right)& 0_{n\times 1}\\ * &-\ks(\d_kc)|_{(\Ms,\ks)}\ep,\\
\Lambda^{(2)}&=\bp *&\frac{1}{(i\ks)^2}\langle \d_\xi\tilde q_j(0),L_0\d_\xi q_{n+1}(0)\rangle
+\frac{1}{i\ks}\langle\d_\xi\tilde q_j(0),L^{(1)}q_{n+1}(0)\rangle\\
&+\frac{1}{i\ks}\langle\tilde q_j(0),L^{(1)}\d_\xi q_{n+1}(0)\rangle+\langle\tilde q_j(0),L^{(2)}q_{n+1}(0)\rangle\\ * & * \ep=\bp *&-(\d_k F)|_{(\Ms,\ks)}\\ * & * \ep.\\
\end{aligned}
$$
%\ea
Therefore, there is no loss of regularity in $\xi$ under the scaled similarity transformation
\be\label{sim}
\tilde\Lambda_\xi\ =\ \frac{1}{i\ks\xi}\bp \Id_{n\times n}&0_{n\times 1}\\0_{1\times n}& i\ks\xi\ep\ \Lambda_\xi\
\bp\Id_{n\times n}& 0_{n\times 1}\\ 0_{1\times n} &(i\ks\xi)^{-1}\ep
\ee
and
$$
\tilde\Lambda_0\ =\ \bp-\left((\d_MF)|_{(\Ms,\ks)}-\cs\Id\right)&-(\d_k F)|_{(\Ms,\ks)}\\
-\ks(\d_Mc)|_{(\Ms,\ks)}&-\ks(\d_kc)|_{(\Ms,\ks)}
\ep.
$$
This gives the result in a straightforward way. For omitted details, we point to similar computations in \cite{NR1,NR2,JZB}. Recall also that this result is proved 
at spectral level (i.e., not including matrix expansion or eigenvector 
information)
by an Evans function approach in \cite{Se,OZ3}.
\end{proof}

The former proof provides more than stated in Proposition \ref{whiteig}. We collect supplementary information in the following lemma.

\bpr[\cite{NR2}]\label{linspecg}
Assuming (H1)--(H3) and (D1)--(D3), there exist $\eps_0>0$, $\xi_0\in(0,\pi)$, $n+1$ analytic curves, $j=1,\dots,n+1$, $\lambda_j:[-\xi_0,\xi_0]\to B(0,\eps_0)$ such that for $\xi\in[-\xi_0,\xi_0]$
$$
\sigma(L_\xi)\cap B(0,\eps_0)\ =\ \left\{\ \lambda_j(\xi)\ \middle|\ j\in\{1,\dots,n+1\}\ \right\}
$$
and for $\theta>0$ as in (D2)
\be\label{lambdag}
\lambda_j(\xi)= -i\ks\xi a_j \xi+(i\ks\xi)^2b_j+\cO(|\xi|^3),
\qquad
a_j,b_j\quad\textrm{real},\quad \ks^2b_j\geq\theta,
\ee
and associated left and right eigenfunctions $\phi_j(\xi)$ and $\tilde \phi_j(\xi)$ satisfying pairing relations
\be\label{orthogg}
\langle\tilde\phi_j(\xi),\phi_k(\xi)\rangle=i\ks\xi\delta^j_k,\qquad 1\le j,k\le n+1,
\ee
obtained as
\be\label{phibifg}
\begin{array}{rcccl}
\displaystyle
\phi_j(\xi)&=&\displaystyle
(i\ks\xi)\sum_{l=1}^n\beta_l^{(j)}(\xi)q_l(\xi)&+&\displaystyle
\beta_{n+1}^{(j)}(\xi)\quad q_{n+1}(\xi)\\
\displaystyle
\tilde\phi_j(\xi)&=&\displaystyle
\quad\sum_{l=1}^n\tilde\beta_l^{(j)}(\xi)\tilde q_l(\xi)&+&\displaystyle
(i\ks\xi)\tilde\beta_{n+1}^{(j)}(\xi)\quad\tilde q_{n+1}(\xi)\\
\end{array}
\ee
where
\begin{itemize}
\item $(q_1(\xi),\dots,q_{n+1}(\xi))$ and $(\tilde q_1(\xi),\dots,\tilde q_{n+1}(\xi))$ are dual bases of spaces associated to the spectrum of respectively $L_\xi$ and its adjoint $L_\xi^*$ in $B(0,\eps_0)$, analytic in $\xi$, bifurcating from $(\d_{M_1}\bar U,\dots,\d_{M_n}\bar U,\bar U')$ and $(e_1,\dots,e_n,\bar u^{adj})$ at $\xi=0$, with $e_j$ the constant function equal to the $j$th standard Euclidean basis element and $\bar u^{adj}$ a generalized zero eigenfunction of $L_0^*$ such that $\langle\bar u^{adj},\partial_M\bar U\rangle=0$ and $\langle\bar u^{adj},\bar U'\rangle=1$, and chosen such that
\be\label{expand_orthogg}
\langle\tilde q_j(0),\d_\xi q_{n+1}(0)\rangle =0,\qquad \langle\d_\xi\tilde q_j(0),q_{n+1}(0)\rangle =0,\qquad 1\leq j\leq n;
\ee
\item $(\beta^{(1)}(\xi),\dots,\beta^{(n+1)}(\xi))$ and $(\tilde \beta^{(1)}(\xi),\dots,\tilde \beta^{(n+1)}(\xi))$ are dual right and left eigenbases, analytic in $\xi$, of the matrix $\tilde\Lambda_\xi$, defined in \eqref{Lambda_xi}-\eqref{sim}, associated to eigenvalues $\lambda_j(\xi)/(i\ks\xi)$ and in particular for $\xi=0$ they form dual right and left eigenbases associated to $a_j$s of $d(F,-\omega)_{|(\Ms,\ks)}-\cs\Id$.
\end{itemize}
\epr

\begin{proof}
This is a direct consequence of the proof of Proposition \ref{whiteig} except for the conditions on $\lambda_j(\xi)$, $a_j$, $b_j$, 
which follow from (D2) and complex conjugate symmetry together with (H3).
\end{proof}

\br\label{scalingrmk}
\textup{
Scaling transform \eqref{sim} is directly related to the fact that the equation for the time evolution of local wavenumber is obtained by differentiating once in space the equation for the local phase. Though unnecessary to prove Proposition \ref{whiteig} in the uncoupled case, since the proposition follows then already from an examination of $\Lambda^{(1)}$, the manipulation is correct regardless of the linear coupling assumption. Yet in the uncoupled case, we may also assume $\tilde \beta_j^{(n+1)}(0)=0$ and $\beta_{n+1}^j(0)=0$ for $j\neq n+1$. Then by replacing $\tilde \phi_{n+1}(\xi)$ with $(i\ks\xi)^{-1}\tilde \phi_{n+1}(\xi)$ and, for $j\neq n+1$, $\phi_j(\xi)$ with $(i\ks\xi)^{-1}\phi_j(\xi)$, we obtain dual critical bases, analytic in $\xi$. For localized data this difference translates at the linear level into different decay rates. Note also that, even in the linearly uncoupled case, manipulations similar to scaling must be performed to get information about eigenf
 unctions.
}
\er

Next, following \cite{JZ3,JNRZ1}, in view of \eqref{phibifg}(i) with $q_{n+1}(0)=\bar U'$, we decompose the solution operator $S(t)=e^{tL}$ as
\be\label{decompg}
S(t)=S^{\rm p}(t)+\tilde S(t),\qquad
S^{\rm p}(t)=\bar U'\ e_{n+1}\cdot s^{\rm p}(t),\qquad
s^{\rm p}(t)=\sum_{j=1}^{n+1}s^{\rm p}_j(t)
\ee
with
\ba\label{spg}
(s^{\rm p}_j(t)g)(x)
=\int_{-\pi}^{\pi}e^{i\xi x}\alpha(\xi)e^{\lambda_{j}(\xi)t}\frac{1}{i\ks\xi}\beta^{(j)}(\xi)
\langle\tilde\phi_{j}(\xi,\cdot), \check g(\xi,\cdot)\rangle_{L^2([0,1])}d\xi,
\ea
and
\ba\label{tildeSg}
(\tilde S(t) g)(x)&:=
\int_{-\pi}^{\pi} e^{i\xi x}(1-\alpha(\xi))(e^{tL_\xi}  \check g(\xi))(x)d\xi
+\int_{-\pi}^{\pi} e^{i\xi x}\alpha(\xi)(e^{tL_\xi} \tilde \Pi(\xi)\check g(\xi))(x)d\xi\\
&
+\int_{-\pi}^{\pi}e^{i\xi x}\alpha(\xi)\sum_{j=1}^{n+1} e^{\lambda_{j}(\xi) t}
\frac{(\phi_{j}(\xi,x)-e_{n+1}\cdot\beta^{(j)}(\xi)\,q_{n+1}(0,x))}{i\ks\xi}
\langle\tilde\phi_{j}(\xi),\check g(\xi)\rangle_{L^2([0,1])} d\xi,\\
\ea
where $\alpha$ is a smooth cutoff function such that $0\leq\alpha\leq1$,
$\alpha(\xi)=1$ for $|\xi|\le \xi_0/2$ and $\alpha(\xi)=0$ for $|\xi|\ge \xi_0$, and
\be\label{Pig}
\Pi^{\rm p}(\xi):=\sum_{j=1}^{n+1}q_j(\xi)\langle\tilde q_j(\xi),\cdot\rangle_{L^2([0,1])}=\frac{1}{i\ks\xi}\sum_{j=1}^{n+1}\phi_j(\xi)\langle\tilde \phi_j(\xi),\cdot\rangle_{L^2([0,1])},\quad
\tilde \Pi(\xi):=\Id-\Pi^{\rm p}(\xi)
\ee
denote respectively the eigenprojection, defined for $|\xi|\leq\xi_0$, onto the critical space 
$$
\Sigma_\xi=\Span\{\phi_j(\xi)\}_{j=1}^{n+1}
$$
bifurcating from $\Sigma_0$ at $\xi=0$, and its complementary projection.

To establish nonlinear stability, there are of course other natural splitting choices available.
In particular, the fact that we have kept in \eqref{spg} the full $\beta^{(j)}(\xi)$
instead of $\beta^{(j)}(0)$ alone will play a role only in the asymptotic behavior study.

Next, we transpose this linear decomposition to the nonlinear level.  Recalling Lemma \ref{lem:canest} and integral equation \eqref{tied},
we start with the equation
$$
(\partial_t-L)(v+\psi \bar U')=\cN,
\qquad
v|_{t=0}=d_0,\ \psi|_{t=0}=h_0,
$$
for the nonlinear residual $v(x,t)$ and the phase shift $\psi(x,t)$ both introduced in \eqref{pertvar}, 
where $ d_0:=\tilde u_0(\cdot-h_0(\cdot))-\bar U\in L^1(\RM)\cap H^K(\RM)$,
$\partial_x h_0\in L^1(\RM)\cap H^K(\RM)$,
%new from MJ
and notice, as in \eqref{tied}, that, after denoting solution operator $S(t):=e^{tL}$, an application of Duhamel's formula 
leads to
%CHANGED-MR: incorrect
%\begin{equation}\label{vimplicit}
%v(x,t)=\bar h_0 U' + S(t)d_0+\int_0^tS(t-s)\N(s)ds.
%\end{equation}
\be\label{vimplicit}
v(t)+\psi(t)\bar U'\ =\ S(t)(d_0+h_0\bar U')+\int_0^tS(t-s)\N(s)ds.
\ee
In order to simultaneously accommodate the initial datum constraint $\psi(0)=h_0$ and 
absorb as much as possible $e_{n+1}\cdot s^{\rm p}(t)$ contributions into 
the equation for $\psi$, as described in the discussion surrounding \eqref{1psi}--\eqref{1v},
we split 
\eqref{vimplicit} as 
\ba\label{psidefg}
\psi(t)&=
e_{n+1}\cdot s^{\rm p}(t) (h_0\bar U'+d_0)+
\int_0^t e_{n+1}\cdot s^{\rm p}(t-s)\N(s)ds\\
&-(1-\chi(t))\left(e_{n+1}\cdot s^{\rm p}(t)(d_0+h_0\bar U')-h_0+\int_0^t e_{n+1}\cdot s^{\rm p}(t-s)\N(s)ds\right),
\ea
and
\ba\label{closedg}
v(t)&=\tilde S(t) (d_0+h_0\bar U') + \int_0^t \tilde S(t-s) \N(s)ds\\
&\quad
+(1-\chi(t))\left(S^{\rm p}(t)d_0+(S^{\rm p}(t)-\Id)(h_0\bar U')+\int_0^t S^{\rm p}(t-s)\N(s)ds\right),
\ea
where $\chi(t)$ is a smooth cutoff that is zero for $t\le 1/2$ and one for $t\ge 1$. We may extract from \eqref{psidefg}-\eqref{closedg} a closed system in $(v,\psi_x,\psi_t)$ (and some of their derivatives), and then recover $\psi$ through the slaved equation \eqref{psidefg}.

\medskip

We proceed stating now the linear estimates needed to bound the terms appearing in \eqref{psidefg}-\eqref{closedg}.

\subsection{Basic linear estimates}\label{s:basicg}

\bpr\label{greenbdsg}
Under assumptions (H1)-(H3) and (D1)-(D3),
for all $t\geq0$, $2\leq p\leq \infty$,
and any $l,m\ge 0$, $r\ge 1$, $1\leq j \leq n+1$,
\begin{align}\label{finaleundiffg}
\left\|
\partial_x^l\partial_t^m s^{\rm p}_j(t)  g \right\|_{L^p(\RM)}
\lesssim
\min \begin{cases}
(1+t)^{-\frac{1}{2}(1-1/p)+\frac12-\frac{l+m}{2}}\|g\|_{L^1(\RM)}\\
(1+t)^{-\frac{1}{2}(1/2-1/p)+\frac12-\frac{l+m}{2}}\|g\|_{L^2(\RM)}
\end{cases}
\end{align}
when $l+m\geq1$,
\be\label{finaleundiffcg}
\|\,s^{\rm p}_j(t)  g\,\|_{L^\infty(\RM)}
\ \lesssim\ \|g\|_{L^1(\RM)\cap L^2(\RM)}
\ee
when ($l=m=0$ and $p=\infty$), for some $\eta>0$
\begin{align}\label{finaleundifflayerg}
\|\,e_{n+1}\cdot s^{\rm p}(t)  g\,\|_{L^p(\RM)}
\ \lesssim\
\min \begin{cases}
\quad (1+t)\quad\|g\|_{L^1(\RM)}\\
\quad e^{-\eta t}\ \|g\|_{L^2(\RM)}\ +\ t^{\frac1p}\ \|g\|_{L^1(\RM)}\\
\end{cases}
\end{align}
when ($l=m=0$ and $2\leq p\leq\infty$), and
\begin{align}\label{finaleg}
\left\|
\partial_x^l\partial_t^m s^{\rm p}_j(t) \partial_x^r g \right\|_{L^p(\RM)}
\lesssim
\min \begin{cases}
(1+t)^{-\frac{1}{2}(1-1/p)-\frac{l+m}{2}}\|g\|_{L^1(\RM)}\\
(1+t)^{-\frac{1}{2}(1/2-1/p)-\frac{l+m}{2}}\|g\|_{L^2(\RM)}
\end{cases},
\end{align}
while, for some $\eta>0$, $0\leq l+2m\leq K+1$, and $2\le p\le\infty$,
\begin{align}\label{undiffeqg}
\left\|\partial_x^l \partial_t^m \tilde S(t)g\right\|_{L^p(\RM)}
&\lesssim
\min
\begin{cases}
e^{-\eta t}\|\partial_x^r g\|_{H^{l+2m+1}(\RM)}+(1+t)^{-\frac{1}{2}(1-1/p)}\| g\|_{L^1(\RM)}\\
e^{-\eta t}\|\partial_x^r g\|_{H^{l+2m+1}(\RM)}+(1+t)^{-\frac{1}{2}(1/2-1/p)}\|g\|_{L^2(\RM)}\\
\end{cases},
\end{align}
and, for $1\le r\le K+1$,
\begin{align}\label{finalgg}
\left\|\partial_x^l \partial_t^m \tilde S(t) \partial_x^r g \right\|_{L^p(\RM)}
&\lesssim
\min
\begin{cases}
e^{-\eta t}\|\partial_x^r g\|_{H^{l+2m+1}(\RM)}+(1+t)^{-\frac{1}{2}(1-1/p)-\frac{1}{2}}\|g\|_{L^1(\RM)}\\
e^{-\eta t}\|\partial_x^r g\|_{H^{l+2m+1}(\RM)}+(1+t)^{-\frac{1}{2}(1/2-1/p)-\frac{1}{2}}\|g\|_{L^2(\RM)}\\
\end{cases}.
\end{align}
\epr

\begin{proof}
{\it (i) (Proof of \eqref{finaleundiffg}).}
First, notice that
\be\label{Shf}
(\d_x^l\d_t^m s_j^{\rm p}(t) g)(x)=
\int_{-\pi}^{\pi} \alpha(\xi)e^{\lambda_j(\xi)t}e^{i\xi x}\frac{1}{i\ks\xi}(i\xi)^l\lambda_j(\xi)^m
\beta^{(j)}(\xi)\langle\tilde\phi_j(\xi),\check g(\xi)\rangle_{L^2([0,1])}d\xi.
\ee
%CHANGED-MJ wording
%In the case $l+m\geq 1$ estimates on $s^{\rm p}$ follows from, for either $s=1$ or $s=2$, introducing $s'$
%such that $1/s+1/s'=1$,
In the case $l+m\geq 1$, estimates on $s^{\rm p}$ follows from, choosing either $s=1$ or $s=2$ fixed and
introducing $s'$ such that $1/s+1/s'=1$, the generalized Hausdorff-Young inequality \eqref{hy} by
%
%CHANGED-MJ changed r_{s,p} to r(s,p) since subscripts within the superscript were hard to see...
\begin{equation}\label{spbd1}
\begin{aligned}
\displaystyle
\Big\|\ x&\displaystyle
\mapsto\int_{-\pi}^{\pi}
e^{i\xi x}\alpha(\xi)e^{\lambda_j(\xi)t}\frac{1}{i\ks\xi}(i\xi)^l\lambda_j(\xi)^m\beta^{(j)}(\xi)
\langle\tilde \phi_j(\xi),\check g(\xi)\rangle_{L^2([0,1])}d\xi\ \Big\|_{L^p(\RM)}\\
&\displaystyle\lesssim
\|\ (\xi,x)\mapsto\alpha(\xi) e^{\lambda(\xi) t}|\xi|^{l+m-1}
|\langle\tilde \phi_j(\xi),\check g(\xi)\rangle_{L^2([0,1])}|\ \|_{L^q([-\pi,\pi],L^p([0,1]))}\\
&\displaystyle\lesssim
%\|\xi\mapsto |\xi|^{l+m-1} e^{-\eta \xi^2 t}\|_{L^{r_{s,p}}([-\pi,\pi])}\ \|\xi\mapsto\alpha(\xi)^{1/2}|\langle \tilde \phi_j(\xi),\check g(\xi) %\rangle_{L^2([0,1])}|\|_{L^{s'}([-\pi,\pi])}\\
\|\xi\mapsto |\xi|^{l+m-1} e^{-\eta \xi^2 t}\|_{L^{r(s,p)}([-\pi,\pi])}\ \|\xi\mapsto\alpha(\xi)^{1/2}|\langle \tilde \phi_j(\xi),\check g(\xi) \rangle_{L^2([0,1])}|\|_{L^{s'}([-\pi,\pi])}\\
&\displaystyle\lesssim
(1+t)^{-\frac{1}{2}(1/s-1/p)-\frac{l+m-1}{2}}\|\xi\mapsto\alpha(\xi)^{1/2}|\langle \tilde \phi_j(\xi),\check g(\xi) \rangle_{L^2([0,1])}|\|_{L^{s'}([-\pi,\pi])},
\end{aligned}
\ee
where $1/p+1/q=1$ and
%$1/s'+1/r_{s,p}=1/q$, so that  $1/r_{s,p}=1/s-1/p$. Here
$1/s'+1/r(s,p)=1/q$, so that  $1/r(s,p)=1/s-1/p$. Here
%ENDCHANGED-MJ
we have used (D2) to get for some $\eta>0$, $|e^{\lambda_j(\xi)t}\alpha^{1/2}(\xi)|
\le e^{-\eta \xi^2 t}$ and \eqref{as} to get $\lambda_j(\xi)=\cO(\xi)$.

%CHANGED-MJ
%Now for $s=2$, we combine this with Cauchy-Schwarz inequality and Parseval identity that yield
Now, for $s=2$, we note by the Cauchy-Schwarz inequality and Parseval identity \eqref{psvl} that
\[
\begin{aligned}
\Big\|\xi\mapsto\alpha(\xi)^{1/2}|\langle \tilde \phi_j(\xi),\check g(\xi)\rangle_{L^2([0,1])}\Big\|_{L^{2}([-\pi,\pi])}
&\leq \sup_{|\xi|\leq\xi_0}\|\tilde \phi_j(\xi)\|_{L^2([0,1])}\ \|\check g\|_{L^2([-\pi,\pi],L^2([0,1])}\\
&\lesssim \sup_{|\xi|\leq\xi_0}\|\tilde \phi_j(\xi)\|_{L^2([0,1])}\ \|g\|_{L^2(\RM)},
\end{aligned}
\]
where $\xi_0$ is given as in Proposition \ref{linspecg}.
%CHANGED-MJ wording, break apart the long paragraph
%while for $s=1$ we write
For $s=1$ on the other hand, we begin by expanding
$$
\langle \tilde \phi_j(\xi),\check g(\xi)\rangle_{L^2([0,1])}=\sum_{j'\in\ZM}\hat{\tilde \phi}_j(\xi,j')^* \widehat g(\xi+2j'\pi),
$$
where $\widehat{\tilde \phi}_j(\xi,j')$ denotes the $j'$th Fourier coefficient in the
Fourier expansion of $2\pi$-periodic function $\tilde \phi_j(\cdot)$, and $z^*=\bar z$ denotes complex conjugate.
Applying the standard Hausdorff-Young inequality for the Fourier transform, i.e. $\|\hat g\|_{L^\infty(\RM)}\le \|g\|_{L^1(\RM)}$,
together with the estimate
\[
\alpha^{1/2}(\xi)\sum_{j'}|\widehat{\tilde \phi}_j(\xi,j')^*|\le
\alpha^{1/2}(\xi)\sqrt{\sum_{j'} (1+|j'|^2)|\widehat{\tilde \phi}_j(\xi,j')|^2\sum_{j'} (1+|j'|^{-2})}
\le C\alpha^{1/2}(\xi) \|\tilde \phi_j(\xi)\|_{H^1([0,1])},
\]
which readily follows from the Cauchy--Schwarz' inequality, we obtain the bound
$$
\begin{aligned}
\|\xi\mapsto\alpha(\xi)^{1/2}|\langle \tilde \phi_j(\xi),\check g(\xi)\rangle_{L^2([0,1])}|\|_{L^{\infty}([-\pi,\pi])}
&\lesssim \sup_{|\xi|\leq\xi_0}\|\tilde \phi_j(\xi)\|_{H^1([0,1])}\ \|g\|_{L^1(\RM)}.
\end{aligned}
$$
Together with \eqref{spbd1}, this establishes \eqref{finaleundiffg}.

\medskip

{\it (ii) (Proof of \eqref{finaleundiffcg}).}
Thanks to \eqref{lambdag}, with the same kind of estimates as above
one can bound in $L^\infty(\RM)$ the difference between $s^{\rm p}_j(t)(g)$ and
$$
x\longmapsto\int_{-\pi}^{\pi} e^{(-i\ks a_j\xi-\ks^2 b_j \xi^2)t}e^{i\xi x}\frac{1}{i\ks\xi}\beta^{(j)}(0)\langle\tilde\phi_j(0),\check g(\xi)\rangle_{L^2([0,1])}d\xi
$$
by $C(1+t)^{-\frac12}\|g\|_{L^1(\RM)}$. Since $\tilde \phi_j(0)$ is constant equal to $\tilde\nu_j:=(\tilde \beta^{(j)}_1(0),\dots,\tilde \beta^{(j)}_n(0))$,
% and, for any $y\in\RM$,
%$$
%(\delta_y)\,\check{}\ (\xi,x)\ =\ \frac{1}{2\pi}\sum_{j'\in\ZM} e^{i2\pi j'x} e^{-iy(2\pi j'+\xi)}\ =\ e^{-i%\xi y}[\delta_y](x),
%$$
%$[\delta_y]$ being the 1-periodic extension of $\delta_y$,
$$
\langle\tilde\phi_j(0),\check g(\xi)\rangle_{L^2([0,1])}=\tilde\nu_j\cdot \widehat{g}(\xi)
$$
and
the last term is recognized to be the convolution of $\tilde\nu_j\cdot g$ with
$$
\begin{aligned}
x\ \mapsto\ \textrm{p.v.} \int_{-\pi}^{\pi} e^{(-i\ks a_j\xi-\ks^2 b_j \xi^2)t}e^{i\xi x}\frac{1}{i\ks\xi}\beta^{(j)}(0) d\xi
&=\textrm{p.v.}\int_\RM e^{(-i\ks a_j\xi-\ks^2 b_j \xi^2)t}e^{i\xi x}\frac{1}{i\ks\xi}\beta^{(j)}(0) d\xi\\
&-\int_{\RM\setminus[-\pi,\pi]} e^{(-i\ks a_j\xi-\ks^2 b_j \xi^2)t}e^{i\xi x}\frac{1}{i\ks\xi}\beta^{(j)}(0) d\xi.
\end{aligned}
$$
By using the Cauchy-Schwarz inequality, we may bound the last integral in $L^2(\RM)$ with $Ce^{-\eta t}$, for some $\eta>0$. The remaining principal value is explicitly computed as a Gaussian error function
$$
\frac{2\pi}{\ks}\beta^{(j)}(0)\ \errfn\left(\frac{x+a_j\ks t}{\sqrt{4\ks^2b_jt}}\right)
$$
and thus is bounded in $L^\infty(\RM)$.
%CHANGED-MR: removed, too much details compared with the rest and badly placed
%%CHANGED-MJ added
%Using Young's inequality $\|h_1*h_2\|_{L^\infty(\RM)}\leq\|h_1\|_{L^2(\RM)}\|h_2\|_{L^2(\RM)}$, this achieves the proof of \eqref{finaleundiffcg}.
This achieves the proof of \eqref{finaleundiffcg}.
%%
%ENDCHANGED

\medskip

{\it (iii) (Proof of \eqref{finaleundifflayerg}).}
We first remark that one can bound in $L^p(\RM)$ the difference between $e_{n+1}\cdot s^{\rm p}(t)(g)$ and
$$
x\longmapsto\int_{-\pi}^{\pi}\quad
e^{i\xi x}\sum_{j=1}^{n+1} e^{(-i\ks a_j\xi-\ks^2 b_j \xi^2)t}\frac{1}{i\ks\xi}e_{n+1}\cdot\beta^{(j)}(0)\ \tilde\nu_j\cdot \widehat{g}(\xi)d\xi
$$
by $C(1+t)^{-\frac12(1-1/p)}\|g\|_{L^1(\RM)}$. Since $\{\beta_j(0)\}_j$ and $\{\tilde\beta_j(0)\}_j$ are dual bases, 
it follows, with $\tilde\nu_j$ still denoting the constant value of $\tilde \phi_j(0)$, that 
$$
\sum_{j=1}^{n+1}\beta^{(j)}(0)\ \tilde\nu_j^T\ =\ \sum_{j=1}^{n+1}\beta^{(j)}(0)\tilde\beta^{(j)}(0)^T\ \begin{pmatrix}\Id_{d\times d}\\0\cdots0\end{pmatrix}
\ =\ \begin{pmatrix}\Id_{d\times d}\\0\cdots0\end{pmatrix},
$$
and therefore
$$
e_{n+1}\cdot\sum_{j=1}^{n+1}\beta^{(j)}(0)\ \tilde\nu_j^T\ =\ \begin{pmatrix}0&\cdots&0\end{pmatrix}.
$$
Thus, the quantity to bound may be written as the convolution of $g$ with the kernel
\be\label{kerlab}
x\longmapsto\int_{-\pi}^{\pi}\quad
e^{i\xi x}\sum_{j=1}^{n} \dfrac{e^{(-i\ks a_j\xi-\ks^2 b_j \xi^2)t}-e^{(-i\ks a_{j+1}\xi-\ks^2 b_{j+1}\xi^2)t}}{i\ks\xi}\
\left(\sum_{j'=1}^j e_{n+1}\cdot\beta^{(j')}(0)\ \tilde\nu_{j'}^T\right)\ d\xi.
\ee
Observing that
$$
\left|\dfrac{e^{(-i\ks a_j\xi-\ks^2 b_j \xi^2)t}-e^{(-i\ks a_{j+1}\xi-\ks^2 b_{j+1}\xi^2)t}}{i\ks\xi}\right|
\ \lesssim\ t
$$
provides \eqref{finaleundifflayerg}(i).

To get \eqref{finaleundifflayerg}(ii), we first write the kernel \eqref{kerlab} above as
$$
\begin{aligned}
x\longmapsto\int_{-\pi}^{\pi}\int_{-1}^1&\quad
\sum_{j=1}^{n}e^{i\xi \left(x+\ks\left(\frac{a_j+a_{j+1}}{2}+\tau\frac{a_j-a_{j+1}}{2}\right)t\right)}\
t\left[\frac{a_j-a_{j+1}}{2}+i\ks\xi(b_j-b_{j+1})\tau\right]\\
&\times e^{-\left(\frac{b_j+b_{j+1}}{2}+\tau^2\frac{b_j-b_{j+1}}{2}\right)\ks^2\xi^2t}\
\left(\sum_{j'=1}^j e_{n+1}\cdot\beta^{(j')}(0)\ \tilde\nu_{j'}^T\right)\ d\tau d\xi.
\end{aligned}
$$
Integration by parts in the $\xi$ variable shows that for any fixed $\eps>0$ the contribution to the $L^p(\RM)$ norm of the kernel of points $x$ lying outside
$$
I^\eps_t\ :=\ \left[\ t\ \left(\min_{1\leq j\leq n+1}a_j-\eps\right)\ ,\ t\ \left(\max_{1\leq j\leq n+1}a_j+\eps\right)\ \right]
$$
is bounded by
$$
C_\eps\ t\ \left(\int_{\eps t}^\infty \frac{dy}{y^p} \right)^{1/p}\ \leq\ C_\eps t^{1/p}.
$$
Following the end of the proof of \eqref{finaleundiffcg}, we may write the remaining part of the convolution as the convolution of $g$ with a term decaying exponentially in time in $L^2(\RM)$ plus its convolution with a bounded function supported on $I^\eps_t$. This yields \eqref{finaleundifflayerg}(ii).

\medskip

{\it (iv) (Proof of \eqref{finaleg}).}
The proof goes similarly
as the one of \eqref{finaleundiffg} thanks to the fact that since $\tilde \phi_j(0)$ is constant, for $r\geq1$, the equality
$$
\begin{aligned}
\langle \tilde \phi_j(\xi)&,(\d_x^rg)\,\check{}\ (\xi)\rangle_{L^2([0,1])}
=\sum_{r'=0}^r\binom{r}{r'}(i\xi)^{r'}\langle \tilde \phi_j(\xi),\d_x^{r-r'}\check g(\xi)\rangle_{L^2([0,1])}\\
&=\xi(-1)^r\left\langle \d_x^{r}\left(\frac{\tilde \phi_j(\xi)-\tilde \phi_j(0)}{\xi}\right),\check g(\xi)\right\rangle
+i\xi\sum_{r'=1}^r\binom{r}{r'}(i\xi)^{r'-1}(-1)^{r-r'}\langle \d_x^{r-r'}\tilde \phi_j(\xi),\check g(\xi)\rangle
\end{aligned}
$$
holds and provides the factor $\xi$ needed to compensate for $(i\ks\xi)^{-1}$.

\medskip

{\it (v) (Proof of \eqref{undiffeqg} and \eqref{finalgg}).}
The last part of \eqref{tildeSg} is bounded as $s_j^{\rm p}$, with an extra $\xi$ factor compensating for $(i\ks\xi)^{-1}$ thus enhancing decay and allowing for $l=m=0$ and $2\leq p\leq\infty$. We focus on the two remaining terms of \eqref{tildeSg}.

To treat these, it is convenient to introduce on $H^l([0,1])$ a family of equivalent norms parametrized by $\xi\in[-\pi,\pi]$,
$$
\|g\|_{\dot{H}^l_\xi([0,1])}^2:=\sum_{j=0}^{l}\|(\partial_x +i\xi)^jg\|_{L^2([0,1])}^2
$$
so that Parseval's identity implies
$$
\|g\|_{H^{l}(\RM)}^2\ =\ (2\pi)\
\|\xi\mapsto\ \|\check g(\xi)\|_{\dot{H}^{l}_\xi([0,1]))}\ \|_{L^2([-\pi,\pi])}^2.
$$
Now, thanks to standard resolvent bounds \cite{He}, assumptions (D1)-(D3) and the fact that $H^{l+2m+1}([0,1])$ and $L^2([0,1])$ spectra coincide \cite{G}, we may use Pr\"uss' Theorem \cite{Pr} and obtain that for some $\eta>0$
$$
|e^{tL_\xi}(1-\alpha(\xi))|_{\dot{H}_\xi^{l+2m+1}([0,1])\to \dot{H}_\xi^{l+2m+1}([0,1])},\quad
|\alpha(\xi) e^{tL_\xi}\ \tilde \Pi(\xi)|_{\dot{H}_\xi^{l+2m+1}([0,1])\to \dot{H}_\xi^{l+2m+1}([0,1])}
\lesssim e^{-\eta t}.
$$
Therefore
$$
\begin{aligned}
\Big\|x\mapsto&\int_{-\pi}^{\pi} e^{i\xi x}(1-\alpha(\xi))(L_\xi^m e^{tL_\xi}\check g(\xi))(x)d\xi
\Big\|_{H^{l+1}(\RM)},\\
\Big\|x\mapsto&
\int_{-\pi}^{\pi} e^{i\xi x}\alpha(\xi)(L_\xi^m e^{tL_\xi} \tilde \Pi(\xi)\check g(\xi))(x)d\xi
\Big\|_{H^{l+1}(\RM)}\\
&\lesssim e^{-\eta t}\|\xi\mapsto\ \|\check g(\xi)\|_{\dot{H}^{l+2m+1}_\xi([0,1]))}\ \|_{L^2([-\pi,\pi])}
\lesssim e^{-\eta t}\|g\|_{H^{l+2m+1}(\RM)}.
\end{aligned}
$$
Since, for $2\le p\le \infty$, $H^{l+1}(\RM)$ is embedded in $W^{l,p}(\RM)$, this completes the proof of the proposition.
\end{proof}

\br\label{layer_behavior}
\textup{
We have included \eqref{finaleundifflayerg}(ii) to give a better account of large-time behavior of the phase. Yet, in our nonlinear stability analysis we 
use \eqref{finaleundifflayerg}
only
 to ensure that the 
initial
time layer remains localized, thus only for intermediate times $1/2\leq t\leq 1$. For this purpose, \eqref{finaleundifflayerg}(i) is sufficient. The proof of the latter estimate is easier to obtain and involves only frequency arguments in the spirit of the rest of the paper.}
\er

\subsection{Linear modulation bounds}\label{s:modg}
The above basic linear estimates are sufficient to control all terms in \eqref{psidefg} and \eqref{closedg}
in the localized case when $h_0\equiv 0$.  To control the additional terms arising from the nonlocalized
initial phase shift $h_0\bar{U}'$, we use the bounds in the following proposition.
Here and elsewhere, we suppress the dependence of bounds on
norms of the background periodic wave $\bar U$, the
periodic right and left eigenbases $\phi_j$ and $\tilde \phi_j$,
or other known periodic functions, which, by (H1), may be seen to be as smooth as needed for the $H^s([0,1])_{\rm per}$ bounds our arguments require.

%CHANGED-DL
For functions $h$ with localized derivative, we will use repeatedly $\hat{h}(\xi)=\frac{1}{i\xi}\widehat{\d_x h}(\xi)$. From now on, to make this possible, we will assume that all such functions are centered, in the sense that $h(-\infty)=-h(\infty)$ (which includes the case where $h$ is itself localized and both terms vanish). In the statements of Theorems \ref{oldmain} and \ref{main}, it corresponds to the assumption that the global phase shift $\psi_\infty$ is $0$. Of course, in doing so, one does not lose in generality since this is achieved by replacing $\bar U$, $h_0$, $\psi$, \emph{etc.} with $\bar U(\,\cdot\,-\psi_\infty)$, $h_0-\psi_\infty$, $\psi-\psi_\infty$, \emph{etc.}

\bpr\label{modpropg}
Under (H1)--(H3) and (D1)--(D3), for all $t\geq0$, $2\leq p\leq \infty$, $l,m\geq0$,\\ $1\leq j \leq n+1$,
\be\label{Spmodg}
\|\d_x^l\d_t^m s^{\rm p}_j(t)(h_0\bar U')\|_{L^p (\RM)}
\lesssim
(1+t)^{-\frac{1}{2}(1-1/p)+\frac{1}{2}-\frac{l+m}{2}}\|\partial_x h_0\|_{L^1 (\RM)}
\ee
when $l+m\ge 1$,
\be\label{Spmodcg}
\|\,s^{\rm p}_j(t)(h_0\bar U')\,\|_{L^\infty(\RM)}
\ \lesssim\
\|\partial_x h_0\|_{L^1 (\RM)\cap L^2(\RM)}
\ee
when ($l=m=0$ and $p=\infty$),
\be\label{Spmodlayerg}
\|\,e_{n+1}\cdot s^{\rm p}(t)(h_0\bar U')-h_0\,\|_{L^p(\RM)}
\ \lesssim\ (1+t^{\frac1p})\quad\|\d_x h_0\|_{L^1(\RM)\cap L^2(\RM)}
\ee
when ($l=m=0$ and $2\leq p\leq\infty$), and, for $0\leq l+2m \leq K+1$,
\be\label{tildeSmodg}
\|\d_x^l \d_t^m \tilde S (t)(h_0\bar U')\|_{L^p(\RM)} \lesssim (1+t)^{-\frac{1}{2}(1-1/p)}
\|\d_x h_0\|_{L^1(\RM)\cap H^{l+2m+1}(\RM)}.
\ee
\epr

\begin{proof}\footnote{
Compare to the similar but much simpler argument of \cite[Proposition 4.1]{JNRZ1}, in the reaction-diffusion case.}
{\it (i) (Proof of \eqref{Spmodg}).}
This follows applying the arguments of the proof of \eqref{finaleundiffg} in Proposition \ref{greenbdsg}, once we have established
the estimate
$$
\begin{aligned}
\|\xi\mapsto\alpha(\xi)^{1/2}&|\langle \tilde \phi_j(\xi),(h_0 \bar U')\ \check{}\ (\xi)\rangle_{L^2([0,1])}|\|_{L^{\infty}([-\pi,\pi])}\\
&\lesssim \left(\sup_{|\xi|\leq\xi_0}\|\d_\xi\tilde \phi_j(\xi)\bar U'\|_{L^1([0,1])}+
\sup_{|\xi|\leq\xi_0}\|\tilde \phi_j(\xi)\bar U'\|_{L^2([0,1])}\right)\ \|\d_x h_0\|_{L^1(\RM)}.
\end{aligned}
$$
The latter bound stems from first re-expressing
$$
\begin{aligned}
\alpha(\xi)^{1/2}\langle\tilde\phi_j(\xi),(h_0 \bar U')\ \check{}\ (\xi)\rangle_{L^2([0,1])}
&=\ -i\,\alpha(\xi)^{1/2}
\left\langle\frac{\tilde\phi_j(\xi)-\tilde\phi_j(0)}{\xi},\bar U'\right\rangle_{L^2([0,1])}\ \widehat{\partial_x h_0}(\xi)\\
&\quad+\sum_{j'\neq 0}\alpha(\xi)^{1/2}
\frac{(\widehat{\tilde \phi_j\bar U'})(\xi,j')^*}{i(\xi+2\pi j')}
\ \widehat {\partial_x h_0}(\xi+2j'\pi),
\end{aligned}
$$
where $(\widehat{\tilde\phi_j\bar U'})(\xi,j')$ denotes the $j'$th Fourier coefficient in the Fourier expansion of periodic function $\tilde\phi_j(\xi)\bar U'$, then applying a Haussdorff-Young estimate, the Mean Value Theorem, Cauchy-Schwarz' inequality, and Parseval's identity.

\medskip

{\it (ii) (Proof of \eqref{Spmodcg}).}
Thanks to \eqref{lambdag}, since $\tilde \phi_j(0)$ is constant equal to $\tilde\nu_j$, with the same kind of estimates one can bound in $L^\infty(\RM)$ the difference between $s^{\rm p}_j(t)(h_0\bar U')$ and
$$
x\mapsto\int_{-\pi}^{\pi} e^{(-i\ks a_j\xi-\ks^2 b_j \xi^2)t}e^{i\xi x}\frac{1}{i\ks\xi}\beta^{(j)}(0)
[-i\langle\d_\xi\tilde\phi_j(0),\bar U'\rangle_{L^2([0,1])}\widehat{\d_x h_0}(\xi)+\tilde\nu_j\cdot\langle\bar U',[\check h_0(\xi)-\widehat{h_0}(\xi)]\rangle_{L^2([0,1])}]d\xi
$$
by $C(1+t)^{-\frac12}\|\d_x h_0\|_{L^1(\RM)}$. Now note that, for $\xi\in[-\pi,\pi]$
$$
\begin{aligned}
\langle\bar U',[\check h_0(\xi)-\widehat{h_0}(\xi)]\rangle_{L^2([0,1])}
&=-\langle\bar U,\d_x(\check h_0)(\xi)\rangle_{L^2([0,1])}\\
&=-\langle\bar U,\check{(\d_xh_0)}(\xi)\rangle_{L^2([0,1])}+\langle\bar U\rangle\widehat{\d_xh_0}(\xi)
+i\xi\langle\bar U,[\check h_0(\xi)-\widehat{h_0}(\xi)]\rangle_{L^2([0,1])}\\
&=-\langle(\bar U\d_xh_0)\ \check{}\ (\xi)\rangle+\langle\bar U\rangle\widehat{\d_xh_0}(\xi)
+i\xi\langle\bar U,[\check h_0(\xi)-\widehat{h_0}(\xi)]\rangle_{L^2([0,1])}
\end{aligned}
$$
leading to 
\be\label{key_initial_data}
\langle\bar U'[\check h_0(\xi)-\widehat{h_0}(\xi)]\rangle
=-[(\bar U-\langle\bar U\rangle)\d_xh_0]\ \widehat{}\ (\xi)
+i\xi\langle\bar U,[\check h_0(\xi)-\widehat{h_0}(\xi)]\rangle_{L^2([0,1])}.
\ee
Again the extra $\xi$ factor makes the contribution of the last term negligible so that we are left with proving a $L^\infty(\RM)$ bound on
$$
x\mapsto\int_{-\pi}^{\pi} e^{(-i\ks a_j\xi-\ks^2 b_j \xi^2)t}e^{i\xi x}\frac{1}{i\ks\xi}\beta^{(j)}(0)
[-i\langle\d_\xi\tilde\phi_j(0),\bar U'\rangle_{L^2([0,1])}\d_x h_0-\tilde\nu_j\cdot(\bar U-\langle\bar U\rangle)\d_xh_0]\ \widehat{}\ (\xi)d\xi.
$$
Then the proof of \eqref{Spmodcg} is achieved as was the one of \eqref{finaleundiffcg}, writing the main contribution as a term exponentially-decaying in time plus a convolution with an explicit errorfunction.\footnote{Up to the
explicit computation of the final convolution kernel, the arguments
of the proofs of \eqref{finaleundiffcg}-\eqref{finaleundifflayerg}
and \eqref{Spmodcg}-\eqref{Spmodlayerg} are the ones refined to obtain \eqref{heatwhitham_init}.
}

\medskip

{\it (iii) (Proof of \eqref{Spmodlayerg}).}
The $L^\infty(\RM)$ bound follows from \eqref{Spmodcg} and $\|h_0\|_{L^\infty(\RM)}\leq \|\d_xh_0\|_{L^1(\RM)}$. By interpolation we only need now the $L^2(\RM)$ bound.

Since $\|\ \xi\longmapsto|\xi|^{-1}\ \|_{L^2(\RM\setminus[-\pi,\pi])}\lesssim 1$,
$$
\left\|\ x\longmapsto h_0(x)-\int_{-\pi}^\pi e^{i\xi x}\quad \widehat{h_0}(\xi)d\xi\ \right\|_{L^2(\RM)}\lesssim \|\d_x h_0\|_{L^1(\RM)}.
$$
Combing this with arguments of the proof of \eqref{Spmodcg}, we bound in $L^2(\RM)$ the difference between $e_{n+1}\cdot s^{\rm p}(t)(h_0\bar U')-h_0$ and
$$
\begin{aligned}
x\mapsto&\int_{-\pi}^{\pi}\ e^{i\xi x}\left[\sum_{j=1}^{n+1}e^{(-i\ks a_j\xi-\ks^2 b_j \xi^2)t}\frac{1}{i\ks}e_{n+1}\cdot\beta^{(j)}(0)
\langle\d_\xi\tilde\phi_j(0),\bar U'\rangle_{L^2([0,1])}\widehat{h_0}(\xi)-\widehat{h_0}(\xi)\right]d\xi\\
&-\int_{-\pi}^{\pi}\ e^{i\xi x}\sum_{j=1}^{n+1}e^{(-i\ks a_j\xi-\ks^2 b_j \xi^2)t}\frac{1}{i\ks\xi}e_{n+1}\cdot\beta^{(j)}(0)
\tilde\nu_j\cdot [(\bar U-\langle\bar U\rangle)\d_xh_0]\ \widehat{}\ (\xi)d\xi
\end{aligned}
$$
by $C\|\d_x h_0\|_{L^1(\RM)\cap L^2(\RM)}$. In the sum the latter term is bounded following the proof of \eqref{finaleundifflayerg}(ii). To bound the first term, we first observe that, thanks to \eqref{expand_orthogg},
$$
\langle\d_\xi\tilde\phi_j(0),\bar U'\rangle_{L^2([0,1])}\ =\ \langle\d_\xi\tilde\phi_j(0),q_{n+1}(0)\rangle_{L^2([0,1])}
\ =\ i\ks\tilde\beta^{(j)}_{n+1}(0)\ =\ i\ks\tilde\beta^{(j)}(0)\cdot e_{n+1}
$$
therefore
$$
\frac{1}{i\ks}\sum_{j=1}^{n+1}e_{n+1}\cdot\beta^{(j)}(0)\langle\d_\xi\tilde\phi_j(0),\bar U'\rangle_{L^2([0,1])}
\ =\ e_{n+1}^T\sum_{j=1}^{n+1}\beta^{(j)}(0)\tilde\beta^{(j)}(0)^T e_{n+1}
\ =\ e_{n+1}\cdot e_{n+1}\ =\ 1.
$$
Then the function to bound is written
$$
x\mapsto\int_{-\pi}^{\pi}\ e^{i\xi x}\sum_{j=1}^{n+1}\frac{e^{(-i\ks a_j\xi-\ks^2 b_j \xi^2)t}-1}{i\ks\xi}\beta_{n+1}^{(j)}(0)
\tilde\beta_{n+1}^{(j)}(0)\ \ks\widehat{\d_xh_0}(\xi)d\xi
$$
and is bounded in $L^2(\RM)$ splitting
$$
\frac{e^{(-i\ks a_j\xi-\ks^2 b_j \xi^2)t}-1}{i\ks\xi}
\ =\ \frac{e^{(-i\ks a_j\xi-\ks^2 b_j \xi^2)t}-e^{-\ks^2 b_j \xi^2t}}{i\ks\xi}+\frac{e^{-\ks^2 b_j \xi^2t}-1}{i\ks\xi}
$$
and using both
$
\left|\frac{e^{-\ks^2 b_j \xi^2t}-1}{i\ks\xi}\right|\ \lesssim\ \sqrt{t}
$
and the arguments of the proof of \eqref{Spmodcg}(ii) where $I^\eps_t$ is replaced with
$$
\tilde I^\eps_t\ :=\ \left[\ t\ \left(\min_{1\leq j\leq n+1}(a_j)_{-}-\eps\right)\ ,\ t\ \left(\max_{1\leq j\leq n+1}(a_j)_{+}+\eps\right)\ \right]
$$
(${}_{-}$ and ${}_{+}$ denoting negative and positive parts).

\medskip

{\it (iv) (Proof of \eqref{tildeSmodg}).}
The last part of \eqref{tildeSg} is bounded as was $s_j^{\rm p}$, with the 
usual improvement in integrability and decay coming from the extra $\xi$ factor. 
Once we have proved
$$
\|\xi\mapsto\ \|\bar U'[\check h_0(\xi)-\widehat{h_0}(\xi)]\|_{\dot{H}^{l+2m+1}_\xi([0,1]))}\ \|_{L^2([-\pi,\pi])}
\lesssim \|\d_x h_0\|_{H^{l+2m+1}(\RM)},
$$
the contribution of the high frequencies of $h_0$ to the remaining terms is estimated following the proof of \eqref{undiffeqg}. Since, for $\xi\in[-\pi,\pi]$ 
there holds
$$
\begin{aligned}
\|\bar U'[\check h_0(\xi)-\widehat{h_0}(\xi)]\|_{\dot{H}^{l+2m+1}_\xi([0,1]))}
&\leq\sum_{j'\neq0} \|e^{2\pi j'\,\cdot}\ \bar U'\|_{\dot{H}^{l+2m+1}_\xi([0,1]))}|\widehat{h_0}(\xi+2\pi j')|\\
&\lesssim\big(1+\|\bar U'\|_{H^{l+2m+1}([0,1]))}\big)\sum_{j'\neq0}|\xi+2\pi j'|^{l+2m+1}|\widehat{h_0}(\xi+2\pi j')| \\
&\lesssim \sum_{j'\neq0}\frac{1}{|\xi+2\pi j'|}\big|\widehat{(\d_x^{l+2m+2}h_0)}(\xi+2\pi j')\big|\\
&\lesssim \sqrt{\sum_{j'\neq0}\big|\widehat{(\d_x^{l+2m+2}h_0)}(\xi+2\pi j')\big|^2},
\end{aligned}
$$
by squaring and integrating we obtain the needed bound.

To deal with low-frequency contributions, we use the following refinements:
$$
\begin{aligned}
|e^{tL_\xi}(1-\alpha(\xi))|_{\dot{H}_\xi^{l+2m+1}([0,1])\to \dot{H}_\xi^{l+2m+1}([0,1])}
&\lesssim |\xi|e^{-\eta t}\\
|\alpha(\xi)^{1/2}(\tilde\Pi(\xi)-\tilde\Pi(0))|_{\dot{H}_\xi^{l+2m+1}([0,1])\to \dot{H}_\xi^{l+2m+1}([0,1])}&
\lesssim |\xi|
\end{aligned}
$$
(for some $\eta>0$ and all $\xi\in[-\pi,\pi]$). Then, since $\tilde\Pi(0)\bar U'=0$,
$$
\alpha(\xi)^{1/2}\ \tilde\Pi(\xi)\bar U'\ =\ \alpha(\xi)^{1/2}\ \tilde\Pi(\xi)(\tilde\Pi(\xi)-\tilde\Pi(0))\bar U',
$$
and, by following the proof of \eqref{undiffeqg}, we reduce the bound on the remaining terms to
$$
\|\xi\mapsto\ \|\bar U'\|_{\dot{H}^{l+2m+1}_\xi([0,1]))}\ |\xi|\,|\widehat{h_0}(\xi)|\ \|_{L^2([-\pi,\pi])}
\lesssim \|\d_x h_0\|_{L^2(\RM)}.
$$
\end{proof}

\br\label{local_nonlocal}
\textup{
Unlike what occurs in the linearly decoupled case, linear bounds for a localized initial datum $d_0\in L^1(\RM)$ or a nonlocalized one of shift type $h_0 \bar U'$, $\d_xh_0\in L^1(\RM)$ provide in general the same decay rates. As a consequence, once these bounds are proved, the proof of nonlinear stability is identical to the one for localized perturbations \cite{JZ3}. In particular, the slow decay rate due to the Jordan block is compensated for 
nonlinear terms by the fact that they come in flux form.
}
\er

\subsection{Nonlinear stability}\label{s:proofg}
From differential equation \eqref{veq} together with integral equation \eqref{psidefg}-\eqref{closedg}, we readily obtain short-time existence, uniqueness and continuity with respect to $t$ of solution $(v,\psi_t,\psi_x)\in H^{K}(\RM)\times H^{K-1}(\RM)\times H^{K}(\RM)$ by a standard contraction-mapping argument treating most of the terms as sources in a heat equation. Associated with this solution define so long as it is finite
\ba\label{szeta}
\zeta(t)&:=\sup_{0\le s\le t}
\|(v, \psi_t,\psi_x)(s)\|_{H^K(\RM)\times H^{K-1}(\RM)\times H^{K}(\RM)}(1+s)^{1/4}.
\ea
Combining linear estimates with Proposition~\ref{damping}, we now prove an inequality for $\zeta$ that will yield global existence of our solutions.

\bl\label{sclaim}
Under assumptions (H1)--(H3) and (D1)--(D3), there exist positive constants $C$ and $\varepsilon$ such that if $(d_0,\partial_x h_0)$ is such that for some $T>0$
$$
E_0:=\|(d_0,\partial_x h_0)\|_{L^1(\RM)\cap H^K(\RM)}\ \leq\ \varepsilon\qquad\textrm{and}\qquad\zeta(T)\ \leq\ \varepsilon
$$
then, for all $0\leq t\leq T$,
$$
\zeta(t)\le C(E_0+\zeta(t)^2).
$$
\el

\begin{proof}\footnote{Compare to the argument of \cite[Lemma 4.2]{JZ2}, regarding localized perturbations in the decoupled case.}
By \eqref{eqn:Q}--\eqref{eqn:S} and corresponding bounds on the derivatives together
with definition \eqref{szeta} and equation \eqref{vperteq2} (used to bound $v_t$),
$$
\| \N(t)\|_{L^1(\RM)\cap H^1(\RM)}
\lesssim \|(v,v_x,v_{xx},\psi_t,\psi_x,\psi_{xx})(t)\|_{H^1(\RM)}^2
\le C\zeta(t)^2 (1+t)^{-\frac{1}{2}},\\
$$
\be\label{sQRSbds}
\|(\cQ,\cR,\cS)(t)\|_{L^1(\RM)\cap H^1(\RM)} \lesssim
\|(v,v_x,\psi_t,\psi_x)(t)\|_{H^1(\RM)}^2\le
C\zeta(t)^2 (1+t)^{-\frac{1}{2}},\\
\ee
so long as $\zeta(t)$ remains small. Applying the bounds \eqref{finaleundiffg}(i)--\eqref{finalgg}(i)
and \eqref{Spmodg}--\eqref{tildeSmodg} of Propositions \ref{greenbdsg} and \ref{modpropg}
to system \eqref{psidefg}-\eqref{closedg}, we obtain for any $2\le p<\infty$
\ba\label{sestg}
\|v(t)\|_{L^p(\RM)}& \le
C(1+t)^{-\frac{1}{2}(1-1/p)}E_0
+C\zeta(t)^2\int_0^{t} (1+t-s)^{-\frac{1}{2}(1-1/p)-\frac{1}{2}}
(1+s)^{-\frac{1}{2}}ds\\
&\le C_p (E_0+\zeta(t)^2) (1+t)^{-\frac{1}{2}(1-1/p)}
\ea
and, with
\be\label{dt_S}
\int_0^t s^{\rm p}(t-s)\d_t\cS(s)ds\ =\
-\int_0^t \d_t[s^{\rm p}](t-s)\cS(s)ds+s^{\rm p}(0)\cS(t)-s^{\rm p}(t)\cS(0),
\ee
\ba\label{sestadg}
\|(\psi_t,\psi_x,\psi_{xx})(t)\|_{W^{K-1,p}(\RM)}&\le
C(1+t)^{-\frac{1}{2}(1-1/p)}(E_0+\zeta(t)^2)\\
&+C\zeta(t)^2\int_0^{t} (1+t-s)^{-\frac{1}{2}(1-1/p)-1/2}(1+s)^{-\frac{1}{2}}ds \\
&\le C_p(E_0+\zeta(t)^2) (1+t)^{-\frac{1}{2}(1-1/p)}.
\ea
Since the hypotheses of Proposition \ref{damping} are verified, from \eqref{Ebds} and \eqref{sestg}--\eqref{sestadg}, we thus obtain
\be \label{v_H1g}
\|v(t)\|_{H^K(\RM)} \le
 C(E_0+\zeta(t)^2) (1+t)^{-\frac{1}{4}}.
\ee
Combining this with \eqref{sestadg} for $p=2$, we obtain the result.
\end{proof}

\medskip

We are now ready to prove Theorem~\ref{oldmain}.

\begin{proof}[Proof of Theorem \ref{oldmain}]
As already mentioned, short-time existence and uniqueness ensuring continuously in time $(v,\psi_t,\psi_x)\in H^{K}(\RM)\times H^{K-1}(\RM)\times H^{K}(\RM)$ are proved in a standard way. Therefore, by Lemma \ref{sclaim} it follows by continuous induction that solutions are global in time and satisfy $\zeta(t)\le 2C E_0$ for $t \ge0$, if $E_0\leq\min(\{1/4C^2,\varepsilon,\varepsilon/2C\})$, yielding by \eqref{szeta} the result \eqref{mainest} for $p=2$.

For any $p_*<\infty$, applying \eqref{sestg}--\eqref{sestadg}, we obtain \eqref{mainest} for $2\le p\le p_*$
with a uniform constant $C$. Now rewrite \eqref{eqn:R} as $\cR=\cR_1+\d_x\cR_2$ with
$$
\cR_1:=- v\psi_t-\ks^2 v\left(\frac{\psi_{x}}{1-\psi_x}\right)_x+ \ks^2\bar U_x\frac{\psi_x^2}{1-\psi_x},\qquad
\cR_2:=\ks^2 v\frac{\psi_{x}}{1-\psi_x}.
$$
Taking $p_*\geq4$ and estimating
\be\label{QRST}
\|(\cQ,\cR_1,\cR_2,\cS)(t)\|_{L^2(\RM)} \lesssim \|(v,\psi_t,\psi_x,\psi_{xx})(t)\|_{L^4(\RM)}^2\le CE_0(1+t)^{-\frac{3}{4}}
\ee
in place of the weaker \eqref{sQRSbds}, then applying to integral terms \eqref{finaleg}(ii) in place of \eqref{finaleg}(i), we obtain, bounding again the $\d_t\cS$ contribution using \eqref{dt_S},
\ba\label{sestad2}
\|(\psi_t,\psi_x,\psi_{xx})(t)\|_{W^{K-1,p}(\RM)}& \le
C(1+t)^{-\frac{1}{2}(1-1/p)}E_0+CE_0^2\int_0^{t} (1+t-s)^{-\frac{1}{2}(1/2-1/p)-1/2}
(1+s)^{-\frac{3}{4}}ds \\
&\le CE_0(1+t)^{-\frac{1}{2}(1-1/p)},
\ea
for $2\le p\le \infty$.
Likewise, using \eqref{QRST} together with bound
$$
\|(\cQ,\cR_1)(t)\|_{H^{K-1}(\RM)}+\|(\cR_2,\cS)(t)\|_{H^K(\RM)}\lesssim E_0^2 (1+t)^{-\frac{1}{2}}
$$
obtained from the bound on $\zeta$, and 
$$
\int_0^t \tilde S(t-s)\d_t\cS(s)ds\ =\
-\int_0^t \d_t[\tilde S](t-s)\cS(s)ds+\tilde S(0)\cS(t)-\tilde S(t)\cS(0),
$$
instead of \eqref{dt_S}, we may use \eqref{finalgg}(ii) rather than \eqref{finalgg}(i) to get, 
provided $l+2m\leq K-3$,
\ba\label{sest2}
\|\d_x^l\d_t^m v(t)\|_{L^p(\RM)}
&\le
C\,(1+t)^{-\frac{1}{2}(1-1/p)}E_0
+C\,E_0^2\int_0^t e^{-\eta (t-s)}(1+s)^{-\frac{1}{2}}ds\\
&
\quad+C\,E_0^2\int_0^{t} (1+t-s)^{-\frac{1}{2}(1/2-1/p)-\frac{1}{2}}(1+s)^{-\frac{3}{4}}ds\\
&\le C\,E_0 (1+t)^{-\frac{1}{2}(1-1/p)}
\ea
and achieve the proof of \eqref{mainest} for $2\le p\le \infty$.

Estimate \eqref{andpsi} then follows through \eqref{psidefg}
using \eqref{finaleg}(i) and \eqref{Spmodg}, by 
$$
\|\psi (t) \|_{L^\infty(\RM)}
\ \le\ C E_0+CE_0^2\int_0^{t} (1+t-s)^{-\frac{1}{2}}(1+s)^{-\frac{1}{2}}ds
\ \le\ CE_0,
$$
yielding nonlinear stability in $L^\infty$ since
$$
\tilde u(x,t)-\bar U(x)= \tilde u(x,t)-\bar U(x+\psi(x,t))+\bar U(x+\psi(x,t),t)-\bar U(x)
$$
so that, by Lemma \ref{switchlem},
$$
\|\tilde{u}(t)-\bar U\|_{L^\infty(\RM)}\lesssim \|v(t)\|_{L^\infty(\RM)}+\|\psi(t)\|_{L^\infty(\RM)}\|\psi_x(t)\|_{L^\infty(\RM)}
+\|\psi(t)\|_{L^\infty(\RM)}\sup_{[0,1]}|\bar U'|\ .
$$
Finally, Lemma \ref{switchlem} provides \eqref{backest}.
\end{proof}

In proving Theorem \ref{oldmain}, we actually got or could get more estimates than announced. Since we need these extra estimates to prove Theorem \ref{main}, we record here for later these other bounds.
\bpr\label{known_rkg} Under the assumptions of Theorem \ref{oldmain} and with its notations, for $2\le p\le\infty$ and $t\geq0$,
\ba\label{known_vg}
\|v(t)\|_{L^p(\RM)}
&\lesssim E_0 (1+t)^{-\frac{1}{2}(1-1/p)}\\
\|\d_x^l\d_t^mv(t)\|_{L^p(\RM)}
&\lesssim E_0 (1+t)^{-\frac{1}{2}}\\
\|v(t)\|_{H^K(\RM)}
&\lesssim E_0 (1+t)^{-\frac{1}{4}},
\ea
when $1\leq l+2m\leq K-3$ and when $3\leq l+m$ and $l+2m\leq K+1$
\ba\label{known_psig}
\|(\psi_t,\psi_x)(t)\|_{L^{p}(\RM)}
&\lesssim E_0 (1+t)^{-\frac{1}{2}(1-1/p)}\\
\|(\psi_{tt},\psi_{tx},\psi_{xx})(t)\|_{L^{p}(\RM)}
&\lesssim E_0 \ln(2+t)(1+t)^{-\frac{3}{4}}\\
\|\d_x^l\d_t^m\psi(t)\|_{L^{p}(\RM)}
&\lesssim E_0 (1+t)^{-\frac{3}{4}},\\
\ea
where $v$ is as in \eqref{pertvar}.
\epr

\begin{proof}
Estimate \eqref{known_vg}(i) is \eqref{mainest}(i) and \eqref{known_vg}(iii) follows from the bound on $\zeta$. Estimate \eqref{known_vg}(ii) is proved incorporating extra decay obtained from differentiation in \eqref{sest2}.

Estimate \eqref{known_psig}(i) is \eqref{mainest}(ii). Bounds \eqref{known_psig}(ii)-(iii) follow along the lines of \eqref{sestad2} benefiting from extra decay provided by differentiation.
\end{proof}

\section{Behavior}\label{s:gbehavior}

The first step in going from Theorem~\ref{oldmain} to Theorem~\ref{main} is to obtain a refined expansion of the solution operator $S(t)=e^{tL}$, allowing us to split further the integral equation \eqref{vimplicit}. This proceeds by refining the spectral expansions from which they ultimately derive.

\subsection{Refined spectral expansion and decompositions}\label{s:refg}
\bl[\cite{NR1,NR2}]\label{keylemg}
Assuming (H1)--(H3), we may choose a parametrization in such a way that, 
for the quantities involved in \eqref{phibifg},
\be\label{kderivg}
\partial_\xi q_{n+1}(0)= i\ks \d_kU_{|(\Ms,\ks)}.
\ee
\el

\begin{proof}
We have already observed that the proof of Proposition \ref{whiteig} provides
$$
L_0(\d_\xi q_{n+1}(0)-i\ks\d_kU_{|(\Ms,\ks)})\in \Sigma_0.
$$
The latter point implies that $\d_\xi q_{n+1}(0)-i\ks\d_kU_{|(\Ms,\ks)}\in\Sigma_0$. Therefore, we only need to show that, for $1\leq j\leq n+1$, $\langle q_j(0),\d_\xi q_{n+1}(0)\rangle=\langle q_j(0),i\ks \d_kU_{|(\Ms,\ks)}\rangle$.

Moreover, the proof of Proposition \ref{whiteig} also yields $\langle q_j(0),\d_\xi q_{n+1}(0)\rangle=0$ for all $j\neq n+1$. We may assume that this relation holds for $j=n+1$ by normalizing $q_{n+1}(\xi)$ according to $\langle\tilde q_{n+1}(0), q_{n+1}(\xi)\rangle=1$. Indeed, once this is done, expanding to first order in $\xi$ this normalization and the duality relation $\langle\tilde q_{n+1}(\xi), q_{n+1}(\xi)\rangle=1$ provides the desired cancellation.

Now from the second part of \eqref{alg_order1}(iii), we already know $\langle q_j(0),i\ks \d_kU_{|(\Ms,\ks)}\rangle=0$ for all $j\neq (n+1)$. Moreover, up to changing parametrization by a $k$-dependent shift, we may add to $\d_kU_{|(\Ms,\ks)}$ a suitable multiple of $\bar U'$ and get
\be\label{norm_kg}
\langle\tilde q_{n+1}(0),i\ks \d_kU_{|(\Ms,\ks)}\rangle\ =\ 0.
\ee
This yields \eqref{kderivg}.
\end{proof}

Accordingly, following \cite{JNRZ2}, we refine \eqref{decompg} and re-express $S(t)$ as
\be\label{decomprg}
S(t)=R^{\rm p}(t)+ R^M(t)+ \tilde R(t),
\ee
where
\be\label{Rpg}
R^{\rm p}(t)=(\bar U'+\partial_k \bar U\ \ks\partial_x)\, e_{n+1}\cdot s^{\rm p}(t),
\ee
with $s^{\rm p}$ as in \eqref{spg},
\be\label{RMg}
R^M(t):=\partial_M \bar U\cdot s^M(t),\qquad
s^M(t):=\sum_{j=1}^{n+1}s^{M}_j(t),\qquad
s_j^M(t):=\begin{pmatrix}\Id_n& \begin{array}{c}0\\ \vdots\\ 0\end{array}\end{pmatrix}\ks\d_x s_j^{\rm p}(t),
\ee
and
\ba\label{tildeSrg}
(\tilde R(t)g)(x)&:=
\int_{-\pi}^{\pi} e^{i\xi x}(1-\alpha(\xi))(e^{tL_\xi}  \check g(\xi))(x)d\xi
+\int_{-\pi}^{\pi} e^{i\xi x}\alpha(\xi)(e^{tL_\xi} \tilde \Pi(\xi)\check g(\xi))(x)d\xi\\
&\displaystyle\
+\int_{-\pi}^{\pi}e^{i\xi x}\alpha(\xi)\quad\sum_{j=1}^{n+1}\ e^{\lambda_{j}(\xi) t}\
\frac{\phi_j^{quad}(\xi,x)}{i\ks\xi}\ \langle\tilde\phi_{j}(\xi),\check g(\xi)\rangle_{L^2([0,1])}\ d\xi,\\
\ea
with
$$
\phi_j^{quad}(\xi,x)\ =\ \phi_{j}(\xi,x)-(i\ks\xi)\sum_{l=1}^n\beta_l^{(j)}(\xi)q_l(0,x)-\beta_{n+1}^{(j)}(\xi)\big(q_{n+1}(0,x)+\xi\d_\xi q_{n+1}(0,x)\big),
$$
where $\alpha$ is an already introduced smooth cutoff supported where $|\xi|\leq\xi_0$, $\tilde\Pi(\xi)$ is, as defined in \eqref{Pig}, the complementary
projection of $\Pi^{\rm p}(\xi)$, the spectral projection on $\Sigma_\xi:=\Span\{\phi_j(\xi)\}_{j=1}^{n+1}$.

\medskip

%(TODO: fix wording)
With \eqref{Rpg}-\eqref{RMg} in mind, for $\tilde{u}$ satisfying \eqref{cons}, 
we may refine the nonlinear decomposition \eqref{pertvar} into
\ba\label{pertvarrg}
z(x,t) &=\tilde{u}(x-\psi(x,t),t)-\bar U(x)-
\d_k\bar U(x) \, \ks\psi_x(x,t)
-\partial_M\bar U(x)  \cdot M(x,t)\\
&=v(x,t)- \d_k\bar U(x) \, \ks\psi_x(x,t) -\partial_M\bar U(x)\cdot M(x,t)
\ea
where $\psi(x,t)$ still satisfies \eqref{psidefg}, and $M(x,t)$ is defined through \eqref{closedrg}(iii) just below. Further, recall $d_0:=\tilde u_0(\cdot-h_0(\cdot))-\bar U$, and let $\chi$ be the smooth cutoff function of \eqref{psidefg}--\eqref{closedg}, with $\chi(t)=0$ for $t\le 1/2$, and $\chi(t)=1$ for $t\ge 1$. 
With definitions \eqref{decomprg}--\eqref{RMg}, we have the following lemma.

\begin{lemma}\label{lem:canestrg}
For $\CalN$ as in \eqref{veq}--\eqref{eqn:S}, the nonlinear residual $z$ defined in \eqref{pertvarrg} satisfies
\ba\label{closedrg}
z(t)&=\tilde R(t) (d_0+h_0\bar U')+ \int_0^t \tilde R(t-s) \N(s) ds\\
&\quad+(1-\chi(t))\Big(R^{\rm p}(t)(d_0+h_0\bar U')-h_0\bar U'-\d_k\bar U\,\ks\d_x h_0+\int_0^t R^{\rm p}(t-s)\N(s)ds\Big),\\
\psi(t)&=e_{n+1}\cdot s^{\rm p}(t)(d_0+h_0\bar U')+\int_0^t e_{n+1}\cdot s^{\rm p}(t-s)\N(s)ds\\
&-(1-\chi(t))\Big(e_{n+1}\cdot s^{\rm p}(t)(d_0+h_0\bar U')-h_0+\int_0^t e_{n+1}\cdot s^{\rm p}(t-s)\N(s)ds\Big),\\
M(t)&=s^M(t)(d_0+h_0\bar U')+\int_0^t s^M(t-s)\N(s)ds.\\
\ea
\end{lemma}

\begin{proof}
Recall, \eqref{psidefg}--\eqref{closedg}, that we have chosen $\psi$ so that $\psi$ and $v$ satisfy both \eqref{closedrg}(ii) and
\ba\label{oldnlg}
v(t)&=\tilde S(t) (d_0+h_0\bar U')+ \int_0^t \tilde S(t-s) \N(s) ds\\
&\quad
+(1-\chi(t))\Big(\bar U' e_{n+1}\cdot s^{\rm p}(t)(d_0+h_0\bar U')-h_0\bar U'+\int_0^t \bar U' e_{n+1}\cdot s^{\rm p}(t-s)\N(s)ds\big),\\
\ea
where $\tilde S(t)=S(t)-\bar U' e_{n+1}\cdot s^{\rm p}(t)$. Using now $z(t)=v(t)-\d_k\bar U\,\ks\psi_x(t)-\partial_M\bar U\cdot M(t)$,
$$
\tilde R(t)=\tilde S(t)-\d_k\bar U \, \ks\d_x e_{n+1}\cdot s^{\rm p}(t)-\partial_M\bar U\cdot s^M(t),
$$
and $R^{\rm p}(t)=(\bar U'+\d_k\bar U\,\ks \d_x)\,e_{n+1}\cdot s^{\rm p}(t)$, $R^M(t)=\partial_M \bar U\cdot s^M(t)$,
equation \eqref{closedrg}(i) follows from \eqref{closedrg}(ii)-(iii) and \eqref{oldnlg}.
\end{proof}

We now establish the refined linear bounds needed to estimate the terms involved in \eqref{closedrg}

\subsection{Refined linear stability estimates}\label{s:linearg}

\subsubsection{Refined basic estimates}

\bpr\label{greenbdsrg}
Under assumptions (H1)-(H3) and (D1)-(D3), for all $t\geq0$, $1\le j\le n$,\\ $2\leq p\leq \infty$,
and for some $\eta>0$, $0\leq l+2m  \leq K+1$, and $2\le p\le \infty$,
\begin{align}\nonumber
\left\|\d_x^l \d_t^m s^M_j(t)\,g\right\|_{L^p(\RM)}
&\lesssim
\min
\begin{cases}
(1+t)^{-\frac{1}{2}(1-1/p)-\frac{l+m}{2}}\| g\|_{L^1(\RM)}\\
(1+t)^{-\frac{1}{2}(1/2-1/p)-\frac{l+m}{2}}\|g\|_{L^2(\RM)}\\
\end{cases}
\end{align}
\begin{align}\label{undiffeqrg}
\left\|\d_x^l \d_t^m \tilde R(t)  g \right\|_{L^p(\RM)}
&\lesssim
\min
\begin{cases}
e^{-\eta t}\|g\|_{H^{l+2m+1}(\RM)}+(1+t)^{-\frac{1}{2}(1-1/p)-\frac{1}{2}}\| g\|_{L^1(\RM)}\\
e^{-\eta t}\|g\|_{H^{l+2m+1}(\RM)}+(1+t)^{-\frac{1}{2}(1/2-1/p)-\frac12}\|g\|_{L^2(\RM)}\\
\end{cases}
\end{align}
and, for $1\le r\le K+1$,
\begin{align}\label{finalgsmg}
\left\|\d_x^l \d_t^m s^M_j(t) \d_x^r g\right\|_{L^p(\RM)}
&\lesssim
\min
\begin{cases}
(1+t)^{-\frac{1}{2}(1-1/p)-\frac{l+m}{2}-\frac{1}{2}}\| g\|_{L^1(\RM)}\\
(1+t)^{-\frac{1}{2}(1/2-1/p)-\frac{l+m}{2}-\frac{1}{2}}\|g\|_{L^2(\RM)}\\
\end{cases}
\end{align}
\begin{align}\label{finalgrg}
\left\|\d_x^l \d_t^m \tilde R(t) \d_x^r g\right\|_{L^p(\RM)}
&\lesssim
\min
\begin{cases}
e^{-\eta t}\|\d_x^r g\|_{H^{l+2m+1}(\RM)}+(1+t)^{-\frac{1}{2}(1-1/p)-1}\| g\|_{L^1(\RM)}\\
e^{-\eta t}\|\d_x^r g\|_{H^{l+2m+1}(\RM)}+(1+t)^{-\frac{1}{2}(1/2-1/p)-1}\|g\|_{L^2(\RM)}\\
\end{cases}.
\end{align}
\epr

\begin{proof}
Bounds on $s^M_j$ follow directly from known bounds on $s^{\rm p}_j$, whereas the proofs of estimates on $\tilde R$ are completely similar to the proofs of bounds on $\tilde S$ in Proposition \ref{greenbdsg}, with extra decay coming from a higher-order expansion of $\phi_j(\xi)$, leading to an extra $\xi$ factor in the third term of \eqref{tildeSrg}.
\end{proof}

\subsubsection{Refined linear modulation bounds}

\bpr\label{modproprg}
Under assumptions (H1)--(H3) and (D1)--(D3),
for all $t\geq0$, $2\leq p\leq \infty$, and $0\leq l+2m \leq K+1$,
\be\label{Spmodrg}
\|\partial_x^l\partial_t^m s^M_j(t)(h_0\bar U')\|_{L^p(\RM)}
\lesssim
(1+t)^{-\frac{1}{2}(1-1/p)-\frac{l+m}{2}}\|\partial_x h_0\|_{L^1 (\RM)
},
\ee
\be\label{tildeSmodrg}
\|\partial_x^l \partial_t^m \tilde R (t)(h_0\bar U')\|_{L^p(\RM)}
\lesssim (1+t)^{-\frac{1}{2}(1-1/p)-\frac12}
\|\partial_x  h_0\|_{L^1 (\RM) \cap H^{l+2m+1} (\RM)
}.
\ee
\epr

\begin{proof}
Again, bounds on $s^M_j$ follow directly from known bounds on $s^{\rm p}_j$, whereas the proof of estimates on $\tilde R$ is completely similar to the proof of bounds on $\tilde S$ in Proposition \ref{modproprg}, with extra decay coming 
from the higher-order expansion of $\phi_j(\xi)$, 
leading to an extra $\xi$ factor in the third term of \eqref{tildeSrg}.
\end{proof}

\subsection{Refined nonlinear stability estimates}\label{s:pertrg}
With these preparations, we obtain the following refinement of Theorem \ref{oldmain}.

\begin{proposition}\label{stepg}
Under the assumptions of Theorem \ref{oldmain}, for all $t\geq0$, $2\le p\le \infty$,
\ba\label{sharpestg}
\|z(t)\|_{L^p(\RM)}
&\lesssim E_0
\ln(2+t)\ (1+t)^{-\frac{3}{4}}\\
\|\d_x^l\d_t^mz(t)\|_{L^p(\RM)}
&\lesssim E_0
(1+t)^{-\frac{3}{4}},\qquad 1\leq l+2m\leq K-3
\ea
and
\ba\label{Mestg}
\|M(t)\|_{L^{p}(\RM)}
&\le CE_0 (1+t)^{-\frac{1}{2}(1-1/p)}\\
\|(M_t,M_x)(t)\|_{L^{p}(\RM)}
&\le CE_0 \ln(2+t)(1+t)^{-\frac{3}{4}}\\
\|\d_x^l\d_t^mM(t)\|_{L^{p}(\RM)}
&\le CE_0 (1+t)^{-\frac{3}{4}},\qquad l+m\geq 2,\ l+2m\leq K+1.\\
\ea
\end{proposition}

\begin{proof}
By \eqref{known_psig} and \eqref{known_vg}, we find
$$
\|\N(t)\|_{H^{K-2}(\R)} \le CE_0 (1+t)^{-\frac{3}{4}}.
$$
Applying bounds \eqref{undiffeqrg}(i)--\eqref{finalgrg}(ii) and \eqref{Spmodrg}-\eqref{tildeSmodrg}(i)
of Propositions \ref{greenbdsrg} and \ref{modproprg} to the system \eqref{closedrg}, 
we obtain for any $2\le p\le\infty$,
%provided $l+2m\leq K-3$,
$$
\begin{aligned}
\|z (t) \|_{L^p(\RM)}& \le
CE_0(1+t)^{-\frac{1}{2}(1-1/p)-\frac{1}{2}} + CE_0 \int_{0}^{t} (1+t-s)^{-\frac{1}{2}(1/2-1/p)-1}(1+s)^{-\frac{3}{4}}ds\\
&\le CE_0\ln(2+t)\ (1+t)^{-\frac{3}{4}},
\end{aligned}
$$
verifying \eqref{sharpestg}(i). Estimate \eqref{sharpestg}(ii) is proved in the same way dropping 
the log term thanks to an extra decay of $(1+t-s)^{-1/2}$ in the integral. 
Bound \eqref{Mestg} follows similarly as for the previous bounds on $\psi_x$, $\psi_t$.
\end{proof}

This verifies the estimate \eqref{refinedest}, thus validating the \emph{ansatz} \eqref{wref}. 
It now remains to establish comparisons with solutions of \eqref{whitham}. 
Note that, as explained in Appendix~\ref{W_app}, taking into account simplifications due to asymptotic equivalence of quadratic approximants and the change of variables $(\Id-\psi(\cdot,t))^{-1}$, it is sufficient to prove comparisons with solutions to system \eqref{quad_whitham},
which we do below.

\subsection{Quadratic approximation}
Before entering into comparisons with the Whitham equations, we identify now the main part of the nonlinear terms. 
To this end, we first need estimates revealing the characteristic speeds of each of the linear parts. For this purpose,
we introduce for $a_j$ as in \eqref{lambdag}
$$
D_j\ =\ \partial_t+a_j\ \ks\partial_x\ .
$$
Then setting
\be\label{eigenvectg}
V_j=\beta^{(j)}(0),\qquad \tilde V_j=\tilde\beta^{(j)}(0),
\ee
we define the following total derivatives
$$
D\ =\ \sum_{j=1}^nV_j\,D_j\,\tilde V_j^T\ .
$$
Note that, more compactly, we actually have
$$
D\ =\ \d_t+\left(\bp dF_{|(\Ms,\ks)}\\ -d\omega_{|(\Ms,\ks)}\ep-\cs\Id\right)\ks\d_x\
\ =\ \d_t+A_*\ks\d_x.
$$

\bl\label{cancelg}
Assuming (H1)-(H3) and (D1)-(D3), for all $t\geq0$, $2\leq p\leq \infty$,
\begin{align}\nonumber
\left\|\d_x^l\d_t^m D s^{\rm p}(t)g\right\|_{L^p(\RM)}
\lesssim
\min
\begin{cases}
(1+t)^{-\frac12(1-1/p)-\frac{l+m}{2}-\frac12}\|g\|_{L^1(\RM)}\\
(1+t)^{-\frac12(1/2-1/p)-\frac{l+m}{2}-\frac12}\|g\|_{L^2(\RM)}\\
\end{cases},
\end{align}
\begin{align}\nonumber
\left\|\d_x^l\d_t^m D s^{\rm p}(t)(\partial_x^r g)\right\|_{L^p(\RM)}
\lesssim
\min
\begin{cases}
(1+t)^{-\frac12(1-1/p)-\frac{l+m}{2}-1}\|g\|_{L^1(\RM)}\\
(1+t)^{-\frac12(1/2-1/p)-\frac{l+m}{2}-1}\|g\|_{L^2(\RM)}\\
\end{cases}
\end{align}
when $1\leq r\leq K+1$,
\be\label{dSpmodg}
\left\|\d_x^l\d_t^m D s^{\rm p}(t)(h_0\bar U')
\right\|_{L^p(\RM)}
\lesssim(1+t)^{-\frac12(1-1/p)-\frac{l+m}{2}-\frac12} \|\partial_x h_0\|_{L^1(\RM)},
\ee
and, for $g$ periodic of period $1$,
\be\label{dSpmod_tg}
\left\|\d_x^l\d_t^mD s^{\rm p}(t)(h_0g)
\right\|_{L^p(\RM)}
\lesssim(1+t)^{-\frac12(1-1/p)-\frac{l+m}{2}} \|\partial_x h_0\|_{L^1(\RM)}.
\ee
\el

\begin{proof}
From the observation that
\ba
(&\d_x^l\d_t^mDs^{\rm p}(t) g)(x)=
\sum_{j=1}^{n+1}V_j\int_{-\pi}^{\pi} \alpha(\xi)e^{\lambda_j(\xi)t}e^{i\xi x}\frac{\lambda_j(\xi)+i\xi a_j}{i\ks\xi}(i\xi)^l\lambda_j(\xi)^m\langle\tilde\phi_j(\xi),\check g(\xi)\rangle d\xi\\
&+\sum_{j,j'=1}^{n+1}V_j\int_{-\pi}^{\pi} \alpha(\xi)e^{\lambda_{j'}(\xi)t}e^{i\xi x}\frac{\lambda_{j'}(\xi)+i\xi a_j}{i\ks\xi}(i\xi)^l\lambda_j(\xi)^m\,
\tilde V_j\cdot(\beta^{(j')}(\xi)-\beta^{(j')}(0))\,\langle\tilde\phi_{j'}(\xi),\check g(\xi)\rangle d\xi
\ea
with for all $j,j'$
$$
\lambda_j(\xi)+i\xi a_j=\cO(\xi^2),\qquad
(\lambda_{j'}(\xi)+i\xi a_j)\tilde V_j\cdot(\beta^{(j')}(\xi)-\beta^{(j')}(0))=\cO(\xi^2),
$$
the proof follows the lines of previous estimates on $s^{\rm p}_j$. In particular, the discrepancy between decay rates in \eqref{dSpmodg} and \eqref{dSpmod_tg} is due to the fact that \eqref{dSpmodg} benefits from cancellations
$$
\langle\tilde\phi_{j}(0),\bar U'\rangle\ =\ 0,\qquad \textrm{for all }\ 1\leq j\leq n+1.
$$
\end{proof}

\br\label{linear_group_velocity}
\textup{
The fact that in \eqref{spg} we have kept $\beta_j(\xi)$ instead of $\beta_j(0)$ precludes similar higher-order 
estimates; indeed, it may be readily checked that application of $D^q$ does \emph{not} enhance decay by $(1+t)^{-q}$.}
\er

The previous linear estimates may be transposed to the nonlinear level
as follows.

\begin{corollary}\label{tradecorg}
Assuming (H1)-(H3) and (D1)-(D3), the phase $\psi$ of Theorem \ref{oldmain} and the mean $M$ of Theorem \ref{main} satisfy for 
all $t\geq0$, $2\leq p\leq \infty$,
\be\label{tradeeqg}
\|\psi_t(t)+\ks dc_{|(\Ms,\ks)}(M(t),\ks\psi_x(t))\|_{L^p(\RM)}
\lesssim E_0\ln(2+t)\ (1+t)^{-3/4}
\ee
$$
\|\d_x^l\d_t^mD\,(M(t),\ks\psi_x(t))\|_{L^p(\RM)}
\lesssim E_0(1+t)^{-3/4},\qquad 0\leq l+2m+2\leq K-1\ .
$$
\end{corollary}

\begin{proof}
Completely similar to the proofs of \eqref{known_psig} in Proposition \ref{known_rkg} and \eqref{Mestg} in Proposition~\ref{stepg}.
\end{proof}

Now the next lemma pulls out the dominant part of $\mathcal{N}(t)$.

\bl\label{sharplemmag}
Assuming (H1)--(H3), (D1)--(D3), we have
$$
\begin{aligned}
\mathcal{N}(t)&=\partial_x\
\left(M(t)^T f^{\rm p}_{MM}\cdot M(t)+\ks\psi_x(t)f^{\rm p}_{kM}\cdot M(t)+f^{\rm p}_{kk}(\ks\psi_x(t))^2+ r_1(t)\right)\\
&+\partial_t\
\left(M(t)^T g^{\rm p}_{MM}\cdot M(t)+\ks\psi_x(t)g^{\rm p}_{kM}\cdot M(t)+g^{\rm p}_{kk}(\ks\psi_x(t))^2+r_2(t)\right),
\end{aligned}
$$
where $f^{\rm p}_{jk}$, $g^{\rm p}_{jk}$ are periodic of period $1$
and $\|r_j(t)\|_{L^1(\RM)}\lesssim E_0^2\ln(2+t)(1+t)^{-1}$, explicitly
$$
\begin{aligned}
f^{\rm p}_{MM}&=-\ks\,\frac{1}{2}d^2f(\bar U)(\partial_M\bar U,\partial_M\bar U)+\ks\d_M\cs\ \d_M\bar U,\\
f^{\rm p}_{kM}&=-\ks\,d^2f(\bar U)(\partial_k\bar U,\partial_M\bar U)+\ks\d_k\cs\ \d_M\bar U +\ks\d_k\bar U\ \d_M\cs+\ks\d_M\bar U',\\
f^{\rm p}_{kk}&=-\ks\,\frac12d^2f(\bar U)(\d_k\bar U,\d_k\bar U)+\ks\d_k\cs\ \d_k\bar U+\bar U'+\ks\d_k\bar U'\\
\end{aligned}
$$
and
$$
\begin{aligned}
g^{\rm p}_{MM}&= 0, \qquad
g^{\rm p}_{kM}=\frac{1}{\ks}\partial_M \bar U,\qquad
g^{\rm p}_{kk}&=\frac{1}{\ks}\partial_k\bar U .
\end{aligned}
$$
\el

\begin{proof} 
This follows by a direct, but tedious, computation, using estimates \eqref{known_psig}, \eqref{sharpestg}, \eqref{Mestg} and \eqref{tradeeqg}, formulas from Lemma~\ref{lem:canest} and $v= z + \d_k\bar U\,\ks\psi_x +\partial_M\bar U\cdot M$.
\end{proof}

\subsection{Comparison with the linearized Whitham equations}
Now we begin comparisons with the Whitham equations by proving 
that the dynamics described by $s^{\rm p}(t)$ are well-approximated by the evolution 
of the linearized Whitham equations. To this end, we linearize about $w\equiv 0$ the quadratic approximant of the second-order Whitham system \eqref{whitham} and get
\be\label{lin_quad_W}
w_t+\ks A_*w_x=\ks^2\tilde B_* w_{xx}.
\ee
We begin with some observations about \eqref{lin_quad_W}.

\medskip

\bl[\cite{NR1,NR2}]\label{dlemg}
Assuming (H1)--(H3), the coefficients $a_j$, $b_j$ of the expansion $\lambda_j(\xi)=-i\ks\xi a_j+(i\ks\xi)^2 b_j+\cO(\xi^3)$ given in \eqref{lambdag} are given by the eigenvalues of the simultaneously-diagonalized coefficient matrices $A_*$, $\tilde B_*$, respectively, of the linearized quadratic approximant \eqref{lin_quad_W}.
\el

\begin{proof} 
The first-order relation has already been established above in Proposition \ref{whiteig}. The second-order relation essentially follows from
the fact that the processes of linearization and formal expansion commute. The exact computations depend on whether we obtained the modulation system with the strategy followed in Section \ref{firstway} or with the one in Section \ref{secondway}. In the latter case, computations are rather light while in the former case they are quite tedious but completely similar to those in \cite{NR1,NR2} in the context of the Saint-Venant and Korteweg-de Vries/Kuramoto-Sivashinsky equations. See also the proof given in \cite{FST} in the case of the Kuramoto-Sivashinsky equation.
\end{proof}

\br\label{diffrmk}
\textup{
It is worth mentioning that the diffusion matrix of \eqref{whitham} is uniquely determined {\it only up to asymptotic equivalence}, i.e., it is $\tilde B_*$ and not $B_*$ that is uniquely returned by the process of formal expansion. The relation between $b_j$ and $\sigma(\tilde B_*)$ was used
in \cite{FST} as a means to verify the spectral stability assumption (D2) for the critical eigenvalues $\Re\lambda_j (\xi)\sim \ks^2\Re 
b_j\xi^2$.  However, in the more recent studies \cite{BJNRZ3,BJNRZ4}, we find it more convenient to instead verify (D2) entirely by numerical Evans function study, determining $a_j$ $b_j$ at the same time by numerical Taylor expansion.
}
\er

\medskip

\bl\label{lincong}
Assuming (H1)--(H3), the solution operator $\Sigma(t)$
of the linearized quadratic approximant \eqref{lin_quad_W} is given by 
$$
\Sigma(t)\ =\ \sum_{j=1}^{n+1} \sigma_j(t) V_j\tilde V_j^T,
$$
where the  $\sigma_j(t)$ are solution operators of the convected heat equations
$$
u_t+a_{j}\ks u_x=b_{j} \ks^2u_{xx},
$$
and the vectors $V_j$ and $\tilde V_j$ are defined in \eqref{eigenvectg}.
\el

\begin{proof}
This follows by a straightforward diagonalization argument.
\end{proof}

\medskip

We now come to linear comparisons, i.e. we aim to verify that the dynamics described by $s^p(t)$ are well-approximated
by the evolution of the linearized Whitham equations \eqref{lin_quad_W}. 
To write the comparison results as compactly as possible, we introduce the operator $\cI$ defined by
\be
\widehat{\cI g}(\xi)\ =\ \frac{1}{i\xi}\widehat{g}(\xi).
\ee
Note that $\cI$ can be identified as the convolution with a step function, and hence takes $L^1(\RM)$ into $L^\infty(\RM)$.

\bpr\label{heatcompg}
Assuming (H1)--(H3) and (D1)--(D3), let $\Sigma (t)$ be the solution operator of \eqref{lin_quad_W} and $g$ be a periodic function on $[0,1]$, $g\in H^1([0,1])$.
Then, for all $t\geq0$, $2\leq p\leq \infty$,
\ba\label{heatwhitham_init}
\big\|
\partial_x^l s^{\rm p}(t)&(h_0 \bar U'+d)
-\frac{1}{\ks}\Sigma(t)\cI\d_x^l\bp  d-(\bar U-\langle\bar U\rangle)\d_x h_0\\ \ks\partial_x h_0\ep \big\|_{L^p(\RM)}\\
&
\lesssim
\begin{cases}
(1+t)^{-\frac{1}{2}(1-1/p)-\frac{1}{2}}t^{-\frac{l-1}{2}}(\|\d_xh_0\|_{L^1(\RM)\cap L^\infty(\RM)}+\|d\|_{L^1(\RM)\cap L^\infty(\RM)}),&\quad l\geq 1\\
(1+t)^{-\frac{1}{2}(1-1/p)}(\|\d_xh_0\|_{L^1(\RM)\cap L^2(\RM)}+\|d\|_{L^1(\RM)\cap L^2(\RM)}),&\quad l=0\\
\end{cases};
\ea
\ba\label{heatwhitham_flux}
\big\|
\partial_x^ls^{\rm p}(t)\d_x(h_0 g)
&-\frac{1}{\ks}\Sigma(t)\cI\d_x^l\bp[\displaystyle\langle g\rangle-i\sum_{l=1}^n\langle\d_\xi\tilde q_l(0),\d_x g\rangle e_l]\d_xh_0\\
\ks\langle \tilde q_{n+1}(0),\d_x g\rangle\d_xh_0\ep \big\|_{L^p(\RM)}\\
&\lesssim
\begin{cases}
(1+t)^{-\frac{1}{2}(1-1/p)-\frac12}t^{-\frac{l-1}{2}}\|\d_xh_0\|_{L^1(\RM)\cap L^\infty(\RM)},&\quad l\geq 1\\
(1+t)^{-\frac{1}{2}(1-1/p)}\|\d_xh_0\|_{L^1(\RM)\cap L^2(\RM)},&\quad l=0\\
\end{cases};
\ea
\ba\label{heatwhitham_time}
\big\|
\partial_x^{l+1} s^{\rm p}(t)(h_0 g)
&-\frac{1}{\ks}\Sigma(t)\cI\d_x^l\bp  \langle g\rangle \d_xh_0\\0\ep \big\|_{L^p(\RM)}\\
&\lesssim
\begin{cases}
(1+t)^{-\frac{1}{2}(1-1/p)-\frac12}t^{-\frac{l-1}{2}}\|\d_xh_0\|_{L^1(\RM)\cap L^\infty(\RM)},&\quad l\geq 1\\
(1+t)^{-\frac{1}{2}(1-1/p)}\|\d_xh_0\|_{L^1(\RM)\cap L^2(\RM)},&\quad l=0\\
\end{cases}.
\ea
\epr

\begin{proof}
{\it (i) (Proof of \eqref{heatwhitham_init}, case $h_0=0$).} Following the proof of \eqref{finaleundiffg} in Proposition \ref{greenbdsg}, we obtain that the difference between $\partial_x^l s^{\rm p}(t)(d)$ and the function
\be\label{aux_LF}
x\longmapsto \sum_{j=1}^{n+1}\int_{-\pi}^{\pi}e^{i\xi x}e^{(-i\ks a_j\xi-\ks^2 b_j\xi^2)t}\frac{1}{i\ks\xi}(i\xi)^l\beta^{(j)}(0)
\langle\tilde\phi_{j}(0,\cdot), \check d(\xi,\cdot)\rangle_{L^2([0,1])}d\xi
\ee
is bounded in $L^p(\RM)$ by $(1+t)^{-\frac{1}{2}(1-1/p)-\frac{l}{2}}\|d\|_{L^1(\RM)}$. Since for $1\leq j\leq n+1$ and $|\xi|\leq\pi$
we have
$$
\beta^{(j)}(0)\langle\tilde\phi_{j}(0,\cdot), \check d(\xi,\cdot)\rangle_{L^2([0,1])}
\ =\ V_j\,\tilde V_j\cdot\bp \langle\check d(\xi,\cdot)\rangle\\0\ep
\ =\ V_j\,\tilde V_j\cdot\bp \widehat d(\xi)\\0\ep,
$$
the function in \eqref{aux_LF} is recognized to be the low-frequency part of $\frac{1}{\ks}\Sigma(t)\cI\d_x^l\bp  d\\0\ep$,
whose high-frequency part is still to be bounded. 
When $l=0$, we bound it in $L^p(\RM)$ by $Ce^{-\eta t}\|d\|_{L^{1/(1/p+1/2)}(\RM)}$  using, for $1\leq j\leq n+1$
$$
\left\|\xi\mapsto e^{-\ks^2b_j\xi^2 t}\xi^{-1}\right\|_{L^2(\RM\setminus[-\pi,\pi])}\lesssim e^{-\eta t}
$$
for some $\eta>0$. When $l\geq 1$, however, we bound the high-frequency part of $\frac{1}{\ks}\Sigma(t)\cI\d_x^l\bp  d\\0\ep$ in $L^p(\RM)$ by $Ct^{-\frac{l-1}{2}}e^{-\eta t}\|d\|_{L^p(\RM)}$, recognizing it as the convolution of $d$ with a kernel that is bounded pointwise (using Haussdorff-Young estimates) by
$$
x\ \longmapsto\ C\ t^{-\frac{l}{2}}e^{-\eta\,t}\frac{1}{1+\frac{x^2}{t}},
$$
which is bounded in $L^1(\RM)$ by $Ct^{-\frac{l-1}{2}}e^{-\eta\,t}$, for some $\eta>0$.

\medskip

{\it (ii) (Proof of \eqref{heatwhitham_init}, case $d=0$).} Following the proof of \eqref{Spmodg} in Proposition \ref{modpropg}, we obtain that the difference between $\partial_x^l s^{\rm p}(t)(h_0\bar U')$ and
\ba\label{aux_LFmod}
x\longmapsto &\sum_{j=1}^{n+1}\int_{-\pi}^{\pi}e^{i\xi x}e^{(-i\ks a_j\xi-\ks^2 b_j\xi^2)t}\frac{1}{i\ks\xi}(i\xi)^l\beta^{(j)}(0)
\langle\tilde\phi_{j}(0,\cdot),\bar U' [\check h_0(\xi,\cdot)-\widehat h_0(\xi)]\rangle_{L^2([0,1])}d\xi\\
&+\sum_{j=1}^{n+1}\int_{-\pi}^{\pi}e^{i\xi x}e^{(-i\ks a_j\xi-\ks^2 b_j\xi^2)t}\frac{1}{i\ks\xi}(i\xi)^l\beta^{(j)}(0)
\langle\tilde\phi_{j}(0,\cdot)+\xi\,\d_\xi\tilde\phi_{j}(0,\cdot),\bar U' \widehat h_0(\xi)\rangle_{L^2([0,1])}d\xi\
\ea
is bounded in $L^p(\RM)$ by $(1+t)^{-\frac{1}{2}(1-1/p)-\frac{l}{2}}\|\d_x h_0\|_{L^1(\RM)}$. Now, as in the proofs of \eqref{Spmodcg}-\eqref{Spmodlayerg} in Proposition \ref{modpropg}, we observe that, for $1\leq j\leq n+1$ and $\xi\in[-\pi,\pi]$,
$$
\begin{aligned}
\langle\tilde\phi_{j}(0),\bar U' [\check h_0(\xi)-\widehat h_0(\xi)]\rangle
&=-\tilde V_j\cdot\bp[(\bar U-\langle\bar U\rangle)\d_xh_0]\ \widehat{}\ (\xi)\\0\ep
+i\xi\,\tilde V_j\cdot\bp\langle\bar U,[\check h_0(\xi)-\widehat{h_0}(\xi)]\rangle\\0\ep\\
\langle\tilde\phi_{j}(0)+\xi\,\d_\xi\tilde\phi_{j}(0),\bar U' \widehat h_0(\xi)]\rangle
&=\xi\langle\d_\xi\tilde\phi_{j}(0),\bar U'\rangle\ \widehat h_0(\xi)
\ =\ \tilde V_j\cdot\bp0\\\ks \d_x h_0\ep.
\end{aligned}
$$
Up to a term that is also bounded in $L^p(\RM)$ by $(1+t)^{-\frac{1}{2}(1-1/p)-\frac{l}{2}}\|\d_x h_0\|_{L^1(\RM)}$, the function in \eqref{aux_LFmod} is recognized as the low-frequency part of $\frac{1}{\ks}\Sigma(t)\cI\d_x^l\bp -(\bar U-\langle U\rangle)\d_x h_0 \\\ks\d_x h_0\ep$. The remaining high-frequency part is bounded as in the case $h_0=0$ above. This completes the proof of \eqref{heatwhitham_init} in the case $d=0$ and by linearity the proof of \eqref{heatwhitham_init} in any case.

\medskip

{\it (iii) (Proof of \eqref{heatwhitham_flux}).} Combining elements of the proofs of \eqref{finaleundiffg} in Proposition \ref{greenbdsg} and \eqref{Spmodg} in Proposition \ref{modpropg}, we obtain that the difference between $\partial_x^l s^{\rm p}(t)\d_x(h_0 g)$ and
$$
\begin{aligned}
x\longmapsto &\sum_{j=1}^{n+1}\int_{-\pi}^{\pi}e^{i\xi x}e^{(-i\ks a_j\xi-\ks^2 b_j\xi^2)t}\frac{1}{i\ks\xi}(i\xi)^l\beta^{(j)}(0)
\langle\tilde\phi_{j}(0,\cdot),i\xi\widehat h_0(\xi)g\rangle_{L^2([0,1])}d\xi\\
&+\sum_{j=1}^{n+1}\int_{-\pi}^{\pi}e^{i\xi x}e^{(-i\ks a_j\xi-\ks^2 b_j\xi^2)t}\frac{1}{i\ks\xi}(i\xi)^l\beta^{(j)}(0)
\langle-\xi\,\d_x\d_\xi\tilde\phi_{j}(0,\cdot),\widehat h_0(\xi)g\rangle_{L^2([0,1])}d\xi\
\end{aligned}
$$
is bounded in $L^p(\RM)$ by $(1+t)^{-\frac{1}{2}(1-1/p)-\frac{l}{2}}\|\d_x h_0\|_{L^1(\RM)}$. Since, for $1\leq j\leq n+1$ and $\xi\in[-\pi,\pi]$,
\ba\label{rel_calc}
\langle\tilde\phi_{j}(0,\cdot),i\xi\widehat h_0(\xi)g\rangle
&=\ \tilde V_j\cdot\bp\langle g\rangle\widehat{\d_x h_0}(\xi)\\0\ep\\
-\xi\langle\d_x\d_\xi\tilde\phi_{j}(0,\cdot),\widehat h_0(\xi)g\rangle
&=\ \tilde V_j\cdot\bp\displaystyle-i\sum_{l=1}^n\langle\d_\xi\tilde q_l(0),\d_x g\rangle e_l\\
\ks\langle \tilde q_{n+1}(0),\d_x g\rangle\ep\widehat{\d_xh_0}(\xi),
\ea
the function in \eqref{aux_LFmod} is the low-frequency part of
$$
\frac{1}{\ks}\Sigma(t)\cI\d_x^l\bp[\displaystyle\langle g\rangle-i\sum_{l=1}^n\langle\d_\xi\tilde q_l(0),\d_x g\rangle e_l]\d_xh_0\\
\ks\langle \tilde q_{n+1}(0),\d_x g\rangle\d_xh_0\ep
$$
whose high-frequency part is bounded as above.

\medskip

{\it (iv) (Proof of \eqref{heatwhitham_time}).} Following the proof of \eqref{Spmodg} in Proposition \ref{modpropg}, we obtain that the difference between $\partial_x^{l+1} s^{\rm p}(t)(h_0g)$ and
$$
x\longmapsto \sum_{j=1}^{n+1}\int_{-\pi}^{\pi}e^{i\xi x}e^{(-i\ks a_j\xi-\ks^2 b_j\xi^2)t}\frac{1}{i\ks\xi}(i\xi)^{l+1}\beta^{(j)}(0)
\langle\tilde\phi_{j}(0,\cdot), \widehat{h_0}(\xi)g\rangle_{L^2([0,1])}d\xi
$$
is bounded in $L^p(\RM)$ by $(1+t)^{-\frac{1}{2}(1-1/p)-\frac{l}{2}}\|\d_x h_0\|_{L^1(\RM)}$. Again this function is the low-frequency part of the expected term, whose high-frequency part may be bounded as in the proof of \eqref{heatwhitham_init} in the case $h_0=0$.
\end{proof}

Note that, since $A_*$ and $\Sigma(t)$ commute, combining \eqref{heatwhitham_time} with \eqref{dSpmod_tg} 
we obtain for any periodic function $g$ the bound
$$
\begin{aligned}
\|\d_x^l \d_ts^{\rm p}(t)(h_0 g)&+\Sigma(t) A_*\d_x^l\bp\langle g\rangle h_0\\0\ep\|_{L^p(\RM)}\\
&\lesssim
\begin{cases}
(1+t)^{-\frac{1}{2}(1-1/p)-\frac12}t^{-\frac{l-1}{2}}\|\d_xh_0\|_{L^1(\RM)\cap L^\infty(\RM)},&\quad l\geq 1,\\
(1+t)^{-\frac{1}{2}(1-1/p)}\|\d_xh_0\|_{L^1(\RM)\cap L^2(\RM)},&\quad l=0.\\
\end{cases}
\end{aligned}
$$

\subsection{Nonlinear connection to the Whitham equations}\label{s:noncom}
Before proving that the local means, wavenumber, and phase are indeed well-approximated by solutions of the appropriate Whitham equations, it is necessary to compute some averages involving the main part of the nonlinear term $\cN$ identified in Lemma~\ref{sharplemmag}. 
This is the purpose of the next lemma.

%TODO: couldn't get typesetting right in second line-KZ
\bl\label{calcsg}
Assuming (H1)--(H3) and (D1)--(D3), we have the identities
$$
\begin{aligned}
\frac{1}{\ks}\bp\displaystyle\langle f^{\rm p}_{MM}\rangle-i\sum_{l=1}^n\langle\d_\xi\tilde q_l(0),\d_x f^{\rm p}_{MM}\rangle e_l\\
\ks\langle\tilde q_{n+1}(0),\d_x f^{\rm p}_{MM}\rangle\ep&=\
\bp  -\frac12\d_M^2F_{|(\Ms,\ks)}+\d_Mc_{|(\Ms,\ks)}\\\frac12\d_M^2\omega_{|(\Ms,\ks)}\ep,\\
\frac{1}{\ks}\bp\displaystyle\langle f^{\rm p}_{kM}\rangle-i\sum_{l=1}^n\langle\d_\xi\tilde q_l(0),\d_x f^{\rm p}_{kM}\rangle e_l\\
\ks\langle\tilde q_{n+1}(0),\d_x f^{\rm p}_{kM}\rangle\ep&-A_*\bp\langle g^{\rm p}_{kM}\rangle\\0\ep=\\
%\qquad\qquad
&\bp  -\d_{kM}^2F_{|(\Ms,\ks)}+\d_kc_{|(\Ms,\ks)}
\Id-\frac{1}{\ks}\left(\d_MF_{|(\Ms,\ks)}-\cs\Id\right)\\\ \d_{kM}^2\omega_{|(\Ms,\ks)}\ep,\\
\frac{1}{\ks}\bp\displaystyle\langle f^{\rm p}_{kk}\rangle-i\sum_{l=1}^n\langle\d_\xi\tilde q_l(0),\d_x f^{\rm p}_{kk}\rangle e_l\\
\ks\langle\tilde q_{n+1}(0),\d_x f^{\rm p}_{kk}\rangle\ep&-A_*\bp\langle g^{\rm p}_{kk}\rangle\\0\ep=\
\bp  -\frac12\d_k^2F_{|(\Ms,\ks)}-\frac{1}{\ks}\d_kF_{|(\Ms,\ks)}\\\frac12\d_k^2\omega_{|(\Ms,\ks)}\ep.\\
\end{aligned}
$$
\el

\begin{proof}
Simple means are computed directly, for example
$$
A_*\bp\langle g^{\rm p}_{kM}\rangle\\0\ep=\
\frac{1}{\ks}\bp\d_MF_{|(\Ms,\ks)}-\cs\Id\\\ks\d_Mc_{|(\Ms,\ks)}\ep.
$$
To compute the remaining terms, we first use the identity $\Id=\sum_{j=1}^{n+1}V_j\tilde V_j^T$ and come back to \eqref{rel_calc}. 
Using the algebraic identities \eqref{alg_order1}(iii), \eqref{alg_order2_MM}, \eqref{alg_order2_kM} and \eqref{alg_order2_kk},
then, we find that
$$
\begin{aligned}
\d_xf^{\rm p}_{MM}&=\ \frac12\bar U'\ \d_M^2\omega_{|(\Ms,\ks)}\ -\ \frac12L_0\d_M^2U_{|(\Ms,\ks)},\\
\d_xf^{\rm p}_{kM}&=\ -\bar U'\ \ks\d_{kM}^2c_{|(\Ms,\ks)}\ -\ L_0\d_{kM}^2U_{|(\Ms,\ks)},\\
\d_xf^{\rm p}_{kk}&=\ \frac12\bar U'\ \d_k^2\omega_{|(\Ms,\ks)}\ -\ \frac12L_0\d_M^2U_{|(\Ms,\ks)}\ -\ \frac{1}{\ks}L_0\d_kU_{|(\Ms,\ks)}.
\end{aligned}
$$
Now, observe that, for $1\leq j\leq n+1$, expanding the fact that for, $|\xi|\leq\xi_0$, $\tilde\phi_j(\xi)$ is a left eigenfunction of $L_\xi$ associated to $\lambda_j(\xi)$ yields for any periodic $g$ the identity
$$
i\langle\d_\xi\tilde\phi_j(0),L_0 g\rangle\ =\
\ks\langle\tilde\phi_j(0),L^{(1)} g\rangle+a_j\ks\langle\tilde\phi_j(0),g\rangle
\ =\ \ks\tilde V_j\cdot\bp\langle L^{(1)}g\rangle\\0\ep+\ks\tilde V_j \cdot A_*\bp\langle g\rangle\\0\ep .
$$
Since, for $1\leq j\leq n+1$,
$$
i\langle\d_\xi\tilde\phi_j(0),\bar U'\rangle\ =\ -\ks,
$$
the proof of the lemma is then achieved by simple direct computations.
\end{proof}

\medskip

We have now in hand all the pieces needed to achieve the proof of Theorem~\ref{main}.

\begin{proof}[Proof of Theorem \ref{main}]
Estimates \eqref{mainest} and \eqref{last} already follow from Proposition \ref{stepg} together with Lemma \ref{switchlem}.

Using Duhamel's principle we may write \eqref{quad_whitham} as
$$
\bp M_W\\ k_W\ep(t)=\Sigma(t)
\bp d_0+m(h_0)\\\ks\partial_x h_0 \ep + \int_0^t \Sigma (t-s)
\partial_x \Big(\frac12 \bp M_W\\ k_W\ep^T \Gamma_* \bp M_W\\ k_W\ep \Big)(s) ds\ ,
$$
where $m(h_0):=-(\bar U-\Ms)\d_xh_0$ and $\Sigma$ is the constant-coefficient solution operator defined in Lemma~\ref{lincong}.
On the other hand, using Propositions \ref{greenbdsg}, \ref{modpropg} and \ref{heatcompg}, and Lemmas \ref{sharplemmag}, \ref{lincong} and~\ref{calcsg}, we find that 
$$
\bp M(t)\\\ks\psi_x(t)\ep
=
\Sigma(t)\bp d_0+m(h_0)\\\ks\partial_x h_0\ep + \int_0^t \Sigma (t-s)
\partial_x \Big(\frac12 \bp M\\\ks\psi_x\ep^T \Gamma_* \bp M\\\ks\psi_x\ep \Big)(s) ds
+\tilde r^{\rm p}(t),
$$
where the residual $r^{\rm p}(t)$ satisfies the bound
$$
\|\tilde r^{\rm p}(t)\|_{L^p(\RM)}\lesssim E_0\ln(2+t)\ (1+t)^{-\frac{1}{2}(1-1/p)-\frac{1}{2}},\quad 2\leq p\leq\infty,
$$
where $E_0$ is defined as in Theorem \ref{oldmain}.  
Thus, subtracting and defining $\delta:=(M,\ks\psi_x)-(M_W,k_W)$, we have
$$
\delta(t)=
\int_0^t \Sigma (t-s) \partial_x \Big(\frac12 \delta^T \Gamma_* \bp M\\\ks\psi_x\ep+\frac12 \bp M_W\\k_W\ep^T \Gamma_* \delta\Big)(s)ds +\tilde r^{\rm p}(t).
$$
Letting $\eta>0$ be fixed but arbitrary and defining
$$
\nu(t):=\sup_{p\in[2,\infty]}\sup_{0\le s\le t}\|\delta(s)\|_{L^p(\RM)}(1+s)^{\frac{1}{2}(1-1/p)+\frac12- \eta},
$$
we thus obtain by the standard bounds (see \cite{LZ})
$\|\Sigma(t)\partial_x f\|_{L^p(\RM)}\le Ct^{-\frac12(1/q-1/p)-\frac12}\|f\|_{L^q(\RM)}$ when $1\le q\le p\le\infty$, that, for
any  $2\le p\le\infty$ and all $t>0$,
$$
\begin{aligned}
\|\delta(t)\|_{L^p(\RM)}&\lesssim
E_0\ln(2+t)\ (1+t)^{-\frac{1}{2}(1-1/p)-\frac{1}{2}}
+\int_0^{t/2} (t-s)^{-\frac12(1-1/p)-\frac12}\nu(t)E_0(1+s)^{-1+\eta}ds\\
&\quad + \int_{t/2}^t (t-s)^{-\frac12(1/2-1/p)-\frac12}\nu(t)E_0(1+s)^{-\frac54+\eta}ds\\
&\lesssim
E_0(\nu(t)+1)(1+t)^{-\frac12(1-1/p)-\frac12+\eta},
\end{aligned}
$$
which in turn yields
$
\nu(t)\le C_\eta E_0\left(1+\nu(t)\right).
$
Thus, if $E_0<1/(2C_\eta)$ it follows that $\nu(t)\le 2C_\eta E_0$ for all $t\geq 0$. This provides the needed bounds on $(M-M_W,\ks\psi_x-k_W)$. Using these bounds a simpler computation yields the result for $\psi_W-\psi$.
\end{proof}

\br\label{simplermk}
\textup{
Though the computation of quadratic coupling coefficients is heavy going, we note that already from Proposition \ref{heatcompg} and Lemma \ref{lincong} one may conclude that long-time behavior is governed to leading order by {\it some} ``Whitham-like'' system \eqref{qcons}, with no computation at all, since in the Duhamel formulation the principal part of nonlinear terms factors on the left as $\int_0^t \Sigma(t-s)\d_x(\dots)(s) ds$.
}
\er

\medskip

{\bf Acknowledgement.}
K.Z. thanks Bj\"orn Sandstede for a number of helpful orienting discussions
regarding modulation of periodic reaction-diffusion waves, 
Guido Schneider for bringing to our attention the 
treatment of B\'enard--Marangoni cells in 
%CHANGED-MR
%\cite{Zi},
\cite{HSZ,Zi}, 
and Denis Serre for his interest in the subject of modulation of periodic solutions and his contributions 
through \cite{Se} and private and public communications.
M.J., P.N., and M.R. thank Indiana University, and K.Z. thanks
the \'Ecole Normale Sup\'erieure, Paris,
the University of Paris 13, and the
Foundation Sciences Math\'ematiques de Paris for 
their hospitality during visits in which this work was partially carried out.
%CHANGEDkz: added:
Finally, special thanks to David Lannes for a careful reading of the manuscript
and many helpful suggestions. 

\appendix

\section{Algebraic relations}\label{algebraic}

We record in this appendix some crucial relations obtained by differentiating the profile equations. 
In order to differentiate, we here consider
variable parameters $(M,k)$ rather than fixed values $(\Ms,\ks)$,
imposing dependence implicitly through the profile equations
(denoting $\langle a\rangle:=\int_0^1 a$):
\be\label{alg_profile}
k^2 U''-k(f(U))'+kc\,U'=0,\qquad \langle U\rangle=M.
\ee

We expand $L_\xi=L_0+ik\xi L^{(1)}+(ik\xi)^2 L^{(2)}$ with
\be\label{alg_L}
\begin{array}{rcl}
L_0v&=&k^2 v''-k((d f)(U)\,v)'+kc\,v'\ ;\\
L^{(1)} v&=&2kv'-(d f)(U)\,v+c\,v\ ;\\
L^{(2)} v&=&v\ .
\end{array}
\ee

Then,
by differentiation of \eqref{alg_profile}, we obtain
\be\label{alg_order1}
\begin{array}{lrcl}
L_0\, U'=0\ ,& \langle U'\rangle&=&0\ ;\\[1ex]
L_0\,\d_M U+kU'\,\d_M c=0\ ,&\langle \d_M U\rangle&=&\Id\ ;\\[1ex]
L_0\,\d_k U+kU'\,\d_k c+L^{(1)} U'=0\ ,&\langle \d_k U\rangle&=&0\ .
\end{array}
\ee
%TODO: 
%(MIGUEL- I don't quite understand what you mean to say just below- also I don't
%follow the notation, e.g. $d^2_{k M}$.. is there a typo? please try
%to clarify a bit...  .???-KZ)
Accordingly, with $\omega=-kc$, using $L^{(1)} U'=kU''$,
we have
$$
<d^2 U>=0\ ,
$$
\be\label{alg_order2_MM}
L_0\,\d_M^2U+2k\,(\d_Mc)\,(\d_MU)'-k\,[(d^2f)(U)(\d_MU,\d_MU)]'=-U'\,k\,\d_M^2c\ ,
\ee
\be\label{alg_order2_kM}
\begin{array}{rcl}
L_0\,\d_{kM}^2U&+&k\,(\d_kc)\,(\d_MU)'+k\,(\d_Mc)\,(\d_kU)'-k\,[(d^2f)(U)(\d_kU,\d_MU)]'\\[1ex]
&+&k\,(\d_MU)''=-U'\,k\,\d_{kM}^2c\ ,
\end{array}
\ee
\be\label{alg_order2_kk}
L_0\,\d_k^2U+2k\,(\d_kc)\,(\d_kU)'-k\,[(d^2f)(U)(\d_kU,\d_kU)]'
+2k\,(\d_kU)''=-U'\,k\,\d_k^2c\ .
\ee

\section{The Whitham equations and asymptotic equivalence}\label{W_app}

In this appendix, we explain how to obtain the needed formal averaged modulation 
system for comparison to our analytical description of asymptotic behavior.
This is performed in three steps.
\begin{enumerate}
\item First, we develop a direct WKB-like formal approximation. 
At this stage we obtain a system that may contain harmless irrelevant terms.
\item Next, we use known results about large-time asymptotic behavior of systems of conservation laws about constant states to get a canonical form for the averaged modulation system.
\item 
Finally, we adapt the system taking into account the fact that the analysis of the main part of the paper 
is carried out after an implicit nonlinear change of coordinates.
\end{enumerate}

\subsection{Formal asymptotics}
Though the full nonlinear analysis may be carried out without distinction between linearly coupled and linearly uncoupled cases, the formal derivation of averaged equations involves resolutions of systems of the form $L_0 g=h$ and therefore requires knowledge of the kernel of $L_0$. We are thus compelled to provide two separate derivations.

Besides, there are at least two ways to obtain relevant averaged equations. The first one is to develop a full WKB-type expansion as in \cite{NR1,NR2}, extending the procedure in \cite{Se} to get higher order equations. This method provides the hyperbolic part of the averaged system in a quick way and a nice form. Its main drawback is that it requires a knowledge of the kernel of $L_0$ for all waves close to the wave under study essentially reducing the scope of the method to the nondegenerate case or to a fully degenerate case where $\d_Mc$ would 
%CHANGED-MR
vanish in a neighborhood of the studied wave.\footnote{
A situation that trivially occurs when some symmetry is present, see Remark~\ref{reflectrmk}.}
The second method is designed to study dynamics about a given wave, so that it does not suffer from the same flaws;
moreover, it is closer to our nonlinear analysis,
and yields a semilinear system.

We derive the system for the generic case with the first method and the one for the linearly uncoupled case with the second one. Note that both methods provide averaged systems with diffusion matrices containing terms that are not relevant for our present analysis.

\subsubsection{Generic case}\label{firstway}
To treat the linearly phase-coupled case, we essentially borrow the derivation of \cite{NR2} for the Korteweg-de Vries/Kuramoto-Sivashinsky equation, a model for which linear phase-coupling is a consequence of assumptions (H1)-(H2) and (D3). In the present derivation, we assume that all the waves involved in the slow-modulation description satisfy (H1)-(H2) and (D3) and are linearly phase-coupled.

Since in this derivation there is no reference wave, thus no privileged frame, we go back to the original equation
\be\label{original}
u_t\ +\ f(u)_x\ =\ u_{xx}.
\ee
We are looking for a formal expansion of a solution $u$ of equation \eqref{original} according to the two-scale \emph{ansatz}
\be\label{ansU}
u(x,t)\ =\ U\left(\frac{\Psi(\varepsilon x,\varepsilon t)}{\varepsilon};\varepsilon x,\varepsilon t\right)
\ee
where
\begin{equation}\label{ansWKB}
\displaystyle
U(y,X,T)=\sum_j\varepsilon^j U_j(y;X,T)\quad\textrm{and}\quad\Psi(X,T)=\sum_j\varepsilon^j\Psi_j(X,T)\ ,
\end{equation}
with the functions $U$ and $U_j$ being 1-periodic in the $y$-variable. 
We insert the \emph{ansatz} (\ref{ansU},\ref{ansWKB}) into \eqref{original} and collect terms of the same order in $\varepsilon$.

First this yields, with $\Omega_0=\d_T\phi_0$ and $\kappa_0=\d_X\Psi_0$,
$\Omega_0\,\d_y U_0+\kappa_0\,\d_y(f(U_0))=\kappa_0^2\,\partial_y^2 U_0$,
which is solved by
\be\label{solve_me_0}
\begin{array}{rcl}
\Omega_0(X,T)&=&-k_0(X,T)\ c(\cM_0(X,T),\kappa_0(X,T)),\\[1em]
U_0(y;X,T)&=&\quad U(y;\cM_0(X,T),\kappa_0(X,T)).
\end{array}
\ee
We have disregarded in \eqref{solve_me_0} the possibility of a phase shift dependent on $(X,T)$ since this is already encoded by $\Psi_1$. We will have to rule out similar problems of uniqueness in the following steps. At this stage the compatibility condition $\displaystyle \d_T\d_X\Psi_0=\d_X\d_T\Psi_0$ already yields the first equation of a Whitham's modulation system:
\begin{equation}\label{w1_0}
\displaystyle
\partial_T \kappa_0+\partial_X\left(\kappa_0\,c(\cM_0,\kappa_0)\right)\ =\ 0\ .
\end{equation}

In the rest of the derivation, we will use 
the notations of Proposition \ref{linspecg} and Appendix \ref{algebraic}, 
with the convention that operators act in $y$ and are associated to the wave profile $U(\,\cdot\,;\cM_0(X,T),\kappa_0(X,T))$. To fix some of the uniqueness issues of the \emph{ansatz}, we pick, for any $(M,k)$, $u^{adj}(\,\cdot\,;M,k)$ a generalized zero eigenfunction of $L_0^*$ such that $\langle u^{adj},\d_M U_{|(M,k)}\rangle=0$ and $\langle u^{adj},U'(\,\cdot\,;M,k)\rangle=1$, set $u_0^{adj}(y;X,T)=u^{adj}(\,\cdot\,;\cM_0(X;T),\kappa_0(X,T))$ and add to the \emph{ansatz} the normalizing condition
\be\label{norm_ans}
\langle u_0^{adj},U_j(\,\cdot\,;X,T)\rangle\ =\ 0,\qquad j\neq0.
\ee

The next step of the identification process gives, with $\Omega_1=\d_T\Psi_1$ and $\kappa_1=\d_X\Psi_1$,
\begin{equation}\label{me_1}
\begin{array}{rcl}
\displaystyle
(\Omega_1+c(\cM_0,\kappa_0)\kappa_1)\d_yU_0&-&
\displaystyle
\kappa_1\,L^{(1)}\d_yU_0\ -\ L_0 U_1\ -\ L^{(1)}\d_X U_0\\[1em]
&-&
\displaystyle
\d_X\kappa_0\,L^{(2)}\d_yU_0\ +\ \d_T U_0+c(\cM_0,\kappa_0)\d_X U_0\ =\ 0,
\end{array}
\end{equation}
whose solvability condition reads
\begin{equation}\label{w2_0}
\displaystyle
\d_T \cM_0+\d_X(F(\cM_0,\kappa_0))\ =\ 0,
\end{equation}
where $F$ denotes the averaged flux $F(M,k)\ =\ \langle f(U(\,\cdot\,;M,k))\rangle$.
To proceed, for arbitrary $(M,k)$ we introduce $g^{k}(\,\cdot\,;M,k)$, $g^{M}(\,\cdot\,;M,k)$ solutions of
\begin{eqnarray}\label{g_k}
L_0(g^{k}(\,\cdot\,;M,k))&=&-\ L^{(1)}\d_k U_{|(M,k)}-\d_k F_{|(M,k)}\ -\ L^{(2)}U'_{|(M,k)}\\[1ex]
&&-\ \d_k U_{|(M,k)}\ k\,\d_k c_{|(M,k)}\ -\ (\d_M U_{|(M,k)}-\Id)\ \d_k F_{|(M,k)}\nonumber\\[1em]
\label{g_M}
\quad L_0(g^{M}(\,\cdot\,;M,k))&=&-L^{(1)}\d_M U_{|(M,k)}-\d_M F_{|(M,k)}+c_{|(M,k)}\Id\\[1ex]
&&-\d_k U_{|(M,k)}\ k\,\d_M c_{|(M,k)}-(\d_M U_{|(M,k)}-\Id)[\d_M F_{|(M,k)}-c_{|(M,k)}\Id]\nonumber
\end{eqnarray}
orthogonal to $u^{adj}(\,\cdot\,;M,k)$ and set $g=\bp g^M&g^k\ep$ and $g_0(\,\cdot\,;X,T)=g(\,\cdot\,;\cM_0(X,T),\kappa_0(X,T))$. Then with \eqref{w1_0}-\eqref{w2_0} and \eqref{alg_order1} equation \eqref{me_1} reads
\begin{eqnarray}\label{me+1}
\displaystyle
L_0\ \Big(U_1
&-&d U_{|(\cM_0,\kappa_0)}\,(\,\cdot\,;\tilde \cM_1,\kappa_1)
-g_0\ \bp\d_X\cM_0\\\d_X\kappa_0\ep\Big)\\
&=&(\Omega_1\,+\ \kappa_0d c_{|(\cM_0,\kappa_0)}\,(\tilde \cM_1,\kappa_1)+c_{|(\cM_0,\kappa_0)}\,\kappa_1)\,\d_yU_0 \nonumber
\end{eqnarray}
for any choice of $\tilde \cM_1$. Let us set $\cM_1=\langle U_1\rangle$. Choosing $\tilde \cM_1$ to get
\begin{equation}\label{Om1}
\displaystyle
\Omega_1+\kappa_0\ d c(\cM_0,\kappa_0)\,[\tilde \cM_1,\kappa_1]+c(\cM_0,\kappa_0)\,\kappa_1=\ 0
\end{equation}
and normalizing the parametrization, as in Lemma \ref{keylemg}, to get, for any $(M,k)$,
\be\label{recall_norm_param}
\langle u^{adj}(\,\cdot\,;M,k),\d_k U_{|(M,k)}\rangle=0,
\ee
equation \eqref{me+1} is reduced to
\begin{eqnarray}\nonumber
\displaystyle
U_1&=&
\displaystyle
d U_{|(\cM_0,\kappa_0)}\,(\tilde \cM_1,\kappa_1)\ +\ g_0\ \bp\d_X\cM_0\\\d_X\kappa_0\ep\\
\nonumber
\displaystyle
\cM_1&=&
\displaystyle
\tilde \cM_1\ +\ \langle g_0\rangle\ \bp\d_X\cM_0\\\d_X\kappa_0\ep.
\end{eqnarray}
Then, compatibility condition $\d_T\kappa_1=\d_X\Omega_1$ yields
\begin{equation}\label{w1_1}
\begin{aligned}
\d_T\kappa_1+\ \d_X(\kappa_0d c_{|(\cM_0,\kappa_0)}\,[\cM_1,\kappa_1]+
&c_{|(\cM_0,\kappa_0)}\,\kappa_1)\ =\\
&\quad  \d_X\left(\kappa_0\d_Mc_{|(\cM_0,\kappa_0)}\ \langle g_0\rangle\ \bp\d_X\cM_0\\\d_X\kappa_0\ep\right).
\end{aligned}
\end{equation}

Returning to the identification process, we obtain an equation of the form
$$
\d_T U_1\ +\ \d_X(df(U_0)\,U_1)\ -\ \d_X^2U_0\ -\ L_0U_2\ +\ \d_y(\ \cdots\ )\ =\ 0,
$$
whose solvability condition is
\begin{equation}\label{w2_1}
\begin{array}{rcl}
\displaystyle
\d_T\cM_1\ +\ \d_X(d F_{|(\cM_0,\kappa_0)}\,[\cM_1,\kappa_1])&=&
\displaystyle
\d_X^2\cM_0\ -\  \d_X\left(\langle df(U_0)\,g_0\rangle\ \bp\d_X\cM_0\\\d_X\kappa_0\ep\right)\\[1ex]
&&\displaystyle
+\ \d_X\left(\d_MF_{|(\cM_0,\kappa_0)}\,\langle g_0\rangle\ \bp\d_X\cM_0\\\d_X\kappa_0\ep\right).
\end{array}
\end{equation}
To write the second order system in a compact form, let us introduce, for arbitrary $(M,k)$,
$$
\begin{array}{rcl}
\displaystyle
d_{1,1}(M,k)&=&
\displaystyle
\Id-\langle df(U(M,k))\,g^M(M,k)\rangle+\d_MF(M,k)\,\langle g^M(M,k)\rangle\\[1ex]
\displaystyle
d_{1,2}(M,k)&=&
\displaystyle
-\langle df(U(M,k))\,g^k(M,k)\rangle+\d_MF(M,k)\,\langle g^k(M,k)\rangle\\[1ex]
\displaystyle
d_{2,1}(M,k)&=&
\displaystyle
k\d_Mc\,(M,k)\ \langle g^M(M,k)\rangle\\[1ex]
\displaystyle
d_{2,2}(M,k)&=&
\displaystyle
k\d_Mc\,(M,k)\ \langle g^k(M,k)\rangle.\\[1ex]
\end{array}.
$$
With these notations, systems (\ref{w1_0},\ref{w2_0}), (\ref{w1_1},\ref{w2_1}) coincide with the first systems obtained in the formal expansion of a solution $(\cM,\kappa)$ of
\be\label{w12_1}
\left\{
\begin{array}{rclcl}
\displaystyle
\d_t \cM&+&
\displaystyle
\d_x(F(\cM,\kappa))&=&
\displaystyle
\d_x\left(d_{1,1}(\cM,\kappa)\,\d_x\cM+d_{1,2}(\cM,\kappa)\,\d_x\kappa\right)\\[1ex]
\displaystyle
\d_t \kappa&+&
\displaystyle
\d_x(\kappa\,c(\cM,\kappa))&=&
\displaystyle
\d_x\left(d_{2,1}(\cM,\kappa)\,\d_x\cM+d_{2,2}(\cM,\kappa)\,\d_x\kappa\right)
\end{array}\right.
\ee
according to the slow \emph{ansatz}
\be\label{ansLF}
(\cM,\kappa)(x,t)\ =\ \sum_j\varepsilon^j (\cM_j,\kappa_j)(\eps x,\eps t).
\ee

We call system \eqref{w12_1} a \emph{(second-order) Whitham's modulation system}.

\subsubsection{Phase-decoupled case}\label{secondway}
For the phase-decoupled case, we propose an alternative derivation that would also work for the uncoupled case. We pick a wave of parameters $(\Ms,\ks)$ and assume that it satisfies (H1)-(H2) and (D3) and is linearly phase-decoupled.

We again insert the \emph{ansatz} (\ref{ansU},\ref{ansWKB}) into \eqref{original} and collect terms of the same order in $\varepsilon$ but this time we specialize to $(\cM_0,\kappa_0)=(\Ms,\ks)$. We keep \eqref{norm_ans} as \emph{ansatz} normalization and \eqref{recall_norm_param} as parametrization normalization. The first nontrivial equation is with $(\Omega_1,\kappa_1)=(\d_T\Psi_1,\d_X\Psi_1)$
\begin{equation}\label{linme_1}
(\Omega_1+c(\Ms,\ks)\kappa_1)\bar U'-
\displaystyle
\kappa_1\,L^{(1)}\bar U'\ -\ L_0 U_1\ =\ 0
\end{equation}
which may also be written as
$$
\displaystyle
L_0\ \Big(U_1
-d U_{|(\Ms,\ks)}\,(\,\cdot\,;\cM_1,\kappa_1)\Big)
=(\Omega_1\,+\ \ks \d_k c(\Ms,\ks)\,\kappa_1+c(\Ms,\ks)\,\kappa_1)\,\bar U'
$$
for any $\cM_1$. Solvability yields
$$
\displaystyle
\Omega_1\,+\ \ks \d_k c(\Ms,\ks)\,\kappa_1+c(\Ms,\ks)\,\kappa_1\ =\ 0
$$
and with our normalization choices \eqref{linme_1} reduces to
$$
\displaystyle
U_1\ =\ d U_{|(\Ms,\ks)}\,(\cM_1,\kappa_1),\quad \cM_1\ =\ \langle U_1\rangle.
$$
Compatibility condition $\d_T\kappa_1=\d_X\Omega_1$ already gives
\begin{equation}\label{linw1_1}
\displaystyle
\d_T\kappa_1+\ \d_X(\ks \d_k c_{|(\Ms,\ks)}\,\kappa_1+c(\Ms,\ks)\,\kappa_1)\ =\ 0.
\end{equation}

At the next step of the identification, we get with $(\Omega_2,\kappa_2)=(\d_T\Psi_2,\d_X\Psi_2)$
$$
\begin{array}{rcl}
\displaystyle
(\Omega_2+c(\Ms,\ks)\kappa_2)\ \bar U'&-&
\displaystyle
\kappa_2\,L^{(1)}\bar U'\ -\ L_0 U_2\ -\ L^{(1)}\d_X U_1\\[1em]
&-&
\displaystyle
(\kappa_1)^2\,L^{(2)}\bar U''\ -\ \kappa_1\,L^{(1)}U_1'\ -\ \d_X\kappa_1\,L^{(2)}\bar U'\\[1em]
&+&
\displaystyle
(\Omega_1+c(\Ms,\ks)\kappa_1)\,U_1'\ +\ \d_T U_1+c(\Ms,\ks)\d_X U_1\\[1em]
&+&
\displaystyle
\ks\,\d_y\left(\frac12d^2f(\bar U)(U_1,U_1)\right)\ +\ \kappa_1\,\d_y\left(df(\bar U)\right)\,U_1\ =\ 0,
\end{array}
$$
which may also be written
\begin{eqnarray}\label{linme+2}
\displaystyle
L_0\ \Big(U_2
&-&
\displaystyle
d U_{|(\Ms,\ks)}\,(\,\cdot\,;\cM_2,\kappa_2)
-\frac12d^2 U_{|(\Ms,\ks)}\,(\,\cdot\,;(\cM_1,\kappa_1),(\cM_1,\kappa_1))\Big)\\
&=&
\displaystyle
\d_T U_1+c(\Ms,\ks)\d_X U_1-\ \d_X\kappa_1\,L^{(2)}\bar U'-\ L^{(1)}\d_X U_1\nonumber\\
&+&
\displaystyle
\left(\Omega_2\,-\ \d_k \omega(\Ms,\ks)\,\kappa_2\,-\frac12d^2\omega(\Ms,\ks)\,((\cM_1,\kappa_1),(\cM_1,\kappa_1))\right)\,\bar U'\nonumber
\end{eqnarray}
for any $\cM_1$. Solvability then reads
\begin{eqnarray}\label{linw2_1}
\displaystyle
\qquad\d_T\cM_1&+&\d_X(d F_{|(\Ms,\ks)}\,(\cM_1,\kappa_1))\ =\ 0,\\
\label{linOm2}
\displaystyle
\Omega_2&-&\d_k \omega(\Ms,\ks)\,\kappa_2\ -\ \frac12d^2\omega(\Ms,\ks)\,((\cM_1,\kappa_1),(\cM_1,\kappa_1))\\
&=&\d_X\kappa_1+\langle u^{adj}(\Ms,\ks), L^{(1)}\d_kU_{|(\Ms,\ks)}\rangle\d_X\kappa_1+
\langle u^{adj}(\Ms,\ks), L^{(1)}\d_MU_{|(\Ms,\ks)}\rangle\d_X\cM_1.\nonumber
\end{eqnarray}
Note that the latter equation yields
\noindent
\begin{eqnarray}\label{linw1_2}
&&\d_T\kappa_2-\d_X\left(\d_k \omega(\Ms,\ks)\,\kappa_2\ +\ \frac12d^2\omega(\Ms,\ks)\,((\cM_1,\kappa_1),(\cM_1,\kappa_1))\right)\\
&=&\d^2_X\kappa_1+\d_X\left(\langle u^{adj}(\Ms,\ks), L^{(1)}\d_kU_{|(\Ms,\ks)}\rangle\d_X\kappa_1+
\langle u^{adj}(\Ms,\ks), L^{(1)}\d_MU_{|(\Ms,\ks)}\rangle\d_X\cM_1\right).\nonumber
\end{eqnarray}

To proceed, we introduce $\tilde g^{k}$, $\tilde g^{M}$, the solutions of
\begin{eqnarray}\nonumber
L_0\ \tilde g^{k}&=&-\ L^{(1)}\d_k U_{|(\Ms,\ks)}-\d_M U_{|(\Ms,\ks)}\d_k F_{|(\Ms,\ks)}\\[1ex]
&&+\ \bar U'\langle u^{adj}(\Ms,\ks), L^{(1)}\d_kU_{|(\Ms,\ks)}\rangle
-\d_k U_{|(\Ms,\ks)}\ \ks\,\d_k c_{|(\Ms,\ks)},\nonumber\\[1em]
\nonumber
\quad L_0\ \tilde g^{M}&=&-L^{(1)}\d_M U_{|(\Ms,\ks)}-\d_M U_{|(\Ms,\ks)}(\d_M F_{|(\Ms,\ks)}-c_{|(\Ms,\ks)}\Id)\\[1ex]
&&+\ \bar U'\langle u^{adj}(\Ms,\ks), L^{(1)}\d_MU_{|(\Ms,\ks)}\rangle,\nonumber
\end{eqnarray}
mean free and orthogonal to $u^{adj}(\,\cdot\,;\Ms,\ks)$ and set $\tilde g=\bp \tilde g^M&\tilde g^k\ep$.
With \eqref{linw2_1} and \eqref{linOm2}, setting $\cM_2=\langle U_2\rangle$, equation \eqref{linme+2} becomes
$$
\displaystyle
U_2\ =\
d U_{|(\Ms,\ks)}\,(\cM_2,\kappa_2)\ +\ \frac12d^2 U_{|(\Ms,\ks)}\,(\,\cdot\,;(\cM_1,\kappa_1),(\cM_1,\kappa_1))
\ +\ \tilde g\ \bp\d_X\cM_1\\\d_X\kappa_1\ep.
$$

Finally, substituting (\ref{ansU},\ref{ansWKB}) into \eqref{original}, 
and comparing terms of order $\eps^3$, we obtain an equation of the form
$$
\d_T U_2\ +\ \d_X\left(df(\bar U)\,U_2+\frac12d^2f(\bar U)(U_1,U_1)\right)\ -\ \d_X^2U_1\ -\ L_0U_3\ +\ \d_y(\ \cdots\ )\ =\ 0,
$$
whose solvability implies
\begin{eqnarray}\label{linw2_2}
\displaystyle
\d_T\cM_2&+&\d_X\left(d F_{|(\Ms,\ks)}\,(\cM_2,\kappa_2)+\frac12d^2 F_{|(\Ms,\ks)}\,[(\cM_1,\kappa_1),(\cM_1,\kappa_1)]\right)\\
\displaystyle
&=&\d_X^2\cM_1\ -\  \d_X\left(\langle df(\bar U)\,\tilde g\rangle\ \bp\d_X\cM_1\\\d_X\kappa_1\ep\right)\nonumber.
\end{eqnarray}

To write the second order system in a compact form, let us introduce, for arbitrary $(M,k)$,
$$
\begin{array}{rclrcl}
\displaystyle
\tilde d_{1,1}&=&
\displaystyle
\Id-\langle df(U(M,k))\,\tilde g^M\rangle,&
\displaystyle
\tilde d_{1,2}&=&
\displaystyle
-\langle df(U(M,k))\,g^k\rangle\\[1ex]
\displaystyle
\tilde d_{2,1}&=&
\displaystyle
\langle u^{adj}(\Ms,\ks), L^{(1)}\d_MU_{|(\Ms,\ks)}\rangle,&
\displaystyle
\tilde d_{2,2}&=&
\displaystyle
1+\langle u^{adj}(\Ms,\ks), L^{(1)}\d_kU_{|(\Ms,\ks)}\rangle\\[1ex]
\end{array}.
$$
With these notations, systems (\ref{linw1_1},\ref{linw2_1}), (\ref{linw1_2},\ref{linw2_2}) coincide with the first nontrivial systems obtained in the formal expansion of a solution $(\cM,\kappa)$ of
\be\label{linw12_1}
\left\{
\begin{array}{rclcl}
\displaystyle
\d_t \cM&+&
\displaystyle
\d_x(F(\cM,\kappa))&=&
\displaystyle
\d_x\left(\tilde d_{1,1}\,\d_x\cM+\tilde d_{1,2}\,\d_x\kappa\right)\\[1ex]
\displaystyle
\d_t \kappa&+&
\displaystyle
\d_x(\kappa\,c(\cM,\kappa))&=&
\displaystyle
\d_x\left(\tilde d_{2,1}\,\d_x\cM+\tilde d_{2,2}\,\d_x\kappa\right)
\end{array}\right.
\ee
according to the slow \emph{ansatz}
$$
(\cM,\kappa)(x,t)\ =\ \sum_j\varepsilon^j (\cM_j,\kappa_j)(\eps x,\eps t),\qquad (\cM_0,\kappa_0)\ =\ (\Ms,\ks).
$$

We call system \eqref{linw12_1}, likewise,
a \emph{(second-order) Whitham's modulation system}.

\medskip

As should be clear from the formal derivations, there is some freedom in the choice of the diffusion matrices. This reflects the fact that many systems of conservation laws share the same asymptotic behavior about constant states. We recall next how to classify these systems according to their asymptotic behavior;
this will provide a canonical modulation system for our nonlinear analysis.

\subsection{Asymptotic equivalence of systems of conservation laws}

\subsubsection{General theory}

We now recall the notion of asymptotic equivalence and behavior of solutions of systems of conservation laws near a constant state,
useful in our context since, being able to prove modulational behavior, we 
reduce the dynamics about a periodic wave to motion of parameters 
near a constant state. Given a general system of conservation laws
\be\label{gcons}
w_t +(g(w))_x= ( B(w)w_x)_x
\ee
and a reference state $w_*$ at which $dg(w_*)$ has distinct eigenvalues, so
that $L_* dg(w_*)R_*$ is diagonal for some $L_*=\bp l_1^*\\ \vdots\\l_n^*\ep$,
$R_*=\bp r_1^* & \dots  & r_n^*\ep$, $L_*R_*=\Id$,
define the {\it quadratic approximant}
\be\label{qcons}
y_t +A_* y_x +\frac{1}{2}(y^t\Gamma_* y)_x = \tilde B_* y_{xx},
\ee
and the {\it decoupled quadratic approximant}
\be\label{dqcons}
z_t +A_* z_x +\frac{1}{2}(z^t\tilde \Gamma_* z)_x = \tilde B_* z_{xx},
\ee
about $w_*$, where
\be\label{q}
A_*=dg(w_*),\qquad
\Gamma_*:=  d^2g(w_*),
\qquad\textrm{and}\qquad
B_*:=B(w_*),
\ee
\be\label{dq}
\tilde \Gamma_*:= L_*^t\diag\{ R_*^t\Gamma_*R_*\}L_*,
\qquad\textrm{and}\qquad
\tilde B_*:= R_*\diag \{ L_* B_*R_*\}L_*.\footnote{
Here and elsewhere, we identify as usual bilinear maps with vector-valued matrices, and in particular $d^2g(w_*)$ with $Hess(g)(w_*)$. Moreover, for these vector-valued matrices, we use coordinate notations so that for instance $\Gamma_*^j\in \RM^{n\times n}$ satisfies
$w^t\Gamma_*^jw=(w^t\Gamma_*w)_j=d^2g_j(w,w)$.}
\ee

Assume 
%CHANGED-MR
%that 
%
the parabolicity condition, 
$\diag \{ L_* B(w_*)R_*\}$ is positive, and
define the self-similar nonlinear (resp. linear if $\gamma_j=0$)
diffusion waves
$ \theta_j(x,t)=t^{-1/2}\bar \theta_j(x/\sqrt{t})$
to be the solutions of the Burgers equations (resp. heat equations if
$\gamma_j=0$)
\be\label{burgers}
\theta_t+ \frac{1}{2}(\gamma_j^* \theta^2)_x=\theta_{xx},
\quad
\gamma_j^*:=[l_j^* (r_j^*)^t\Gamma_* r_j^*]/[l_j^*B_*r_j^*],
\ee
with
delta-function initial data
$l_j^* m_0\ \delta(\cdot)$,
where
$m_0:=\int z_0(x)dx$.
Then, we have the following fundamental result describing behavior
of \eqref{gcons}--\eqref{dqcons} with respect to localized initial perturbations.

\begin{proposition}[ \cite{Ka,LZ} ]\label{cl_lemma}
Let $\eta>0$. Let $w$ and $z$ be solutions of \eqref{gcons} and
\eqref{dqcons} with initial data $w_0$ and $z_0=w_0-w_*$ such that 
$E_1:=\|z_0\|_{L^1(\RM)\cap H^4(\RM)}+\||\cdot|z_0\|_{L^1(\RM)}$ is sufficiently small.
Then, for $1\leq p\leq\infty$, $m_0:=\int_\RM z_0$, and $\theta_j$ as in \eqref{burgers}, 
\be\label{aeq1}
\|w(t)-w_*-z(t)\|_{L^p(\RM)}\lesssim E_1(1+t)^{-\frac{1}{2}(1-1/p)-\frac{1}{4}+\eta}
 ;
\ee
and
\be\label{aeq2}
\|z(t)-\sum_j \theta_j(\cdot -a_j^*(1+t),b_j^* (1+t))\ r_j^*\|_{L^p(\RM)}\lesssim E_1(1+t)^{-\frac{1}{2}(1-1/p)-\frac{1}{4}+\eta} ,
\ee
with $a_j^*:=l_j^*A_*r_j^*$, $b_j^*:=l_j^*B_*r_j^*$, whence\footnote{Computing
$\|\theta_j(t)\|_{L^p(\RM)}=
t^{-1/2}\|\bar \theta_j (\cdot/\sqrt{t})\|_{L^p(\RM)}\sim
t^{-\frac{1}{2}(1-1/p)}$.}, if $\eta<1/4$,
$$
\|w(t)-w_*\|_{L^p(\RM)},\, \|z(t)\|_{L^p(\RM)}
\gtrsim |m_0|\,(1+t)^{-\frac{1}{2}(1-1/p)}\ .
$$
\end{proposition}

Proposition \ref{cl_lemma} asserts that \eqref{gcons} and \eqref{dqcons}
(hence also \eqref{qcons})
are {\it asymptotically equivalent} with respect to small localized initial
data $w_0-w_*=z_0\in L^1(\RM,(1+|x|)dx)\cap H^3(\RM)$, in the sense that the difference
between solutions $z(t)$ and $w(t)-w_*$ decays at rate
$(1+t)^{-\frac{1}{2}(1-1/p)-\frac{1}{4}+\eta}$ approximately
$(1+t)^{-\frac{1}{4}}$ faster
than the (Gaussian) rate $|m_0|\,(1+t)^{-\frac{1}{2}(1-1/p)}$
at which either one typically (i.e., for data with small
$L^1$ first moment) decays.
Moreover, through \eqref{aeq2}, it gives a simple description of
asymptotic behavior as the {\it linear superposition of scalar diffusion
waves} $\theta_j$ moving with characteristic speeds (eigenvalues $a_j^*$)
in the characteristic modes (eigendirections $r_j^*$) of
$dg(w_*)$, satisfying Burgers equations \eqref{burgers}.

We have also the following more elementary result comparing
to the full quadratic approximant.

\begin{proposition}[\cite{JNRZ2}\footnote{
Though stated in \cite[Lemma 1.2]{JNRZ2} for scalar equations,
the proof applies equally to the system case;
see Appendix~\ref{cl_proof}.}]\label{scalar_lemma}
Let $\eta>0$. Let $w$ and $y$ be solutions of \eqref{gcons} and
\eqref{qcons} with initial data $w_0$ and $y_0=w_0-w_*$ such that $E_0:=\|y_0\|_{L^1(\RM)\cap H^4(\RM)}$ is sufficiently small.
Then,\\ for $1\leq p\leq\infty$,
$$
\|(w-w_*-y)(t)\|_{L^p(\RM)}\lesssim
E_0 (1+t)^{-\frac{1}{2}(1-1/p)-\frac{1}{2}+\eta}.
$$
\end{proposition}

An important consequence of Proposition \ref{scalar_lemma} is that only the
quadratic order quantities appearing in \eqref{qcons}
need be taken into account in the study of asymptotic behavior
of \eqref{gcons} to the order of approximation
considered in Theorem \ref{main}.
Finally, we note the following result following from a proof similar
to but much simpler than the one for Proposition \ref{scalar_lemma}
given in \cite[Appendix A]{JNRZ2}.

\bl\label{special_lemma}
Let $k$ satisfy $k(0)=0$ and
\be\label{specialeq}
k_t + ak_x + (\gamma k^2)_x- dk_{xx}= (Fk)_x,
\ee
where $a,\gamma, d$ are constant, $d>0$ and $F$ is a given function such that $\|F(t)\|_{L^2(\RM)}\le E_0(1+t)^{-\frac14}$.
Then, for any $\eta>0$, provided $E_0$ is small enough, for $1\le p\le \infty$,
$$
\|k(t)\|_{L^p(\RM)}\lesssim E_0(1+t)^{-\frac{1}{2}(1-1/p)-\frac{1}{2}+\eta}.
$$
\el

For the sake of completeness, we recall the proof of the previous Proposition in Appendix~\ref{cl_proof}.

\subsubsection{A first application}\label{s:pfCorphasethm}
As an immediate application, we may now establish the improved decay bounds \eqref{quadimp}--\eqref{linimp} of
Corollary \ref{phasethm}. We will use these tools again in establishing \eqref{refinedest_comp}.

\begin{proof}[Proof of Corollary \ref{phasethm}]
Bound \eqref{linimp} follows from the assumption $\ks\partial_xh_0=0$. For, a solution $(M_W,k_W)$, with an initial data $(*,0)$, of the decoupled approximating equations \eqref{dqcons} to \eqref{whitham} satisfies $k_W(t)\equiv0$, since the $k$ equation decouples in \eqref{dqcons} for the linearly phase-decoupled case. Comparing to the actual solution of \eqref{whitham} using \eqref{aeq1}, we obtain the result.
Bound \eqref{quadimp} goes similarly, observing that in the quadratically decoupled case, the $k$ equation in the full quadratic approximating
system \eqref{qcons} to \eqref{whitham}, though it does not completely decouple, is of the form \eqref{specialeq} with $F=\cO(M)$.
\end{proof}

\br\label{Liu_explanation}
\textup{
Analogous to \eqref{specialeq} in the quadratically decoupled case,
the rate-determining bound in the linearly decoupled case of
Proposition \ref{cl_lemma} is the key estimate
$$
\|k(t)\|_{L^p}\lesssim E_0(1+t)^{-\frac{1}{2}(1-1/p)-\frac{1}{4}}
$$
established by Liu \cite{L} for quadratic coupling terms involving different modes, thus obeying $k(0)=0$,
$k_t + ak_x + (\gamma k^2)_x- dk_{xx}= (\tilde\theta^2)_x$, where $\tilde\theta(x,t)=\theta(x-\tilde at,\tilde b t)$ with $\tilde a\ne a$ and $\theta$ a self similar solution of a Burgers equation \eqref{burgers}.
The anomalous rate $(1+t)^{\frac14}$ is different from the powers of $(1+t)^{\frac12}$ arising in scalar convection--diffusion
processes, reflecting the additional complications present in the system case.
}
\er

\subsubsection{Quadratic approximants of modulation systems}
For later reference, let us write, in the original frame (and not the co-moving one), as
\be\label{quad_whitham_exp}
\d_t\bp M\\ k\ep\ +\ \d_x\bp dF_{|(\Ms,\ks)}(M,k)\\d\omega_{|(\Ms,\ks)}(M,k)\ep\
+\ \frac12\d_x\bp d^2F_{|(\Ms,\ks)}(M,k)\\d^2\omega_{|(\Ms,\ks)}((M,k),(M,k))\ep\ =\ \tilde B_*\d_x^2\bp M\\k\ep
\ee
the quadratic approximant of \eqref{whitham} (obtained as \eqref{w12_1} and \eqref{linw12_1} above). As pointed out in Remark~\ref{diffrmk}, it follows from Lemma~\ref{dlemg} that this system is independent of the choices made in the course of the formal derivation.

From the general theory, we know that instead of comparing $(\cM,\kappa)$ in Theorem~\ref{main} to a solution $(\cM_W,\kappa_W)$ of \eqref{whitham}, we only need to compare it with $(\Ms,\ks)+(M_W,k_W)$ with $(M_W,k_W)$ a solution of \eqref{quad_whitham_exp} expressed in the co-moving frame.

\subsection{Implicit change of variables}\label{s:implicitchange}
Our nonlinear analysis begins with an implicit nonlinear change of variable \eqref{pertvar}. We explain now how the modulation system is affected by this change of variables. We could have first performed this implicit change of variables then carried out the formal modulation process, but we find more enlightening to change the system \emph{a posteriori}.

Since our diffeomorphism is close to identity, only nonlinear terms should be changed, and from the asymptotic equivalence theory we know that nonlinear terms are relevant only in the hyperbolic part. Therefore it is enough to investigate how (\ref{w1_0},\ref{w2_0}) is altered. Let us introduce $\Phi_0$ such that $\Phi_0(\Psi_0(X,T),T)=X$. Recall that $\d_T\Psi_0=\omega(\cM_0,\d_X\Psi_0)$. Therefore if $A,B$ are such that $\d_T A+\d_X B=0$ then $(\tilde A,\tilde B)(X,T)=(A,B)(\Phi_0(X,T),T)$ implies
$$
\d_T\tilde A\ -\ \frac{\d_T\Phi_0}{\d_X\Phi_0}\ \d_X\tilde A\ +\ \frac{1}{\d_X\Phi_0}\ \d_X\tilde B\ =\ 0
$$
also written $\d_T\left(\d_X\Phi_0\tilde A\right)+\d_X\left(\tilde B-\d_T\Phi_0\tilde A\right)=0$ or
$$
\d_T\left(\d_X\Phi_0\tilde A\right)\ +\ \d_X\left(\tilde B-c\left(\tilde\cM_0,\frac{1}{\d_X\Phi_0}\right)\tilde A\right)\ =\ 0
$$
with $\tilde \cM_0(X,T)=\cM_0(\Phi_0(X,T),T)$. Note that this kind of manipulation is completely similar to the ones needed to perform usual Lagrangian change of coordinates and of course closely related to the computations involved in the proof of Lemma \ref{lem:canest}. As expected, applying this to \eqref{w1_0} leads to a trivial equation while an application on the trivial equation $\d_T(1)+\d_x(0)=0$ gives
$$
\d_T\left(\d_X\Phi_0\right)\ -\ \d_X\left(c\left(\tilde\cM_0,\frac{1}{\d_X\Phi_0}\right)\right)\ =\ 0.
$$
Equation \eqref{w1_0} is changed into
$$
\d_T\left(\d_X\Phi_0\tilde \cM_0\right)\ +\ \d_X\left(F\left(\tilde\cM_0,\frac{1}{\d_X\Phi_0}\right)-c\left(\tilde\cM_0,\frac{1}{\d_X\Phi_0}\right)\tilde \cM_0\right)\ =\ 0.
$$

At the hyperbolic level, we are thus lead to the system
$$
\begin{array}{rcl}
\d_T p&-&\d_X\left(c\left(\frac{\cM}{p},\frac1p\right)\right)\ =\ 0\\
\d_T \cM&+&\d_X\left(F\left(\frac{\cM}{p},\frac1p\right)-c\left(\frac{\cM}{p},\frac1p\right)\frac{\cM}{p}\right)\ =\ 0
\end{array}
$$
whose quadratic expansion in
$$
(p,\cM)\ =\ \left(\frac{1}{\ks},\frac{\Ms}{\ks}\right)+\left(\frac{-k}{\ks}\frac{1}{\ks},\frac{-k}{\ks}\frac{\Ms}{\ks}+\frac{M}{\ks}\right)
$$
gives
$$
\begin{array}{rcl}
\d_T k-\ks\,c(\Ms,\ks)\d_Xk&-&\ks\d_X\left(d\omega_{|(\Ms,\ks)}(M,k)+\frac12d^2\omega_{|(\Ms,\ks)}((M,k),(M,k))\right)\ =\ 0\\
\d_T M-\ks\,c(\Ms,\ks)\d_XM&+&\d_X\left(dF_{|(\Ms,\ks)}(M,k)+\frac12d^2F_{|(\Ms,\ks)}((M,k),(M,k))\right)\\
&+&\ks\d_X\left(\dfrac{k}{\ks}\left(dF_{|(\Ms,\ks)}(M,k)-c(\Ms,\ks)M\right)-dc_{|(\Ms,\ks)}(M,k)M\right)\ =\ 0.
\end{array}
$$
Two main comments are in order: 
1. We end up naturally with equations expressed in a co-moving frame thus no further change is needed. 
2. The wavenumber equation remains unaltered at this level of description. 
This explains why the fact that the implicit change of variables could change the modulation equations was not revealed by previous studies \cite{JNRZ2,SSSU} focusing on situations where no other wave parameter is involved,

\medskip

\br
\textup{
Though we do not need it for the present semilinear analysis, let us describe for the sake of generality what would happen for a full quasilinear parabolic system. For
$$
\d_t\bp\cM\\\Psi_x\ep\ +\ \d_x A(\cM,\Psi_x)\ =\ \d_x\left(D(\cM,\Psi_x)\d_x\bp\cM\\\Psi_x\ep\right)
$$
with
$$
\d_t\Psi\ +\ A_{n+1}(\cM,\Psi_x)\ =\ D_{n+1}(\cM,\Psi_x)\d_x\bp\cM\\\Psi_x\ep
$$
where $A_{n+1}=e_{n+1}\cdot A$, $D_{n+1}=e_{n+1}\cdot D$, the transformation $\Phi=\Psi^{-1}$, $\tilde\cM=\cM\circ\Phi$ leads 
to
$$
\d_t\Phi\ -\ \Phi_x\,A_{n+1}(\tilde\cM,1/\Phi_x)\ =\ -D_{n+1}(\tilde\cM,1/\Phi_x)\d_x\bp\tilde\cM\\1/\Phi_x\ep
$$
with
$$
\begin{array}{rcl}
\d_t\tilde\cM&+&\d_x \left(A_\perp(\tilde\cM,1/\Phi_x)-\d_x\Phi\,A_{n+1}(\tilde\cM,1/\Phi_x)\,\tilde\cM\right)\\
&=&\d_x\left(\frac{1}{\d_x\Phi}\,D_\perp(\tilde\cM,1/\Phi_x)\d_x\bp\tilde\cM\\1/\Phi_x\ep-D_{n+1}(\tilde\cM,1/\Phi_x)\d_x\bp\tilde\cM\\1/\Phi_x\ep\,\tilde\cM\right)\\
\d_t(\d_x\Phi)&-&\d_x\left(\Phi_x\,A_{n+1}(\tilde\cM,1/\Phi_x)\right)\ =\ -\d_x\left(D_{n+1}(\tilde\cM,1/\Phi_x)\d_x\bp\tilde\cM\\1/\Phi_x\ep\right)
\end{array}
$$
where $A_\perp=\bp\Id_{d\times d}&\begin{array}{c}0\\\vdots\\0\end{array}\ep A$, $D_\perp=\bp\Id_{d\times d}&\begin{array}{c}0\\\vdots\\0\end{array}\ep D$.
}
\er

\bigskip

Collecting the results of this appendix, we find that to validate the formal Whitham modulation approximation, 
we only need to compare the couple $(M,\ks\psi_x)$ of Theorem~\ref{main} to a solution $(M_W,k_W)$ of 
$$
\begin{aligned}
\d_t\bp M\\ k\ep&+\ \ks A_*\d_x\bp M\\k\ep\
+\ \frac12\ks\d_x\bp d^2F_{|(\Ms,\ks)}((M,k),(M,k))\\d^2\omega_{|(\Ms,\ks)}((M,k),(M,k))\ep\ -\ \ks^2\tilde B_*\d_x^2\bp M\\k\ep\\
&+\ \ks\d_x\bp\dfrac{k}{\ks}\left(dF_{|(\Ms,\ks)}(M,k)-c(\Ms,\ks)M\right)-dc_{|(\Ms,\ks)}(M,k)M\\0\ep\ =\ 0,
\end{aligned}
$$
where $A_*=\d_{(M,k)} (F-\cs M,-\omega-\cs k)|_{(\Ms,\ks)}$. For writing convenience, we denote this system by
%CHANGED-DL : signs in front og \Gamma_*
\be\label{quad_whitham}
\d_t\bp M\\ k\ep+\ \ks A_*\d_x\bp M\\k\ep\ -\ \d_x \Big(\frac12 \bp M\\ k\ep^T \Gamma_* \bp M\\ k\ep \Big)
\ =\ \ks^2\tilde B_*\d_x^2\bp M\\k\ep.
\ee
Likewise, $\psi$ in Theorem~\ref{main} needs then to be compared with $\psi_W$ a solution of 
\be\label{last_whitham}
\d_t\psi+\ e_{n+1}\cdot A_*\bp M_W\\k_W\ep\ -\ \frac{1}{\ks}e_{n+1}\cdot\Big(\frac12 \bp M_W\\ k_W\ep^T \Gamma_* \bp M_W\\ k_W\ep \Big)
\ =\ \ks e_{n+1}\cdot \tilde B_*\d_x\bp M_W\\k_W\ep.
\ee

\section{Asymptotic equivalence of quadratic approximants}\label{cl_proof}

For completeness, we include here a proof of Proposition \ref{scalar_lemma}
including the treatment of off-diagonal diffusion terms not
arising in the scalar case considered in \cite{JNRZ1}.

\begin{proof}[Proof of Proposition \ref{scalar_lemma}]
{(\it Case $B_*=\tilde B_*$.)}
We first review the case $B_*=\tilde B_*$
treated in \cite{JNRZ1}.
By the general results of \cite{Ka},
provided $E_0:=\|y_0\|_{L^1\cap H^3(\RM)}$ is sufficiently small, we have
for $1\le p\le\infty$
$$
\begin{aligned}
\|w(t)-w_*\|_{L^p(\RM)},\,
\|y(t)\|_{L^p(\RM)} &\lesssim E_0 (1+t)^{-\frac{1}{2}(1-\frac{1}{p})},\\
\|w_x(t)\|_{H^1(\RM)},\,
\|y_x(t)\|_{H^1(\RM)} &\lesssim E_0 (1+t)^{-\frac34}.
\end{aligned}
$$

Setting $\delta:=w_*+y-w$, we have, subtracting and rearranging,
$$
\delta_t  +A_* \delta_x- B_*\delta_{xx}= \partial_x \CalF,\qquad
\CalF=\cO((|w-w_*|+|y|)\delta)
+\cO(|w-w_*|^3)+\cO(|w-w_*| |w_x|),
$$
with $ \delta|_{t=0}=0$ and $A_*$ and $B_*$ as in \eqref{q}--\eqref{dq}.
By Duhamel's formula,
$$
\delta (t)= \int_0^t \sigma(t-s)\d_x\CalF(s) ds,
$$
where $\sigma$ is the solution operator of the
parabolic system of conservation laws $u_t + A_* u_x-B_* u_{xx}=0$.
Applying the standard bounds \cite{LZ}
$\|\sigma(t)\partial_x^r h\|_{L^p(\RM)}\lesssim
t^{-\frac{1}{2}(\frac1q-\frac{1}{p})-\frac{r}{2}}\|h\|_{L^q(\RM)}$,
$1\le q\le p\le \infty$,
together with
$$
\|\CalF(t)\|_{L^q(\RM)}\ \lesssim\
E_0 (1+t)^{-\frac12(1-1/q)-\frac14}(\|\delta(t)\|_{L^2(\RM)}
+\|w_x(t)\|_{L^2(\RM)})+E_0^2(1+t)^{-\frac12(1-1/q)-1},
$$
$1\le q\le 2,$
we find, defining $\nu(t):=
%\sup_{p\in[2,\infty]}
\sup_{0\le s\le t}\|\delta(s)\|_{L^2(\RM)}(1+s)^{\frac12(1-1/p)+\frac12 -\eta}$,
that, for all $1\le p\le \infty$,
$$
\begin{aligned}
\|\delta(t)\|_{L^p(\RM)}&\lesssim
\int_0^{t/2} (t-s)^{-\frac12(1-1/p)-\frac12}\| \CalF (s)\|_{L^1(\RM)} ds\\
&\quad +
\int_{t/2}^t (t-s)^{-\frac12(1/(\min(2,p))-1/p)-\frac12}
\| \CalF (s)\|_{L^{\min(2,p)}(\RM)} ds \\
&\lesssim
\int_0^{t/2} (t-s)^{-\frac12(1-1/p)-\frac12} (\nu(t)E_0+E_0^2)(1+s)^{-1+\eta}ds\\
&+\int_{t/2}^t (t-s)^{-\frac12(1/(\min(2,p))-1/p)-\frac12}(\nu(t)E_0+E_0^2)(1+s)^{-1+\eta-\frac12(1-1/(\min(2,p)))}ds\\
&\lesssim
E_0(E_0 +\nu(t)) (1+t)^{-\frac12(1-1/p)-\frac12 +\eta},
\end{aligned}
$$
whence $ \nu(t)\le C_\eta E_0\left(E_0+\nu(t)\right) $.
This implies that
$\nu(t)\le 2C_\eta E_0^2$ for $E_0<1/(2C_\eta)$, giving
$$
\|\delta(t)\|_{L^p(\RM)}\leq 2C_\eta E_0^2
(1+t)^{-\frac12(1-1/p)-\frac12 +\eta} ,
\qquad 1\le p\le \infty.
$$

{(\it General case.)}
We treat now the general case that $B(w)=B_*+\cO(w-w_*)$ with $B_*$ constant but not equal to $\tilde B_*$.
Defining again $\delta:=w_*+y-w$, and denoting by $\tilde \sigma(t)$
the solution operator of linear system $u_t + A_* u_x-\tilde B_* u_{xx}=0$,
we have by Duhamel's principle
$$
\begin{array}{rcl}
\delta(t)&=&(\tilde \sigma-\sigma)(t) w_0
+\int_0^t  (\tilde \sigma-\sigma)(t-s) \partial_x \cO(|w-w_*|^2)(s)ds\\[1ex]
&&+\int_0^t  \tilde \sigma(t-s) \partial_x \cO(|\delta|(|w-w_*|+|y|))(s)ds\\[1ex]
&&+\int_0^t \sigma(t-s) \partial_x \cO(|w-w_*|^3+|w-w_*||w_x|)(s)ds.
\end{array}
$$
From \cite{Ka}, provided $E_0:=\|y_0\|_{L^1\cap H^4(\RM)}$ is sufficiently small, we have
for $1\le p\le\infty$
$$
\begin{aligned}
\|w(t)-w_*\|_{L^p(\RM)},\,
\|y(t)\|_{L^p(\RM)} &\lesssim E_0 (1+t)^{-\frac{1}{2}(1-\frac{1}{p})},\\
\|w_x(t)\|_{H^2(\RM)},\,
\|y_x(t)\|_{H^2(\RM)} &\lesssim E_0 (1+t)^{-\frac34}.
\end{aligned}
$$
Applying the bounds \cite{Ka,LZ}
$\|(\tilde \sigma-\sigma)(t)\partial_x^r h\|_{L^p(\RM)}\lesssim
t^{-\frac{1}{2}(\frac1q-\frac{1}{p})-\frac{r}{2}}(1+t)^{-\frac12}\|h\|_{L^q(\RM)}+e^{-\theta\,t}\|\d_x^rh\|_{L^p(\RM)}$,
$1\le q\le p\le \infty$ for $r=0,1$ (and some $\theta>0$),
estimating
$$
\begin{aligned}
\|\int_0^t  (\tilde \sigma&-\sigma)(t-s) \partial_x \cO(|w-w_*|^2)(s)ds\|_{L^p(\RM)}\\
&\lesssim
\int_0^{t/2} (t-s)^{-\frac12(1-1/p)-1}\| |w-w_*|^2 (s)\|_{L^1(\RM)} ds+\int_0^t e^{-\theta\,(t-s)}\| |w-w_*|^2(s)\|_{L^p(\RM)} ds\\
&\qquad\qquad+\int_{t/2}^t (t-s)^{-\frac12}(1+t-s)^{-\frac12}\| |w-w_*|^2(s)\|_{L^p(\RM)} ds \\
&\lesssim
E_0\int_0^{t/2} (t-s)^{-\frac12(1-1/p)-1}(1+s)^{-\frac12} ds
+E_0\int_0^t e^{-\theta\,(t-s)}(1+s)^{-\frac12(1-1/p)-\frac12}  ds\\
&\qquad\qquad+
E_0\int_{t/2}^t (t-s)^{-\frac12}(1+t-s)^{-\frac12}
(1+s)^{-\frac12(1-1/p)-\frac12}  ds\\
&\lesssim E_0(1+t)^{-\frac12(1-1/p)-\frac12}\log(2+t),
\end{aligned}
$$
and other terms either similarly or similarly as in the previous case, we obtain the result.
\end{proof}

\section{Generalizations}\label{s:gen}
We conclude in this appendix by describing briefly 
extensions to more general types of equations arising
in applications, and the modifications in our arguments
that are needed to accomplish this, discussing also, when
possible, the verification of (H1)--(H3) and (D1)--(D3) in 
specific cases.

\subsection{Extensions in type: quasilinear and partially parabolic
systems}\label{s:type}
Our analysis carries over in straightforward fashion
to divergence-form systems of general quasilinear $2r$-parabolic type.
For example, the spectral preparation results of Lemma \ref{jordanlem}, Proposition \ref{whiteig},
and Proposition \ref{linspecg} all go through essentially as
written, depending on no special structure other than divergence form.
From these low-frequency/Bloch number descriptions,
we obtain the same linear bounds on the critical modes $s^{\rm p}$
as described here in the $2$-parabolic semilinear case.
The high-frequency and or high Bloch number analysis also go
through unchanged, the former depending again only on the spectral
preparation results and the latter depending only (through
Pr\"uss' Theorem) on high-frequency resolvent bounds following
from (but not requiring) sectoriality of the linearized operator $L$
about the wave. This completes the linear analysis.

Likewise, by Remark \ref{goodchoice}, we obtain the useful
representation \eqref{veq} of the nonlinear perturbation
equations stated in Lemma \ref{lem:canest}, with
sources $\mathcal{Q}$, $\mathcal{R}$, $\mathcal{S}$ of quadratic
order in $v$, $\psi_x$, $\psi_t$, and a finite number of their derivatives,
which was all that was needed for our nonlinear arguments.
To obtain the nonlinear damping estimate of Proposition \ref{damping},
we note that \eqref{vperteq2} becomes
%\be\label{vperteq2}
%(1-\psi_x) v_t-\ks^2v_{xx}=\ks\cs v_x-\psi_t(\bar U_x+v_x)-\ks(f(\bar U+v)-f(\bar U))_x
%+\ks\left(\frac{\ks\psi_x}{1-\psi_x}(\bar U_x+v_x)\right)_x,
%\ee
$$
(1-\psi_x) v_t
+(-1)^r\ks^{2r} \partial_x( B(\tilde U,\dots, \partial_x^{2r-2}\tilde U)\partial_x^{2r-1} v)=
\hbox{\rm lower order terms},
%\ks\cs v_x-\psi_t(\bar U_x+v_x)-\ks(f(\bar U+v)-f(\bar U))_x
$$
$\tilde U=\bar U+v$.
Thus, taking the $L^2(\RM)$ inner product against
$\sum_{j=0}^K \dfrac{(-1)^{j}\d_x^{2j}v}{1-\psi_x}$,
integrating by parts, and rearranging, we obtain
$
\frac{d}{dt} \|v\|_{H^K(\RM)}^2(t)
 \leq -\tilde{\theta} \|\partial_x^{K+r} v(t)\|_{L^2(\RM)}^2 +
\hbox{\rm lower order terms},
$
similarly as in the second-order semilinear case, leading thereby to
\[
\frac{d}{dt}\|v\|_{H^K(\RM)}^2(t) \leq -\theta \|v(t)\|_{H^K(\RM)}^2 +
C\left( \|v(t)\|_{L^2(\RM)}^2+\|(\psi_t, \psi_x)(t)\|_{H^K(\RM)}^2 \right)
\]
and (by Gronwall's inequality) the result.
See the proof of \cite[Proposition 3.4]{BJNRZ3}, for 
full details in the fourth-order semilinear case.

Combining these ingredients, we obtain, modulo an appropriate
increase in the integer $K$ encoding regularity requirements,
stability, as stated in Theorem \ref{oldmain},
and refined stability, as stated in Proposition \ref{stepg},
yielding the first part
\eqref{refinedest} of description of asymptotic behavior
in Theorem \ref{main}.
By Remark \ref{simplermk}, we get also a partial version
of the second part
\eqref{initial_data}--\eqref{last} of Theorem \ref{main},
but describing comparisons not to the Whitham system,
but only to a second-order hyperbolic-parabolic system
agreeing with the Whitham system in its
linearization about the constant state $(\Ms,\ks)$.
This in turn yields the conclusions of \eqref{linimp},
Corollary \ref{phasethm}, regarding decay with
respect to localized perturbations for linearly phase-decoupled
systems.

Finally, to recover the full result 
\eqref{initial_data}--\eqref{last} of Theorem \ref{main},
comparing to the exact Whitham system, and
thus the sharpened decay rate
\eqref{quadimp} for localized data in the quadratically 
decoupled case,
we have only to observe that performing the same computations
as in Appendix \ref{algebraic} (differentiating the traveling-wave
ODE), 
and in the proof of 
Lemma \ref{sharplemmag} (pulling out quadratic order parts of
nonlinear term $\mathcal{N}$)
while carrying along the additional higher-order terms arising in the
general case, we obtain a higher-order analog of 
Lemma \ref{calcsg}, expressing the resulting quadratic coupling
constants (means) in terms of derivatives of 
first-order terms arising in the Whitham system, 
after which computations go as before to yield the result;
see the proof of Theorem \ref{main}, Section \ref{s:noncom}.

This completes the treatment of the quasilinear $2r$-parabolic case.
Reviewing the above discussion, 
%CHANGED-MR: partial parabolicity changes the form !
but omitting algebraic considerations on which we focus in the next section,
we find that the two ingredients needed to treat more general
divergence-form systems are the nonlinear damping estimate
used to control higher-derivative by lower-derivative norms,
and the high-frequency linearized resolvent bounds
used to apply Pr\"uss' Theorem.
For, these were the only two places where we used the parabolic
form of the equations; the rest of the argument was completely general,
Moreover, the second, linearized, estimate can typically be obtained
by a linearized version of the same energy estimate that is used
to obtain the first, damping-type estimate.
This allows us, in particular, to treat (partially parabolic)
symmetric hyperbolic--parabolic equations such as arise in continuum
mechanics, using ``Kawashima-type'' energy estimates as
described in \cite{Ka}, and variants thereof. 
See, for example, 
%CHANGED-MR: light changes
%the proofs of 
Proposition 4.4 (proved in Appendix A) and Lemma B.1 
%(Appendix B).
in \cite{JZN}.

\br\label{damprmk}
\textup{
The strategy of using a common energy estimate to get,
simultaneously, 
damping high-frequency resolvent, and high-frequency decay estimates,
with derivative gains in the first compensating for derivative losses
in the third, originates in the study of viscous shock stability;
see \cite[Section 4.2.1]{Z6}.
%(CIME).
For simpler, and somewhat sharpened, versions in this context,
see \cite{KZ,NZ}.
}
\er

\subsection{Extensions in form: an abstract continuum of models}\label{s:cont}
Still more generally, we may treat the full class of systems
\be\label{fullclass}
u_t+f(u)_x=g(u)+(B^1(u)u_x)_x
+(B^2(u,u_x)u_{xx})_x + \dots \, ,
\ee
$u,f, g\in \RM^n$, $B_j\in \RM^{n\times n}$,
with $u=\bp u_1\\u_2\ep$, $f=\bp f_1\\f_2\ep$,
$g=\bp 0\\g_2\ep$, $B^j=\bp B^j_{11}& B^j_{12}\\B^j_{21}& B^j_{22}\ep$,
$u_2\in \RM^r$,
including both divergence- and nondivergence-type equations.
Note that this includes both reaction diffusion and conservation
law cases as limits $f\equiv 0$ and $g\equiv 0$,
but also many cases in between:
for example, the viscous relaxation case $n=2$, $r=1$
occurring for the Saint-Venant equations \eqref{sveq},
or the case $n=3$, $r=1$ occurring for the B\'enard--Marangoni
model \eqref{BMmodel} below.

For such models, integrating the conservative $u_1$ equation
in the traveling-wave ODE, and writing as an $N\times N$ first-order system,
we obtain from the requirement of periodicity $N$ constraints, while
we have $N+n-r+2$ degrees of freedom consisting of the initial
condition $u(0)$, the wave number $k\in \RM$, the speed $c\in \RM$, 
and the constant of integration
$q_1\in \RM^{n-r}$ arising from integration of the $u_1$ equation;
thus, we expect generically a manifold of periodic solutions of
dimension $n-r+2$.
In the reaction-diffusion case $r=n$, this returns the familiar value $2$,
or, up to translation, a one-dimensional family (generically) 
indexed by wave number $k$.
In the conservation law case $r=0$, it returns the value $n+2$, leading,
up to translation, to an $(n+1)$-dimensional family 
as in hypothesis (H2) of the introduction.

Substituting this value $n-r+1$ in hypotheses (H2) and (D3), therefore,
we readily obtain by the same derivation as for \eqref{whitham}
a modified Whitham system consisting of the
$(n-r+1)\times (n-r+1)$ system of viscous conservation laws
\ba\label{whithamcont}
\cM_t+\ks(F-\cs \cM)_x &=\ks^2(d_{11} \cM_x + d_{12}\kappa_x)_x,\\
\kappa_t+\ks(-\omega-\cs\kappa)_x&=\ks^2(d_{21}\cM_x
+ d_{22}\kappa_x)_x,\\
\ea
where $\omega(\cM,\kappa)=-\kappa c(\cM,\kappa)$ denotes time frequency,
$\cM:= \int_0^1 U_1^{\cM,\kappa}(x) dx$ and
$$
F(\cM,\kappa):=\int_0^1
\big( f_1(U^{\cM,\kappa}(x)) 
-\sum_j (B^j_{11},B^j_{12})( U^{\cM,\kappa, \dots }(x)) 
(U^{\cM,\kappa})'(x)) 
\big) \, dx
$$
denote mean and mean ``total flux'' in the $u_1$ coordinate,
and $d_{ij}(\cM,\kappa)$ are determined by higher-order corrections.
(Note that, for $B^j\equiv \const$, the terms involving $B^j$ are 
perfect derivatives, so disappear; this explains the fact that they
were not present in the discussion of the second-order semilinear case.)

Likewise, we obtain in straightforward fashion analogs of 
the spectral preparation results of Lemma \ref{jordanlem}, 
Proposition \ref{whiteig}, and Proposition \ref{linspecg},
thus yielding 
corresponding linear bounds on critical modes $s^{\rm p}$.
%CHANGED-MR
Note that the slow decay rates that may arise at the linear level from a possible 
Jordan block will still be compensated by the special structure of the nonlinear terms, coming now
in the form
$$
\cN\ =\ \d_t\cN_0+\d_x\cN_1\ +\ \bp0_{(n-r)\times(n-r)}&0_{(n-r)\times r}\\0_{r\times(n-r)}&\Id_{r\times r}\\\ep\cN_2.
$$
See for example \cite{NR1} for a careful derivation of the 
second-order derivative
Whitham system up to linear and quadratic order in first-order derivative
terms, and \cite{JZN} for a proof of the needed
spectral preparation results 
in the Saint-Venant case \eqref{sveq}.
%, along with linearized and nonlinear stability, 
%
Indeed, so long as the nonlinear structure of the equations permits
a nonlinear damping estimate
as in Proposition \ref{damping},
and high-frequency linearized resolvent estimates as needed to
apply Pr\"uss' Theorem in estimating high-frequency linearized
behavior as in 
the proof of \eqref{undiffeqg} and \eqref{finalgg}) above,
we obtain again 
(modulo increase in the exponent of regularity $K$) 
the stability results of Theorem \ref{oldmain} and Proposition \ref{stepg},
and a partial version of Theorem \ref{main}
describing comparisons to a second-order hyperbolic-parabolic system
agreeing with \eqref{whithamcont} in its
linearization about the constant state $(\Ms,\ks).$,
yielding again the result \eqref{linimp},
of Corollary \ref{phasethm} asserting decay with
respect to localized perturbations for linearly phase-decoupled
systems.

{\it That is,
we obtain in this case
exactly the conclusions cited in the examples of the introduction,
obtained by examination of the 
linearization of the first-order part of the Whitham equations.}

To recover the full result \eqref{initial_data}--\eqref{last} 
of Theorem \ref{main} showing convergence to the exact Whitham system, 
one also needs an analog of Lemma \ref{calcsg}. But, the only difference between the 
(formal) computations of the derivation in Subsection~\ref{secondway} 
and the ones of Lemma \ref{calcsg} is that the former are carried out 
before the implicit change of variables, while the latter are carried out after. 
Thus, analogs of Lemma \ref{calcsg} essentially follow by
commutation of an implicit change of variables and expansions to a desired order.

\subsection{Verification of (H1)--(H3), (D1)--(D3)}\label{s:Dver}
Regarding verification of our stability hypotheses, 
we recall that, assuming the trivial regularity hypothesis (H1), 
hypothesis (H2) is implied by
(D1)--(D3), by Lemma \ref{DHlem}, 
while (H3) by Proposition \ref{whiteig}
can generally be verified by the same spectral expansion
process needed to verify (D2).
Meanwhile, (D1)--(D3) can be verified numerically
by Galerkin approximation or numerical Evans function computation
(see, e.g., \cite{FST,BJNRZ3,BJNRZ4}), or,
in some cases, analytically, using bifurcation theory 
(see, e.g., \cite{S2}) or
singular perturbations (see, e.g., \cite{JNRZ3}).
For general discussion, see \cite{BJNRZ1,JZN,BJNRZ4}.

\subsection{Applications revisited}\label{s:apps}
We now discuss previous examples and some new ones in a bit more depth.

\subsubsection{The Korteweg-de Vries/Kuramoto-Sivashinsky equation}
A more canonical form of \eqref{kseq} is
%CHANGED-MR
%\be\label{e:KS}
$$
u_t+\gamma \partial_x^4u+\eps \partial_x^3 u +
\delta \partial_x^2u+\partial_x f(u) =0,
\quad
\gamma,\delta >0,
$$
%\ee
modeling phenomena from plasma and flame-front instabilities to inclined
thin-film flow \cite{KT,S,PSU,NR1}.
As a fourth-order parabolic equation, this fits the framework of
Section \ref{s:type}, so that all of the results of this paper apply.
Spectral stability has been studied in detail in \cite{CDK,BJNRZ3},
indicating the existence of both spectrally stable and unstable waves;
in particular, (H1)--(H3) and (D1)--(D3) have been shown in \cite{BJNRZ3}
to hold for a wide variety of waves.
We note that stability under these hypotheses has been proven
for localized perturbations in \cite{BJNRZ3}; the new observations
here are asymptotic behavior, and decay for nonlocalized
perturbations.

\subsubsection{The Saint-Venant equations}
Recall, in Lagrangian coordinates, the Saint-Venant equations 
\ba \label{e:stvenant}
\tau_t - u_x&= 0,\\
u_t+ ((2F)^{-1}\tau^{-2})_x&=
1- \tau u^2 +\nu (\tau^{-2}u_x)_x ,
\ea
where $\tau:=h^{-1}$, $h$ is fluid height,
$u$ is fluid velocity, and $x$ is a Lagrangian marker.
These are not parabolic, yet nonlinear damping and high-frequency
resolvent estimates can still be carried out, yielding by the
discussion of Section \ref{s:cont} all of the results of this paper.
Specifically, nonlinear damping is established in
\cite[Proposition 4.4]{JZN} (proved in Appendix A of the reference),
under the ``slope condition'' $\nu \bar u_x<F^{-1}$, 
where $\bar U=(\bar \tau, \bar u)$, a technical condition that
appears to hold in most cases of interest, but which we expect
can be dropped.\footnote{
Linearized analysis suggests that 
the sharp condition is, rather, 
some averaged version of this one, which holds trivially by the fact that
perfect derivatives have zero mean \cite{BJNRZ4}.}
Again, the new observation here is asymptotic behavior, and also stability
under nonlocalized perturbations, stability under localized data having
been established in \cite{JZN}.
As noted earlier, \eqref{e:stvenant} is a balance law rather than
a conservation law, with nonconservative source term $g$.

\subsubsection{The capillary Saint-Venant equations}
With capillary pressure effects, \eqref{e:stvenant} becomes
\ba 
%CHANGED-MR
%\label{e:capstvenant}
\tau_t - u_x&= 0,\\
u_t+ ((2F)^{-1}\tau^{-2})_x&=
1- \tau u^2 +\nu (\tau^{-2}u_x)_x 
%CHANGED-MR: from Eulerian form \sigma h h_{xxx}
%- (c(\tau_x)\tau_xx)_x,
-\sigma(\tau^{-5}\tau_{xx}-\frac{5}{2}\ \tau^{-6}(\tau_x)^2)_x
%
%CHANGED-MR
\nonumber
\ea
where 
%CHANGED-MR
%$c>0$ 
$\sigma>0$
is the coefficient of capillarity.
These equations can be reduced by Kotschote's \cite{Ko}
method of auxiliary variables
(introducing $z:=\tau_x$) to a $3\times 3$ second-order quasilinear
parabolic system
\ba 
%CHANGED-MR
%\label{e:3capstvenant}
\tau_t - u_x+z_x&= \tau_{xx},\\
z_t &= u_{xx},\\
u_t+ ((2F)^{-1}\tau^{-2})_x&=
1- \tau u^2 +\nu (\tau^{-2}u_x)_x 
%CHANGED-MR
%-(c(z)z_x)_x,
-\sigma(\tau^{-5}z_x-\frac{5}{2}\ \tau^{-6}z^2)_x
%
%CHANGED-MR
\nonumber
\ea
to which standard techniques can be applied \cite{Y,Z4}.
This fits the framework of Section \ref{s:type},
yielding all of the results of this paper, the only change being in the regularity assumptions on data,
which must be incremented by one to accommodate the new variable $z=\tau_x$.
Existence and spectral stability or instability 
of these waves is a topic of ongoing investigation \cite{BJNR}.

\subsubsection{B\'enard--Marangoni flow}\label{s:BM}
A qualitative model introduced in 
%CHANGED-MR
%\cite{Zi}
\cite{HSZ} 
for B\'enard--Marangoni flow,
or flow driven by temperature-induced surface tension variation,
is
\ba\label{BMmodel}
u_t &= -(1+u_{xx})_{xx} +
\eps^2 u + f(u,v_x, w_x),\\
v_t&= v_{xx}+v_x + g_1(u,v,w)_x,\\
w_t&= w_{xx} -w_x + g_2(u,v,w)_x,\\
\ea
with
$ f(u,v_x,w_x)=
-u^3 +\gamma (uv_x+ uw_x)$,
$g_1(u,v,w)=-uv$,
$g_2(u,v,w)=-uw. $
Though of mixed fourth-order parabolic/second-order parabolic form,
it is readily seen that these equations are both sectorial and admit
a nonlinear damping estimate; moreover, they are of the mixed
conservative/nonconservative form \eqref{fullclass}.
Thus, by the discussion of Sections \ref{s:type} and \ref{s:cont},
the main results of this paper apply, giving stability and behavior
in terms of a $3\times 3$ hyperbolic--parabolic system agreeing with the
Whitham system \eqref{whithamcont}.
%CHANGED-MR
% in its linearization about $(\Ms,\ks)$.

Let us now discuss existence, the form of the Whitham equations, 
and validation of (D1)--(D3).
Setting $v\equiv w\equiv 0$, we find that the equations reduce
to the Swift--Hohenberg equation \eqref{sheq} for $u$, with
bifurcation parameter $r=\eps^2$ restricted
to the positive side of the bifurcation point $r=0$
at which periodic solutions appear.
Thus, we inherit from the Swift--Hohenberg equations a special
class of periodic solutions with $(v,w)$ vanishing.
Up to translation, such solutions are given by the
$2$-parameter family of zero-speed
$\frac{ 2\pi}{1+\eps\omega}$-periodic Swift--Hohenberg solutions
\be\label{SHwaves}
\bar U^{\omega,\eps}(x)= 
\frac{2 \eps(\sqrt{1-4\omega^2}}{\sqrt{3}}\cos((1+\eps\omega)x)
+\cO(\eps^2),
\quad
(\bar V^{\omega,\eps},\bar W^{\omega,\eps})(x)\equiv (0,0),
\ee
where $\eps$, recall, is the bifurcation parameter, a fixed constant
in \eqref{BMmodel}.
However, there are many other solutions for which $(\bar V,\bar W)
\not \equiv (0,0)$, yielding an additional two parameters in the
description of nearby periodic traveling waves.
Moreover, though the Swift--Hohenberg solutions are zero speed
as a result of reflection symmetry (see Remark \ref{reflectrmk}),
reflection symmetry of \eqref{BMmodel} is broken as soon as
$(v,w)\not \equiv (0,0)$, and so in general these waves may have
arbitrary speed.
It is our expectation, therefore, that the Whitham system 
is not phase-decoupled even about such special waves.

Numerical experiment by Galerkin approximation in 
\cite{Zi} indicate that solutions \eqref{SHwaves} satisfy 
stability conditions (D1)--(D3) for $\omega=0$ and $\eps>0$ 
in a moderate range. 
%though resolution not completely clear and discussion quite confused
%...- don't mention that here though..
Here, we demonstrate the same conclusion for $|\omega|< 1/2\sqrt{3}$
and  $\eps<<1$,
using decoupling of the equations and known analytical results for 
the Swift--Hohenberg equation,
at the same time obtaining the limiting $\eps\to 0$ coefficients of 
the linearized Whitham system about $(\Ms,\ks)=(0,\ks)$.
It would be interesting to carry out a systematic numerical stability
investigation as in \cite{BJNRZ2,BJNRZ3} on the entire parameter range, 
and in particular to determine phase-coupling or decoupling of
the associated Whitham system. 

\begin{proof}[Proof of (D1)--(D3)]
About the special solutions \eqref{SHwaves}, 
the linearized eigenvalue equations are
\ba\label{BMeig}
\lambda u&= L^0u + Mv + Nw,\\
\lambda v&= L^+ v,\\
\lambda w&= L^-w,\\
\ea
where $L_0$ is the linearized operator of the Swift--Hohenberg equation
about $\bar U$ and
%CHANGED-MR
%\be\label{Lpm}
$$
L^\pm:= \partial_x^2 \pm \partial_x -\partial_x \bar U^{\omega,\eps}.
$$
%\ee 

By upper triangular form of \eqref{BMeig}, the eigenvalues
of $L_\xi$, counted by algebraic multiplicity,
consist of the union of the eigenvalues of $L^0_\xi$ and $L^\pm_\xi$.
Let us first consider the eigenvalues of the Swift--Hohenberg
operator $L^0_\xi$. 
In
\cite{Eck,CE1,CE2} (see also \cite{M,S2}) 
it was analytically verified\footnote{It has also been shown numerically that there exist 
bands of stable periodic Swift--Hohenberg solutions 
in the parameter space $(\omega, \kappa,\eps)$ \cite{M,BJNRZ2},
for $|\eps|$ not necessarily small.} 
that for $\eps<<1$,
solutions $\bar U^{\omega,\eps}$ in \eqref{SHwaves}
are spectrally stable for
%CHANGED-MR
%\begin{equation}\label{eckbdry}
$$
\left|4\omega^2\right|<\frac{1}{3}+\mathcal{O}(\eps)
$$
%\end{equation}
%
(in particular, for $\omega=0$).
From the fact that the waves are of speed $c\equiv 0$, we find that
the characteristic speed of the associated scalar Whitham equation is
$a^0\equiv 0$, and the associated critical mode has expansion
$\lambda^0(\xi)= -d^0 \xi^2$.
Turning to the operators $L^\pm$, and noting that 
$\bar U^{\omega,\eps}\to 0$ uniformly in all derivatives as $\eps\to 0$,
we find that as $\eps\to 0$ their eigenvalues approach uniformly the
eigenvalues of the limiting constant-coefficient operators
$$
\bar L^\pm_\xi:= (\partial_x+i\xi)^2 \pm (\partial_x+i\xi), 
$$
which, by direct (discrete Fourier transform) computation, are
$$
\bar \lambda^\pm (\xi)= -(j+\xi)^2 \pm i(j+\xi),
$$
where the Fourier frequency $j$ runs through the integers.
By continuity, these are therefore spectrally stable for $|\eps|<<1$,
with approximate critical mode expansions (obtained at $j=0$)
of $\pm i \xi - \xi^2$. Combining these facts, we
find that the limiting linearized Whitham system has characteristic
speeds $a_j=0,\pm 1$, with corresponding (diagonal) viscosity coefficients
$d$, $1$, $1$.
This verifies (D1)--(D3) and (by distinctness of $a_j$) (H3) for
$|\eps|$ sufficiently small, yielding spectral stability by the
discussion of Section \ref{s:Dver}
\end{proof}

\subsubsection{Inclined Marangoni flow}\label{s:iM}
The related inclined thin-film equation
\be\label{iMeq}
H_t + (H^2-H^3)_x=-(H^3H_{xxx})_x
\ee
models Marangoni flow driven by a thermal gradient up an inclined silicon 
wafer, where $H$ denotes fluid height \cite{BMFC,BMSZ}.
As a cousin of the Kuramoto--Sivashinsky equation, it would be
interesting to investigate whether this model too supports stable
periodic traveling-waves solutions.

\subsubsection{Surfactant-driven Marangoni flow}\label{s:SM}
Finally, we mention the surfactant model \cite{MT}
\ba\label{surf}
H_t + \frac12(H^2 \sigma'(\Gamma) \Gamma_x)_x&=0,\\
\partial_t \Gamma + \partial_x(\Gamma H\sigma'(\Gamma) \partial_x \Gamma)&=
{\rm Pe}_{\rm s}^{-1} \partial_x^4 \Gamma,
\ea
modeling flow in a thin horizontal film driven by surfactant induced
gradients in surface tension,
where $H$ is fluid height and $\Gamma$ surface surfactant concentration,
and $ {\rm Pe}_{\rm s}$ is the modified Peclet number, a dimensionless
constant, and $\sigma(\Gamma)=1-\Gamma$ is an equation of state
encoding the dependence of surface tension on surfactant density.
Like \eqref{iMeq}, this appears to be an interesting example for 
study by the methods developed here and in \cite{JZ2,JZ3,JZN,BJNRZ3}.
Note, as the second equation is conservative, 
that the associated Whitham approximation is indeed of system form.

\end{document}